\documentclass{amsart}

\usepackage[utf8]{inputenc}
\usepackage[T1]{fontenc}

\usepackage[ruled,titlenotnumbered,vlined]{algorithm2e}
\SetKwInput{Input}{Input}
\SetKwInput{Output}{Output}
\IncMargin{0.5em}
\SetNlSkip{0.5em}

\usepackage{tikz}
\usepackage{forest}
\usetikzlibrary{arrows.meta, positioning, quotes}

\usepackage{amssymb}
\usepackage{color}
\usepackage{enumerate}
\usepackage[normalem]{ulem}

\usepackage{fancyvrb}
\fvset{fontshape=normal}
\fvset{listparameters={\setlength{\topsep}{0pt}\setlength{\partopsep}{0pt}}} 

\usepackage{enumitem}

\usepackage{array,tabularx}
\newcommand{\cellscript}[1]{%
  \parbox[c]{\linewidth}{%
    \centering\everymath{\scriptstyle}
    $\begin{array}{@{}l@{}}
        #1
      \end{array}
    $%
  }%
}
\newcolumntype{P}[1]{>{\centering\arraybackslash}p{#1}}
\newcolumntype{Y}{>{\centering\arraybackslash}X}

\newcommand{\K}{\mathbb{K}}
\newcommand{\Ka}{\mathbb{K}^\mathbb N}

\newcommand{\Q}{\mathbb{Q}}
\newcommand{\N}{\mathbb{N}}
\newcommand{\rel}{\mathbb{Z}}
\newcommand{\nat}{\mathbb{N}}
\newcommand{\ree}{\mathbb{R}}

\newcommand{\valSn}{{\operatorname{val}_{S_n}}}

\newcommand{\Horl}{\operatorname{Hor_{left}}}
\newcommand{\Horr}{\operatorname{Hor_{right}}}
\newcommand{\HorrI}{\operatorname{Hor_{right}}}
\newcommand{\HorrN}{\operatorname{Hor_{right}}}
\newcommand{\HorrZ}{\operatorname{Hor_{right}}}
\newcommand{\Hortr}{\operatorname{Hor_{\theta,right}}}

\newcommand{\cDiffden}{{\mathcal{T}}_{\text{den}}} 
\newcommand{\Xnum}{{X^\psi_{\text{num}}}}
\newcommand{\Thnum}{{\Theta}_{\text{num}}}
\newcommand{\Xden}{{X^\psi_{\text{den}}}}
\newcommand{\Thden}{{\Theta}_{\text{den}}}
\newcommand{\Anum}{{A}_{\text{num}}} 
\newcommand{\Aden}{{A}_{\text{den}}}

\newcommand{\Dnum}{{D^\psi_{\text{num}}}} 
\newcommand{\Dden}{{D^\psi_{\text{den}}}}
\newcommand{\Ddeni}{{D^\psi_{\text{den},i}}}
\newcommand{\Ddenr}{{D^\psi_{\text{den},r}}}
\newcommand{\Ddenrmu}{{D^\psi_{\text{den},r-1}}}
\newcommand{\Ddenrt}{{D^\psi_{\text{den},r_2}}}
\newcommand{\Ddenrmi}{{D^\psi_{\text{den},r-i}}}
\newcommand{\Ddenrmipu}{{D^\psi_{\text{den},r-i+1}}}
\newcommand{\Ddenrmk}{{D^\psi_{\text{den},r-k}}}

\newcommand{\Ddenrmkj}{{D^\psi_{\text{den},r-k_j}}}

\newcommand{\Ipsi}{I_{\psi}}

\newcommand{\Jnum}{{J}_{\text{num},\varepsilon}}
\newcommand{\Jden}{{J}_{\text{den},\varepsilon}} 
\newcommand{\Hnum}{{H}_{\text{num}}} 
\newcommand{\Hden}{{H}_{\text{den}}}

\newcommand{\Hlnum}{{H}_{\ell,\text{num}}} 
\newcommand{\Hlden}{{H}_{\ell,\text{den}}}

\newcommand{\Pnum}{{P}_{\text{num}}} 
\newcommand{\Pden}{{P}_{\text{den}}}

\newcommand{\Qnumi}{{Q}_{\text{num,$i$}}}

\newcommand{\Qdeni}{{Q}_{\text{den,$i$}}}

\def\assocal{\langle A_{\text{al}}(L)\rangle_{\RecI}^{\psi_n}}
\def\assocan{\langle A_{\text{an}}(L)\rangle_{\RecI}^{\psi_n}}

\newcommand{\equivZA}{\equiv_{\mathcal Z_\N}}

\def\Rec{{{\mathbb R}\textnormal{ec}}}
\def\RecI{{{\mathbb R}\textnormal{ec}}}
\def\RecN{{{\mathbb R}\textnormal{ec}}}
\def\RecZ{{{\mathbb R}\textnormal{ec}}}
\def\Frec{{{\mathbb P}\textnormal{rec}}}
\def\FrecI{{{\mathbb P}\textnormal{rec}}}

\def\FrecZ{{{\mathbb P}\textnormal{rec}}}
\def\Sh{{{\mathbb S}\textnormal{h}}}

\def\Ann{\operatorname{Ann}}
\def\Za{{\mathcal Z_{\mathbb N}}}
\def\Zn{{\mathcal Z_{\mathbb N}}}

\newcommand{\lclm}{\operatorname{lclm}} 
\newcommand{\mclm}{\operatorname{mclm}} 
\newcommand{\lcrm}{\operatorname{lcrm}} 
\newcommand{\gcld}{\operatorname{gcld}} 
\newcommand{\gcrd}{\operatorname{gcrd}} 

\newcommand\incircrel[1]{\mathrel{%
    \ooalign{\small{$#1$}\cr$\bigcirc$}%
}}
\newcommand\circledless{\incircrel{\kern.2ex<}}
\newcommand\circledgtr {\incircrel{\kern.45ex>}}
\newcommand\circledleq {\incircrel{\kern.15ex\leq}}
\newcommand\circledgeq {\incircrel{\kern.6ex\geq}}

\newcommand{\divides}{\mathrel|}

\newcommand{\ppcm}[2]{{\left(#1\right)}^{#2}}

\renewcommand{\le}{\leqslant}
\renewcommand{\ge}{\geqslant}
\renewcommand{\leq}{\leqslant}
\renewcommand{\geq}{\geqslant}
\renewcommand{\phi}{\varphi}

\numberwithin{equation}{section}

\newtheorem{theorem}{Theorem}[section]
\newtheorem{lemma}[theorem]{Lemma}
\newtheorem{corollary}[theorem]{Corollary}
\newtheorem{proposition}[theorem]{Proposition}

\theoremstyle{definition}
\newtheorem{definition}[theorem]{Definition}
\newtheorem{example}[theorem]{Example}

\theoremstyle{remark}
\newtheorem{remark}[theorem]{Remark}

\usepackage{thmtools,thm-restate}

\usepackage{zref-clever}
\zcsetup{cap} 
\zcsetup{abbrev=true} 
\zcsetup{nameinlink=false}
\AddToHook{env/lemma/begin}{\zcsetup{countertype={theorem=lemma}}}
\AddToHook{env/corollary/begin}{\zcsetup{countertype={theorem=corollary}}}
\AddToHook{env/proposition/begin}{\zcsetup{countertype={theorem=proposition}}}
\AddToHook{env/problem/begin}{\zcsetup{countertype={theorem=problem}}}
\AddToHook{env/definition/begin}{\zcsetup{countertype={theorem=definition}}}
\AddToHook{env/example/begin}{\zcsetup{countertype={theorem=example}}}
\AddToHook{env/notation/begin}{\zcsetup{countertype={theorem=notation}}}
\AddToHook{env/remark/begin}{\zcsetup{countertype={theorem=remark}}}
\AddToHook{cmd/appendix/before}{\zcsetup{countertype={section=appendix}}}
\zcRefTypeSetup{theorem}{Name-sg-ab=Thm., Name-pl-ab=Thm.}
\zcRefTypeSetup{proposition}{Name-sg-ab=Prop., Name-pl-ab=Prop.}
\zcRefTypeSetup{definition}{Name-sg-ab=Def., Name-pl-ab=Def.}
\zcRefTypeSetup{corollary}{Name-sg-ab=Cor., Name-pl-ab=Cor.}
\zcRefTypeSetup{algocf}{Name-sg=Algorithm, Name-pl=Algorithms,
Name-sg-ab=Algo., Name-pl-ab=Algo.}
\zcRefTypeSetup{hypothesis}{Name-sg-ab=Hyp., Name-pl-ab=Hyp.,
refbounds= {, (,),}}

\usepackage{hyperref}
\hypersetup{colorlinks=true,citecolor=blue,linkcolor=blue,
    filecolor=blue,urlcolor=blue}

\begin{document}

\title[Recurrence Operators for Generalized Fourier Series]{Fractions of Recurrence Operators for Generalized Fourier Series in
Classical Orthogonal Polynomials
}

\author[A. Benoit]{Alexandre Benoit}
\address{Éducation nationale, France}
\curraddr{}
\email{alexandrebenoit@yahoo.fr}
\thanks{}

\author[N. Brisebarre]{Nicolas Brisebarre}
\address{Université de Lyon, CNRS, ENS de Lyon, Inria, Université
    Claude-Bernard Lyon~1, Laboratoire LIP (UMR 5668), Lyon, France.}
\curraddr{}
\email{Nicolas.Brisebarre@ens-lyon.fr}
\author[B. Salvy]{Bruno Salvy}
\address{Université de Lyon, CNRS, ENS de Lyon, Inria, Université
    Claude-Bernard Lyon~1, Laboratoire LIP (UMR 5668), Lyon, France.}
\curraddr{}
\email{Bruno.Salvy@ens-lyon.fr}

\begin{abstract}
  We consider series expansions in bases of classical orthogonal polynomials. 
  When such a series solves a linear differential equation with
 polynomial coefficients, its coefficients satisfy a linear recurrence
 equation. We
 interpret this equation as the numerator of a fraction of linear
 recurrence operators. This interpretation lets us give a simple and
 unified view of previous algorithms computing these recurrences,
 with a noncommutative Euclidean algorithm as the algorithmic engine. 
  Finally, we demonstrate the effectiveness of our approach on various examples.
\end{abstract}



\maketitle


\section{Introduction}

Let  $( \psi_n(x))_{n}$ be a family of 
 orthogonal polynomials
 associated to a weight function~$w(x)$ over an interval $
 (a,b)\subset\mathbb R$
 and let $f(x)$ be a sufficiently regular function defined on~$(a,b)$ 
 (see \zcref{sec:orthopoly} for background and precise definitions.)
 Classically, the square norm of the
 polynomials is denoted
 \begin{equation}\label{eq:hn}
h_n=\int_a^b{\psi_n(x)^2w(x)\,dx}
 \end{equation}
 and one considers the sequence of integrals
 \begin{equation}\label{eq:integrals}
a_n=\frac1{h_n}\int_a^b{f(x)\psi_n(x)w(x)\,dx},\quad n\in \N.
\end{equation}
 When these integrals converge, they are called the \emph{generalized
 Fourier coefficients} of~$f(x)$ with respect to the family~$(\psi_n
 (x))$
 and the sum
 \[\sum_{n\ge0}{a_n\psi_n(x)}\]
 is called a \emph{generalized Fourier series}. There is a large
 literature
 on conditions under which this expansion actually converges
 to~$f(x)$ in some sense and the speed of this convergence~\cite
 [\S10.19-10.20]
 {Erdelyi1981V2},\cite
 [\S8.3]{Luke1969},\cite[Chap.~2,\S 8.2]
{NikiforovUvarov1988},\cite[Chap.~9]{Szego1975}. These expansions find
 applications in spectral methods~%
 \cite{Boyd2000,GottliebOrszag1977,Wimp1984,OlverTownsend2013,BenoitJoldesMezzarobba2017}
 and
 in the computation of connections or linearization coefficients
 between
 families of orthogonal polynomials (where convergence is not an
 issue)~\cite{Thacher1964,Askey1975,RonveauxZarzoGodoy1995,
GodoyRonveauxZarzoArea1997,Szwarc1992}.
\nocite{GovaertsHounkonnouMsezane2002}

We focus here on a special but frequent situation:
\begin{itemize}
    \item[--] $f (x)$ satisfies a linear differential equation with
    polynomial coefficients;
    \item[--] $(\psi_n(x))$ is a family of 
\emph{classical} orthogonal polynomials.
\end{itemize}
In that case, when the integrals \eqref{eq:integrals} converge,
they satisfy a linear recurrence with polynomial
coefficients, whose computation is the aim of this work. Indeed, such a recurrence can lead to a faster
computation of the series and in some cases, can also lead to a
closed form solution, in combination with computer algebra algorithms
such as Petkovšek's algorithm~\cite{Petkovsek1992}.
This
recurrence is also an entry
point for the automatic proof of numerous identities related to these
 coefficients~%
\cite{FieldsWimp1961,Luke1959,LukeColeman1961,Verma1966,WimpLuke1962}.
(See also \zcref{sec:examples} for examples.)

\subsection{Previous work} \label{sec:previous}

It seems that Clenshaw~\cite{Clenshaw1957} was the first to use
recurrences satisfied by the coefficients of Chebyshev expansions 
(expansions in the basis of Chebyshev polynomials). He did not give an
algorithm to compute recurrences but used them in an implicit way to
compute numerically the first coefficients of a Chebyshev series
solution to a linear differential equation. His method was
extended by Elliott~\cite{Elliott1959} to ultraspherical (or
Gegenbauer) expansions. It leads to the construction of a system of
linear recurrences relating the generalized Fourier coefficients of
the successive derivatives of the function. From there, several
techniques have been developed for a numerically stable use of these
recurrences~\cite{BenoitJoldesMezzarobba2017,OlverTownsend2013}.

It is also possible to obtain a single linear recurrence directly,
by first turning the differential equation into an integral equation
and then extracting the coefficients. This has been
shown by several authors, first for small orders~%
\cite[Chap.~5]{FoxParker1968}, \cite[\S8.7]{Luke1969},
then in more generality and with algorithms 
in the context of early symbolic
computation
by Geddes~\cite{Geddes1977}. 

Another approach working directly at the level of differential
equations was developed by Paszkowski in the case of Chebyshev
polynomials~\cite{Paszkowski1975}.
Lewanowicz extended this method to the case of Gegenbauer polynomials~\cite{Lewanowicz1976} and gave a new algorithm computing a
smaller order recurrence in some cases. 
Later, Lewanowicz extended his work to the Jacobi polynomial case~%
\cite{Lewanowicz1983,Lewanowicz1986,Lewanowicz1991,Lewanowicz1992} and
to the case of classical orthogonal polynomials~%
\cite{Lewanowicz1995,Lewanowicz1996}.
This algorithm has
not been discussed much in the literature (it is described as
``very complicated'' by Wimp~\cite[p. 186]{Wimp1984}.) Our approach
relies on many of Lewanowicz' ideas. We use a more algebraic
framework that lets us reformulate his
algorithm as a sort of Horner evaluation of the linear
differential equation that is hopefully more transparent (see 
\zcref{algo:RightHorner}). In later work, alone and jointly with
Woźny, for the case of semi-classical orthogonal polynomials~%
\cite{Lewanowicz2002,LewanowiczWozny2004}, Lewanowicz noted that other
techniques can be applied to reduce the order of the resulting
recurrence equation in cases where the differential equation has
singularities at~0 or~$\pm 1$
\cite{Lewanowicz1986,Lewanowicz1991,Lewanowicz1992,Lewanowicz2002,LewanowiczWozny2004}. These techniques can also be integrated in our framework, see~\zcref{sec:lewa2}. 

A method using integral operators in another way was designed by
Rebillard~\cite{Rebillard1997,Rebillard1998} for the case of Chebyshev
series expansions, see~\zcref{sec:rebillard}. His method is used later,  independently, by  Olver and Townsend~\cite{OlverTownsend2013}. Similar to Paszkowski's algorithm, it gives rise to a simpler handling of the operators.
Later, Rebillard and Zakrajšek~\cite{RebillardZakrajsek2007}
 addressed the case of expansions in a hypergeometric polynomial basis
 (see~\cite{NikiforovUvarov1988} for a definition.) Other approaches have been developed for the case of connection coefficients between classical (or semi-classical) orthogonal polynomials, by  Area, Godoy, Ronveaux and Zarzo~\cite{RonveauxZarzoGodoy1995,GodoyRonveauxZarzoArea1997}.
 These methods are  discussed in \zcref{sec:previous-algo}.

Also related is the work by Maroni and da Rocha~%
\cite{MaronidaRocha2008,MaronidaRocha2013} who present a
guess-and-prove approach to compute connection coefficients.

Finally, a completely different approach can be based on the recent
progress in
computer algebra around the method of creative telescoping initiated
by Zeilberger~\cite{Zeilberger1991a}. These methods can compute a
linear recurrence starting from the integral representation~%
\eqref{eq:integrals} and a system of operators for the family~$
(\psi_n)$ and the function~$f$. However, this general-purpose method
does not deal with convergence issues and may return an incorrect
recurrence. We report on experiments with recent algorithms
and
implementations~\cite{Koutschan2010,BostanChyzakLairezSalvy2018} in
\zcref{sub:creative_telexcoping}.

\medskip

For a recurrence found by a purely algebraic approach to be
satisfied by the Fourier coefficients, convergence conditions have to
hold in general. Elliott indicates a sufficient condition without
proof~\cite{Elliott1959}.
Lewanowicz works under the assumption that, if
$r$~denotes the order of the linear differential equation satisfied
by the function~$f$ whose coefficients are considered, then $f^{(r)}$
 has a Gegenbauer (or Jacobi) expansion that is uniformly convergent
 in~$[-1,1]$~%
\cite{Lewanowicz1976,Lewanowicz1983,Lewanowicz1986,Lewanowicz1991,Lewanowicz1992}. Rebillard~\cite{Rebillard1997,Rebillard1998} and Olver
 \& Townsend~\cite{OlverTownsend2013} who deal, respectively, with
 Chebyshev and Gegenbauer expansions, use the stronger assumption that
 the function $f$ is holomorphic in the neighborhood of $
 [-1,1]$. These assumptions allow these authors to prove that the
 recurrences they compute are satisfied by the coefficients of
 convergent generalized Fourier expansions. Our approach is slightly
 different: we do not demand convergence of the expansions. Instead,
 we prove
 that the recurrences computed by our algorithms hold provided the
 generalized Fourier coefficients of the function and its first few
 derivatives exist (the integrals~\eqref{eq:integrals} converge),
 which is a much weaker condition.

 This article is an extension and an improvement of~\cite{BenoitSalvy2009}, where we focused on the case of the Chebyshev basis.

\subsection{Contribution}\label{subsec:contribution}
The principle of Lewanowicz' method is to construct recurrences of the
form
\begin{equation}\label{eq:rec-rec}
b_0(n)v_n+\dots+b_\ell(n)v_
{n+\ell}=a_0(n)u_n+\dots+a_k(n)u_{n+k},\quad n\in\mathbb N
\end{equation}
with $a_i$ and~$b_j$ rational functions of~$n$, 
that relate the generalized Fourier coefficients~$u_n$ of a
function~$f$
to the generalized Fourier coefficients~$v_n$ of the
application $L\cdot f$ of a linear differential operator~$L$ to~$f$.
In particular, when~$f$
is a solution of~$L\cdot f=0$, then the coefficients~$v_n$ are all~0
and the right-hand side is the desired recurrence equation.

The starting point is given by relations between the
coefficients of a function and those of its derivative or its product
by the variable. The construction of this recurrence then proceeds by
induction using algorithms computing relations for the coefficients of
sums and of compositions of linear differential operators. 
Our main observation is that these operations can be expressed
as operations on fractions of linear recurrence operators.
Indeed \zcref{eq:rec-rec} can be expressed as $A \cdot u_n + B \cdot 
(-v_n) = 0$,  where~$B$ and~$A$ are the recurrence operators on its
left-hand and right-hand sides. Conceptually, we can rephrase this
relationship with abuse as $v_n=B^{-1}A\cdot u_n$, hence the fraction
term. This has to be used with caution: common left factors of~$A$
and~$B$ cannot be simplified out in general.
The algorithm reduces to sums and products of such fractions,
performed in a simple way using a non-commutative variant of the
Euclidean algorithm for the computation of gcds. These constructions
are closely related to
those used by Ore~\cite{Ore1933} to formalize
fractions of operators.

For the recurrence~\eqref{eq:rec-rec} to hold for
all~$n\in\mathbb N$, the rational function coefficients~$a_i$
and~$b_j$ must not vanish
on~$\mathbb N$. Thus the recurrence operators that we consider are
different from the shift operators that are commonly used in
computer algebra, where all rational
functions are allowed in the coefficients (\zcref{sec:shift-rec}).
This has strong algebraic consequences: the ring of
recurrence operators is not left principal and does
not have least common left multiples in general. However, we show how
to construct common left multiples of minimal order using Euclid's
algorithm. This leads to an incremental construction of
fractions of recurrence
operators $(A,B)$ for which the expression $B\cdot v_n=A\cdot u_n$
from \zcref{eq:rec-rec} holds (\zcref{sec:module-of-pairs}).

Having defined these operations of sum and product on fractions, the
construction of recurrences for Taylor or generalized Fourier
coefficients is a direct translation from the linear differential
operator of interest. To each of the differentiation operator and
the multiplication by the variable~$x$ is associated a fraction,
which depends on the basis of classical orthogonal polynomials (or
$x^n$ for Taylor expansions). Then a fraction is constructed by Horner
evaluation of the differential operator on these fractions, keeping
track of convergence requirements that were often absent in previous
works. Several
constructions of fractions are possible, depending on how one chooses
to write the differential operator (\zcref{sec:rec_coeff}).

A feature of the fractions of recurrence operators that departs
from the classical Ore
fractions is that they are not equivalence classes for a relation
expressing that the numerator of the difference is~0. Indeed, reducing
a fraction by factoring out a common left factor between numerator
and denominator would amount to simplifying both sides of
\zcref{eq:rec-rec}, which is not valid in general. Instead, numerator
and denominator are left untouched. Then it becomes important to
design algorithms that compute \emph{irreducible} fractions, i.e.,
fractions for which the greatest common left divisor between numerator
and denominator does not involve the shift operator. 
This is achieved by studying the fractions obtained from different
representations of the same linear differential operator, by right
Horner and left Horner evaluations, already present with a different
language in the works of Lewanowicz and Paszkowski. Our algebraic 
framework makes it possible to combine two fractions
obtained by these different evaluations into a new one that is always
irreducible.
This is summarized in the following, which is
made more precise in \zcref{sec:irred-fracs}.
\begin{theorem}[Main result]
Let $(\psi_n)$ be a family of {classical} orthogonal polynomials 
and let $L$
be a linear
differential operator with polynomial coefficients%
. Then \zcref [S]
{algo:main} with $L$ as input returns an irreducible
fraction  $(\Hnum,\Hden)$ of linear recurrence operators. If $f$ is a
solution of $L\cdot f=0$
such that
\begin{equation}\tag{H}
\text{the sequences }([\psi_n](f^{(j)}))_{n \in \nat}\text{ exist
for }j\le\deg_\partial L,
\end{equation}
then 
its
sequence of generalized Fourier coefficients is annihilated by the
numerator: \[\Hnum\cdot[\psi_n](f)=0,\qquad n\in \mathbb N.\]
\end{theorem}
Extensions of this result with hypotheses weaker
than~\eqref{eq:hypH} are given in \zcref{sec:singular-case}.

\subsection{Plan of the article} 
\zcref{sec:orthopoly} recalls the required notation on classical orthogonal polynomials. Shift and recurrence operators are described in \zcref{sec:shift-rec}. The former are classical, the
latter are designed to prevent poles at {nonnegative} integers in the
coefficients. Fractions of recurrence operators are then defined in 
\zcref{sec:module-of-pairs}, together with operations on them that
boil down to a noncommutative version of Euclid's algorithm. Their
relation with Ore fractions is clarified. With these fractions, a 
general method of construction of recurrences is given in 
\zcref{sec:rec_coeff}, first for coefficients
of Taylor expansions and then for generalized
Fourier
coefficients.  This general
method allows one to develop different algorithms depending on choices
of
differential terms representing a linear differential operator. 
\zcref{sec:irred-fracs} studies the fractions that one obtains by
right or left Horner evaluation and gives algorithms that always
produce
irreducible fractions. In some
singular cases, it is possible to obtain recurrences that escape the
convergence constraints of the algorithms in that section. This is 
discussed in \zcref{sec:singular-case}. 
\zcref{sec:previous-algo} summarizes the relations between our
approach and previous work. 
A long
\zcref{sec:examples} presents the results of a Maple
implementation,
together with timings. A few further comments are made in \zcref{sec:conclusion}.

\section{Classical orthogonal
polynomials}\label{sec:orthopoly}
We  recall  basic definitions related to  classical orthogonal
polynomials.
We refer to classical treatises for more information~\cite{Chihara1978,Szego1975,NikiforovUvarov1988,KoekoekLeskySwarttouw2010}.

\subsection{Orthogonal polynomials}\label{subsec:cop}

An interval $-\infty \leq a < b \leq +\infty$ is given together with a
measurable function $w : (a,b) \to \ree_+$, called the \emph{weight
function}, that is non-negative and such that $\int_a^bw
(x)\,\mathrm{d}x>0$ and 
$\int_a^b x^n w(x) \, \mathrm{d}x < +\infty$ for all $n \in \nat$.

Equipped with the inner product 
\begin{equation}\label{eq:inner-product}
\langle f|g \rangle_w :=  \int_a^b f(x) g(x) w(x) \, \mathrm{d}x,
\end{equation}
the space $L^2((a,b),w)$ is a Hilbert space.

A family $(\psi_n)_{n\in \nat}$ of polynomials in~$\mathbb R[x]$ such
that $\deg\psi_n=n$ is a \emph{family of orthogonal polynomials} with
respect to~$w$ if they are orthogonal for the inner 
product~\eqref{eq:inner-product}. If $h_n$ denotes the
squared norm $\langle\psi_n|\psi_n\rangle_w$ as in \zcref{eq:hn}, then 
the triple $( (a,b), w, (h_n)_{n \in \nat})$ determines the family~$
(\psi_n)$ in a unique way.

A family $(\psi_n)_{n \in \nat}$ of orthogonal polynomials satisfies a 3-term recurrence~\cite[Chap. 2 \S6.3]{NikiforovUvarov1988}
\begin{equation}\label{eq:3term-rec}
x \psi_{n}(x) = a(n) \psi_{n+1}(x) + b(n) \psi_{n}(x)+c(n) \psi_{n-1}(x), \quad n \geq 0,
\end{equation}
where we set $\psi_{-1} = 0$.

\subsection{Classical orthogonal polynomials}

The \emph{classical} orthogonal polynomials occur as eigenfunctions of linear differential operators of order~2
\begin{equation} \label{eq:SL}
  \sigma(x) \psi_n''(x) + \tau(x) \psi_n'(x) + \lambda_n \psi_n(x) = 0, \quad n \geq 0, 
\end{equation}
with $\lambda_n= - n((n-1)\sigma''+2\tau')/2$. The
coefficients~$\sigma,\tau$ are polynomials of degree at most~2 and~1, related to the weight~$w$ by the Pearson differential equation
\begin{equation}\label{eq:Pearson}
(\sigma(x)w(x))'=\tau(x)w(x).
\end{equation}
Up to a suitable change of variable, the intervals $(-1,1), (0,+\infty)$ and $(-\infty,+\infty)$ describe all the possible $(a,b)$. 
Then, the case when $\sigma$ has degree~2 leads to the Jacobi
polynomials $P_n^{(\alpha, \beta)}(x), \alpha, \beta > -1$ and its
special cases (Chebyshev, Gegenbauer $C_n^{(\lambda)}(x)$, $\lambda >
- \frac{1}{2}, \lambda \neq 0$); degree~1 gives the Laguerre
polynomials  $L_n^{(\alpha)}(x), \alpha >-1$; the Hermite polynomials
are obtained for degree~0. The corresponding $a,b,w,h_n$ are given in
\zcref{table:classiques}~\cite[Table 18.3.1 and \S18.3.34]
{OlverLozierBoisvertClark2010,NIST:DLMF}, while
$\sigma,\tau,\lambda_n$ are presented in  \zcref{table:classiques2}. 

\begin{table}
   \begin{center}
\renewcommand{\arraystretch}{1.5}
     \begin{tabular}{l | c | c | c | c }
                Name & $\psi_n(x)$ & $(a,b)$ & $w(x)$ &  $h_n\ (n>0)$ 
                \\\hline
Chebyshev  & $T_n(x)$ & $(-1,1 )$ & $(1-x^2)^{-1/2}$ &  ${\pi}/
{2}$  \\\hline 
 Gegenbauer  & $C_n^{(\lambda)}(x), \genfrac{}{}{0pt}{1}{\lambda > - \frac{1}{2},}{\lambda \neq 0 }$ & $(-1,1 )$ & $(1-x^2)^{\lambda -1/2}$ &  $\frac{2^{1-2\lambda}\pi \Gamma (n+2\lambda)}{(n+\lambda) (\Gamma(\lambda))^2 n!}$ \\
       \hline
 Jacobi &  $P_n^{(\alpha, \beta)}(x), \alpha, \beta > -1$ & $(-1,1 )$ & $(1-x)^{\alpha} (1+x)^{\beta}$ &  $\frac{2^{\alpha+\beta+1} \Gamma (n+\alpha+1) \Gamma (n+\beta+1)}{(2n+\alpha+\beta+1) \Gamma(n+\alpha+\beta+1) n!}$ \\  \hline
Laguerre & $L_n^{(\alpha)}(x), \alpha > -1$ & $(0,+\infty)$ &  $e^{-x} x^{\alpha}$ &  $\frac{\Gamma (n+\alpha+1)}{n!}$ \\\hline
 Hermite & $H_n(x)$ & $(- \infty,+\infty)$ & $e^{-x^2}$ & $\pi^{1/2} 2^{n} n!$ \\
\end{tabular}
  \caption{Classical orthogonal polynomials. For $n=0$, the formulas
  for~$h_n$ still apply, except for Chebyshev ($h_0=\pi$) and Jacobi
  when $\alpha+\beta+1=0$ ($h_0=-\pi/\sin(\pi\alpha)$). See also
  \zcref{sec:note_on_chebyshev_series}.}
  \label{table:classiques}
  \end{center}
\end{table}

The polynomials $\psi_n$ also satisfy a \emph{structure
relation}
\begin{equation} \label{eq:structure}
   \sigma(x) \psi_{n+1}'(x) =  \delta_2(n) \psi_{n+2}(x) + \delta_1(n) \psi_{n+1}(x) + \delta_{0}(n) \psi_{n} (x), \quad n \geq 0,
\end{equation}
where the $\delta_k$s are rational functions in $\mathbb R(n)$.
This relation is actually a
characterization of classical orthogonal polynomials~%
\cite{AlSalamChihara1972}. For a slightly
larger class of measures, the Bessel
polynomials~\cite[\S18.34]{OlverLozierBoisvertClark2010,NIST:DLMF} can also be included in the classical family. For
simplicity, we omit them from our discussion, but our implementation
does handle them, see \zcref{example:bessel_expansions}.

\begin{table}
   \begin{center}
     \renewcommand{\arraystretch}{1.2}
     \begin{tabular}{l | c | c | c |c}
                Name & $\psi_n(x)$ &  $\sigma(x)$ & $\tau(x)$ & $\lambda_n$ \\\hline
Chebyshev  & $T_n(x)$ & $1 - x^2$ & $-x$ & $n^2$\\\hline 
Gegenbauer  & $C_n^{(\lambda)}(x), \genfrac{}{}{0pt}{1}{\lambda > - \frac{1}{2},}{\lambda \neq 0 }$  & $1 - x^2$ & $- (2\lambda + 1)x$ & $n(n+2\lambda)$\\ \hline
Jacobi &  $P_n^{(\alpha, \beta)}(x), \alpha, \beta > -1$ & $1 - x^2$ & $\beta - \alpha - (\alpha+\beta+2)x$ & $n(n+\alpha+\beta+1)$ \\  \hline
Laguerre & $L_n^{(\alpha)}(x), \alpha > -1$ & $x$ & $\alpha+1-x$ & $n$ \\\hline
  Hermite & $H_n(x)$ & $1$ & $-2x$ & $2n$ \\
\end{tabular}
  \caption{Classical orthogonal polynomials: Sturm-Liouville
  coefficients}\label{table:classiques2}
  \end{center}
\end{table}

\subsection{Generalized Fourier series}
We write $[\psi_n](f)$ for the generalized Fourier coefficient~$a_n$
of \zcref{eq:integrals}. This notation suggests that $[\psi_n]$ is the
linear operator extracting the coefficient of~$\psi_n$ in the
generalized Fourier series for~$f$ when it converges, but we use it
for the integrals~$a_n$ without any reference to convergence of the
series.

\subsection{Note on Chebyshev series}
\label{sec:note_on_chebyshev_series}
It is classical to use the symbol $\Sigma'$ for Chebyshev series, as
in
\[\sum_{n\ge0}{}'{a_nT_n}.\]
This means that the term in $T_0$, if there is one, is to be
halved~\cite[p.~27]{MasonHandscomb2003}. Indeed, the formulas for
$a_n$ often become more regular when stated that way. One reason for
this is that $T_n(\cos\theta)=\cos(n\theta)$ makes it natural to
define~$T_{-n}=T_n$ and consider sums over~$n\in\mathbb Z$, where
now~0 does not need special treatment. Working over~$\mathbb N$ leads
to formulas with special cases at~$n=0$, as in 
\zcref{table:classiques,table:pairs_rec_cop}. Since our algebraic
algorithms require operators with coefficients that are rational
functions in~$n$, they cannot contain a Kronecker symbol. The
recurrences produced by our algorithms in the
case of Chebyshev expansions have to be understood as holding for the
sequence~$a_n$ as in the formula above, i.e., with $a_0$ which is twice
the coefficient of $T_0$.

\section{Shift vs recurrence operators}\label{sec:shift-rec}

This section introduces two rings of operators with rational function
coefficients and relates their properties. Shift operators
encode the operations of shifting a sequence and multiplying it by
rational functions. For instance, the right-hand side of the structure
relation \zcref{eq:structure} is represented as
\[(\delta_2(n)S_n^2+\delta_1(n)S_n+\delta_0(n))\cdot\psi_n(x),\]
where $S_n$ denotes the \emph{shift} operator that maps a sequence $
(u_n)_{n\in\mathbb N}$ to the sequence $(u_{n+1})_{n\in\mathbb N}$,
while
`$\cdot$' denotes the
application of an operator to a sequence. Shift operators form a ring,
which is left Euclidean. Since the coefficients may have denominators
that vanish at positive integers, in general, the application of an
operator to a sequence requires $n$ sufficiently large. This
difficulty is avoided by \emph{recurrence operators} defined below,
that operate on sequences.
Again, these operators form a ring, but now it is not left principal.
However, it has minimal-degree common left multiples, that will be
sufficient for our constructions.

\subsection{Definitions}
\begin{definition}[Sequences and their germs]
Let $\K$ be a (commutative) field of characteristic~0. 
We write $\Ka$ for the algebra of \emph{sequences} of elements of
$\K$, 
with coordinate-wise addition and multiplication.
Taking the equivalence relation $(u_n)\sim(v_n)$
when there exists $N\in\N$ such that the sequences are identical for
$n>N$, the equivalence class of $(u_n)$ is called the \emph{germ} of $
(u_n)$ and the set of equivalence classes also forms a $\K$-algebra~%
\cite{PutSinger1997}.
\end{definition}

\begin{definition}[Shift operators] 
A \emph{shift operator} over $\K$ is a polynomial of the form
$\sum_{k = 0}^N a_k S_n^k$, with $a_k(n) \in \K (n)$.
These polynomials are added like commutative polynomials and
multiplied
with the commutation $S_nq(n)=q(n+1)S_n$ for $q(n)\in\K(n)$. The set
of these polynomials
forms a ring denoted $\Sh=\K(n)\langle S_n\rangle$. 
\end{definition}
The ring $\Sh$ is left and right principal, with greatest common right
and left divisors (gcrds and gclds) and least common left and right multiples 
(lclms and lcrms) computed by non-commutative analogues of Euclid’s
algorithm~\cite{Ore1933}\cite[\S1.1]{Cohn2006}. These are defined up
to multiplication by a rational function of~$\K(n)$. For
definiteness, the use of a definite
article (e.g., \emph{the} gcrd) corresponds to taking these
polynomials as monic with respect to~$S_n$.
They
are denoted $\gcrd_\Sh$,
$\gcld_\Sh$,
$\lclm_\Sh$ and $\lcrm_\Sh$ in this work to distinguish them from
operations on recurrence operators below. The $\lclm_\Sh$ of two operators is an operator that cancels $\K$-linear combinations of \emph{germs} of their solutions but not the solutions themselves in general.

\begin{example}\label{ex:PQlclm} Take $P=(n-1)S_n-n$ and $Q=(n-2)S_n-
(n-1)$. Then 
\[
\lclm_\Sh(P,Q)=(S_n-1)^2
\]
and the sequence of $\Q^\N$ defined by $u_0=u_1=0$ and $u_n=n-1$ for
$n\ge2$ is a solution of $P\cdot u_n=(n-1)u_{n+1}-nu_n=0$ that is not a solution of the $\lclm_\Sh$, since $u_2\neq2u_1-u_0$.
\end{example}

\begin{definition}[Recurrence operators]\label{def:rec-ops} Let $\Za$ be the set of polynomials in $\K[n]$ that do not vanish on $\N$. 
 Let $\K [n]_ {\Za}$
be the localization of $\K[n]$ at $\Za$; it is the ring of rational
functions whose poles lie outside of $\mathbb N$. We define the
ring of \emph{recurrence operators} over $\K$ as the subring 
$\RecI=\K[n]_{\Za}\langle S_n\rangle$ of $\Sh$. 
\end{definition}

The algebra of sequences~$\Ka$ is a left
$\RecI$-module, where the operation of $P=a_k(n)S_n^k+\dots+a_0
(n)\in\Rec$ on a sequence $(u_n)\in\Ka$ is the sequence $(v_n)$
defined
by
\[
v_i=a_k(i)u_{i+k}+\dots+a_0(i)u_i,\quad\text{for all }i\in \N. 
\]
This equation defines $v_i\in\K$ for all $i\in{\mathbb N}$ since none
of the denominators of the $a_j$ vanishes on ${\mathbb N}$. This is denoted
$P\cdot
u_n=v_n$.

\smallskip
Note that it is also possible to work with the ring~$\K[n]\langle
S_n\rangle$ instead of~$\Rec$. In practice however, the computations
with~$\K [n]\langle
S_n\rangle$ tend to produce bigger polynomials in intermediate
expressions and result in much slower computations, which
is why this article focuses on~$\Rec$.

\subsection{Common left multiples}

As $\K[n]_\Za$ is not a field but merely an integral domain, some
algebraic structure is lost in $\Rec$ compared to $\Sh$. For the
algorithms in
later sections however, it is important to combine elements of~$\Rec$
in ways similar to what is done for Ore fractions, using the Euclidean
algorithm.

\subsubsection{$\RecI$ is not left principal}
In
$\RecI$, $\gcld,\gcrd,\lclm,\lcrm$ do not
exist in general.
Here are simple examples illustrating these facts.
\begin{example} 
A greatest common left or right gcd of $n-1$ and $S_n$
would
have to divide~$n-1$ and thus have degree~0 in~$S_n$. 
Moreover, these gcds would divide all polynomials of the forms $U
(n-1)+VS_n$ or $(n-1)U+S_nV$. Thus, their constant coefficient wrt
$S_n$ would be divisible by $n-1$. This leaves only~$n-1$ as a
possible gcd, but it is neither a right nor a left divisor of $S_n$ in $\RecI$ and thus no $\gcrd$ or $\gcld$ exists. (Note however that $n-1$ is a right divisor of $S_n^2$ {in $\RecN$}, since $S_n^2=(\frac1
{n+1}S_n^2)(n-1)$.)
\end{example}

\begin{example}\label{ex:mclm}
With $P=n-2$ and $Q=S_n+1$, the following two
polynomials are left multiples of both $P$ and $Q$:
\begin{align*}
L&=(n-1)(n-2)(S_n+1)
&&=((n-2)S_n+n-1)P&=(n-1)(n-2)Q,\\
L'&=nS_n^2+2(n-1)S_n+n-2
&&=(S_n+1)^2P&=(nS_n+n-2)Q.
\end{align*}
Among the common left multiples of $P$ and $Q$ in $\RecI$, $L$ has
minimal degree in $S_n$, since it has the same degree as $Q$. Among
the degree~1 common left multiples, it can be checked that its leading
coefficient $(n-1)(n-2)$ has minimal degree in~$n$.
Thus if there existed a least common left multiple of $P$ and $Q$ it
would be equal to~$L$. The left multiples of $L$ of degree~2
in $S_n$ have leading coefficient divisible by~$n(n-1)$. This is not
the case of $L'$, showing that $L$ is not a lclm in $\RecI$.
\end{example}

\subsubsection{Common left multiples}\label{ssub:common_left_multiples}

While $\RecI$ is not left principal, since $\K[n]_{\Za}$ is an
integral domain, $\RecI$ is a left Ore domain~\cite[Prop.~1.1.4]
{Cohn2006}, which means any two elements have a nonzero common left
multiple. A constructive proof is given in \zcref{prop:mclm-rec}
below. Among all possible common multiples, 
it produces one that has minimal
order and leading coefficient of minimal degree, like $L$ in 
\zcref{ex:mclm} above. That this is possible comes from the
principality
of~$\K[n]_{\Za}$.

\begin{lemma}[Minimal elements in left ideals in $\RecI$]
\label{lemma:minimal} Every left ideal $\mathcal I\neq0$ in $\RecI$ has
a \emph{minimal} element, defined as the unique element of $\mathcal
I$ of minimal degree in $S_n$ with leading coefficient a
monic polynomial in $\K[n]$ of minimal degree.
\end{lemma}
\begin{proof} Any two polynomials in $\mathcal I$ of minimal degree in $S_n$ differ only by factors in $\K(n)$: if $A$ and $B$ are two such polynomials and $a,b$ are their leading coefficients, then $bA-aB\in\mathcal I$ has smaller degree in $S_n$ and must therefore be 0. The leading coefficients of the polynomials of minimal degree in $\mathcal I$ form an ideal in $\K[n]_{\Za}$. Since this ring is principal~%
\cite{Samuel1971}, this
ideal is generated by one element. Up to multiplication by an
invertible element of $\K[n]_{\Za}$ (a fraction with
numerator and denominator in $\Za$), it can be made
polynomial, of minimal degree and with leading coefficient~1.
\end{proof}
More is known on the structure of these ideals and the computation of
their Gröbner bases, see~\cite [Ch.~46]{Mora2016}, but will not be
needed in this work. See also the related algorithms for
desingularization of shift operators~\cite{ChenKauersSinger2016}.

\begin{definition}[mclm in $\RecI$]\label{def:mclm} If
$P$ and $Q$ are recurrence operators in $\RecI$, the 
\emph{minimal-degree common left multiple} of $P$ and $Q$ is the
minimal
element $\mclm_\RecI(P,Q)$ of the left ideal $(P)\cap(Q)$ in $\RecI$.
The cofactors $U,V$ in $\RecI$ such that $UP=VQ=\mclm_\RecI(P,Q)$ are denoted $(Q)^P$ and $(P)^Q$. 
\end{definition}
\begin{example}With the operators of \zcref{ex:PQlclm}, 
\[
\mclm_\RecI(P,Q)=(n-1)(S_n-1)P=n(S_n-1)Q=n(n-1)(S_n-1)^2.
\]
Now the sequence $u_n$ of \zcref{ex:PQlclm} does satisfy $\mclm_\RecI
(P,Q)\cdot u_n=0$.
\end{example}
Since $\K[n]_{\Za}$ is a subring of the field $\K(n)$, one can obtain a common left multiple by computing a lclm in $\Sh$ and clearing denominators~\cite[p. 79]{Mora2016}. This is made precise in the following. 
\begin{proposition}[Computation of $\mclm_\RecI$]\label{prop:mclm-rec}
Let $A,B$ be in $\RecI$ and $U,V$ left-coprime in $\Sh$ be such that
\[
UA=VB=\lclm_\Sh(A,B).
\]
Let $d_u,d_v$ in $\K[n]$ be the largest monic factors of the
denominators of $U,V$ that vanish only at nonnegative integers; 
let $d=\operatorname{lcm}(d_u,d_v)$
with $d$ monic. Then
\[
\mclm_\RecI(A,B)=d\lclm_\Sh(A,B)
\]
and the corresponding cofactors are $dU,dV$.
\end{proposition}
\begin{proof}
The left-coprimality of $U,V$ is a property of lclms in $\Sh$.

  By design, $d\lclm_\Sh (A,B)=(dU)A=(dU)B$ is a common left multiple
  of $A$ and $B$ in $\RecI$. It has minimal degree in $S_n$, since
  $\lclm_\Sh$ does in~$\Sh$. 
This implies that there exists $\delta\in\K[n]$ such that
$\mclm_\RecI(A,B)=\delta\lclm_\Sh(A,B)$. Right division by $A$ shows that $\delta U\in\RecI$. Similarly, $\delta V\in\RecI$. This implies that $\delta$ is a multiple of $d$ and therefore is equal to $d$ by minimality.
\end{proof}
This result is convenient, as the computation of $\lclm_{\Sh}$ is
implemented in several computer algebra systems. An alternative
possibility for a more direct computation is described by Mora~\cite[46.11.4]{Mora2016},  using a variant of the Euclidean algorithm based on pseudo-divisions.

\smallskip

Finally, although $\gcld$s do not exist in $\RecI$, the following
result gives an analogue of \zcref{prop:mclm-rec} that will be
sufficient for our needs. Recall that $\gcld_\Sh$ denotes
the \emph{monic} greatest common left divisor in~$\Sh$.

\begin{proposition}[Almost $\gcld$ in $\RecI$]\label{prop:almost-gcld}
Let $A,B$ in $\RecI$ and $G\in\Sh$ be such that $G=\gcld_\Sh(A,B)$.
Then there exist $\delta(n)\in\K[n]_\Za$, $\hat G,U,V$ in $\RecI$ such
that
\[
  \deg_{S_n} \hat G = \deg_{S_n} G,\quad \delta(n)A=\hat GU,\quad\delta
(n)B=\hat GV,\quad\deg_{S_n}\gcld_\Sh
(U,V)=0.
\]
\end{proposition}
\begin{proof}
The proof is effective. Let $\tilde U$ and $\tilde V$ be such that
$A=G\tilde U$ and $B=G\tilde V$. Let $g(n)$ and $u(n)$ be minimal
elements of $\K[n]_\Za$ (for the degree of the part that vanishes
only at
{elements of $\mathbb N$}) such that $g(n)G$,
$U:=u(n)\tilde U$ and $V:=u
(n)\tilde V$ all belong to $\RecI$. Define $\hat G$ and $d(n)$ by 
\begin{equation}\label{eq:commutewithG}
\lclm_\RecI(g(n)G,u(n))=\hat G u(n)=d(n)g(n)G
\end{equation}
and let $\delta(n)=d(n)g(n)$. Then,
$\delta(n)A=\delta(n)G\tilde U=\hat Gu(n)\tilde U=\hat GU$
and similarly for $B$. Finally, 
\[
  \gcld_\Sh(U,V)=\gcld_\Sh(u(n)\tilde U,u(n)\tilde V)=\gcld_\Sh
  (\tilde U,\tilde V)
\]
and the last $\gcld_\Sh$ has degree~0 by maximality of~$\gcld_\Sh
(A,B)$.
\end{proof}

\subsection{Operators in \texorpdfstring{$S_n$}{Sn} and its inverse}
 Another approach that has been exploited in the literature uses the
 inverse~$S_n^{-1}$ of the shift operator~$S_n$ and polynomials in the
 ring of Laurent polynomials~$\K[n]\langle S_n,S_n^{-1}\rangle$~%
 \cite{Lewanowicz1976,Lewanowicz1983,Lewanowicz1986,Lewanowicz1991,Lewanowicz1992,Lewanowicz1995,Lewanowicz1996,Rebillard1998,Lewanowicz2002,LewanowiczWozny2004,RebillardZakrajsek2007}. This is convenient when working with germs, as simplifications by powers of~$S_n$ are harmless in that context. 
 A difficulty is that the presence of $S_n^{-1}$ yields recurrence
 operators valid for some negative indices. This could force to extend
 the definition of orthogonal polynomials to negative indices, in
 particular in accordance with~\zcref{eq:3term-rec}. The adaptation is
 fine in the Chebyshev case, and more delicate in the other ones.
 Besides, when applying these operators for sequences rather
 than germs, one has to follow carefully the operations on operators, keeping track of a~$n_0$ such that all recurrences hold for~$n\ge n_0$. Using recurrence operators as in \zcref{def:rec-ops} avoids these difficulties and lets one still make use of Euclid's algorithm as shown in \zcref{prop:mclm-rec}.

\section{Fractions of recurrence operators}
\label{sec:module-of-pairs}

We now define fractions of recurrence operators.
Their addition and multiplication rely on the
cofactors in the
mclm (\zcref{def:mclm}). The algorithms of \zcref{sec:irred-fracs} produce 
\emph{irreducible} fractions, a notion that is
presented in \zcref{sub:irreducible_fractions}.  These fractions of
recurrence operators
resemble, but
differ from, Ore fractions; the difference is clarified at the end of
this section.

\subsection{Annihilators}
\begin{definition}[Annihilators] Given a $k$-tuple $U=
(u_1,\dots,u_k)$ of sequences in $\K^{\mathbb N}$, the 
\emph{annihilator}
of $U$ is the set
\[
\Ann_\RecI U=\{(p_1,\dots,p_k)\in\RecI^k\mid p_1\cdot u_1+\dots+p_k\cdot
u_k=0\}.
\]
It is a left $\Rec$-submodule of $\RecI^k$, $\RecI$ operating by left
multiplication on each component.
\end{definition}

The case $k=1$ corresponds to the annihilating left ideal of a
sequence. The main focus of this work is the
case when $k=2$ and the elements of these modules are called
\emph{pairs} of recurrence operators.

\subsection{Fractions of recurrence operators}
\label{sub:fractions_of_recurrence_operators}
The set $\RecI^\times$ of invertible elements of $\RecI$ is the set of
fractions in $\K(n)$ with both numerator and denominator belonging
to~$\Za$. 
\begin{definition}[$\Za$-equivalence and fractions]
\label{def:fractions} We write
$\FrecI$
for the
equivalence classes of pairs of recurrence operators in
$\RecI\times\RecI$ by the relation
\[
(N_1,D_1)\equiv_\Za(N_2,D_2)\Leftrightarrow\exists a\in\RecI^\times,
N_1=aN_2,D_1=aD_2.
\]
The elements of $\FrecI$ are called \emph{fractions of recurrence
operators}; their first and second components are the 
\emph{numerator} and \emph{denominator} of the fraction. 
\end{definition}
Checking that $\equivZA$ is an equivalence relation is straightforward, as is the following result that gives a motivation for considering this relation.

\begin{lemma}
Let $u,v$ be sequences in $\K^{\mathbb N}$, then $
(N,D)\in\Ann_\RecI (u,v)$
if and only if $(aN,aD)\in\Ann_\RecI(u,v)$ for all $a\in\RecI^\times$.
\end{lemma}
The set of $\Za$-equivalence classes of the elements of $\Ann_\RecI(u,v)$ is denoted $\Ann_\FrecI(u,v)$.

\subsection{Proper fractions}\label{sub:proper_fractions}
We call \emph{proper} a fraction that is either~$(0,0)$ or has a
nonzero denominator. 
Our algorithms computing annihilators of pairs of sequences will
mostly deal with $\RecI$-modules~$M$ of proper fractions, i.e., such
that $M\cap(\{0\}\times\RecI)=\{(0,0)\}$, $\Rec$ acting by left
multiplication on both components.
These modules behave almost
like principal ideals, in the following sense.

\begin{proposition}\label{lemma:quasi-gen-pair} If $M\neq\{(0,0)\}$
is a left $\Rec$-module of proper fractions, 
there
exists a fraction $G \in M$ that is a generator of the
$\Sh$-module generated by~$M$. Thus
for all $A \in M$, there exists $a
(n)\in \K [n]\setminus \{ 0 \} $ and $H \in \Rec$ that satisfy 
$a(n)A=H G$.
Such a pair is called a
\emph{quasi-generating fraction} of~$M$.
\end{proposition}
\begin{proof} The proof reduces to modules over the ring $\Sh$,
whose left-principal character is then exploited. 
Let $\tilde{M}$ be the left-$\Sh$-module generated
by
$M$. 

1. Properness of $M$ implies
that $ \tilde{M} \cap(\Sh\times\{0\})=\{
(0,0)\}$: The elements of~$\tilde M$ are
finite linear
combinations~$C=a_1
(n_1,d_1)+\dots+a_k (n_k,d_k)$
of elements of $M$ with coefficients~$a_i$ in~$\Sh$ rather than
$\Rec$. If such a
combination has a~0 denominator,
then, multiplying~$C$ by a common multiple of the denominators
of the~$a_i$ gives
an element of~$M$. This element is proper, which implies
that its numerator
is~0 and thus also that of~$C$.

2. $\tilde M$ has a generating pair: The projection of
$\tilde M$ on the second
coordinate is a nonzero
left ideal 
of $\Sh$. Let $y\in\Sh$~be a generator of it and~$\tilde G=(x,y)$ a
corresponding element of~$\tilde M$. For any nonzero element~$
P=(u,v)\in\tilde M$, there exists $a\in\Sh$ such that $v=ay$ and then
$P-a\tilde G=(u-ax,0)=(0,0)$ (by properness) implies that $P=a\tilde
G$. Clearing
denominators gives~$c\in\Za$ such that~$G:=c
\tilde G\in M$.

3. The pair $G$ generates~$M$: For any $A \in {M}$, there exists
$h \in \Sh$,  such that $A = h\tilde G=hc^{-1}G$.
Multiplying~$hc^{-1}$ by an appropriate polynomial~$a$ so that
$H=ahc^{-1}\in\RecI$ concludes.
\end{proof}

The following equivalence relation is classical in Ore fractions.
\begin{definition}[$\RecI$-Equivalence]\label{def:equivalent-pairs}
Two
fractions $P$ and $Q$ in $\Frec$ are  
\emph{$\RecI$-equivalent} if there exists $r,s$ in $\RecI\setminus 
\{0\}$ such that $rP=sQ$. This is denoted $
P\equiv_{\RecI}Q$.
\end{definition}
\begin{corollary}\label{lemma:equivalent-pairs}
Let $M\neq\{(0,0)\}$ be a left $\RecI$-module of proper fractions.
Then all nonzero fractions in~$M$ are $\RecI$-equivalent.
\end{corollary}
\begin{proof}
Let~$A,B$ be two nonzero fractions in~$M$ and $G$ a quasi-generating
fraction of~$M$. Then there exists $a,b$ in~$\K[n]\setminus\{0\}$ and
$H_A,H_B$ in~$\Rec$ such that
\[aA=H_AG,\quad bB=H_BG.\]
Let~$UH_A=VH_B$ be a common left multiple of~$H_A,H_B$ in~$\Rec$.
Then~$UaA=VbB$, proving the $\RecI$-equivalence of~$A,B$.
\end{proof}

\subsection{Irreducible fractions}\label{sub:irreducible_fractions}

The algorithms presented in the next sections compute fractions
annihilating pairs of sequences. The $\Za$-equivalence used to
define these fractions does not allow to simplify these fractions by
cancelling out common left factors unless they belong to $\Za$. Extra
work is thus needed to obtain quasi-generating fractions.
A simple criterion is in terms of irreducible fractions.
\begin{definition}[Irreducible fraction]
\label{def:irred-frac}\label{def:irreducible} 
A fraction $(p,q)\in\FrecI \setminus \{ (0,0) \}$ is called 
\emph{irreducible}
if the monic greatest common \emph{left} divisor of $p,q$ in $\Sh$ is~1.
\end{definition}
\noindent Note that this definition does not depend on the equivalence
class of
$(p,q)$. Irreducibility is easily detected by a gcld computation
in~$\Sh$.

\begin{lemma}[Minimality criterion]
\label{prop:minimalitySh} Let $M\neq\{(0,0)\}$ be a left
$\RecI$-module of proper fractions. 
 If $(p,q)\in M$ is irreducible,  then $ (p,q)$ is a quasi-generating
 fraction of~$M$ and all quasi-generating fractions of~$M$ are
 irreducible.
\end{lemma}
\begin{proof}
This is a direct corollary of \zcref{lemma:quasi-gen-pair}. If $
(p,q)$ is irreducible and a multiple of a quasi-generating
fraction~$G$,
then~$G$ itself has to be irreducible and therefore~$(p,q)=c(n)G$ for
some~$c(n)\in\K(n)$. This makes~$(p,q)$ a quasi-generating fraction
of~$M$. Also, all other quasi-generating fractions are multiples of
that same~$G$ by some element of~$\K(n)$, which makes them
irreducible.
\end{proof}

\subsection{Addition and multiplication of fractions}
\label{sub:operations_on_fractions}
Up to this point, the operations that have been used in $\Frec$ are
those that make it a $\RecI$ module, i.e., pairwise addition and
simultaneous multiplication by
elements of $\RecI$. We now define two more operations on these
fractions.
\begin{definition}[Addition and multiplication of fractions]%
\footnote{See \zcref{def:mclm} for the notation for cofactors of
the mclm in $\Rec$.}
\label{def:op_pairs} One defines two maps from $\FrecI\times\FrecI$ to
$\FrecI$:
\begin{align*}
  (p_1,p_2)\oplus (q_1,q_2)&=((q_2)^{p_2}p_1+(p_2)^
  {q_2}q_1,\mclm_\RecI (p_2,q_2));
\\
  (p_1,p_2)\otimes (q_1,q_2)&=((p_1)^{q_2}q_1,(q_2)^
  {p_1}p_2). 
\end{align*}
\end{definition}
\begin{proof}[Proof that the operations do not depend on the
equivalence class for $\equivZA$.]
Let $a,b$ be in $\RecI^\times$ and $p,q$ in $\RecI$. We first relate the
cofactors of $ap$ and $bq$ to those of $p,q$. First, since $\lclm$s
are taken monic with respect to $S_n$, 
we have
\[
  \lclm_\Sh \left (p, q\right) = \lclm_\Sh \left (ap,bq\right).
\]
\zcref{prop:mclm-rec} then implies
\[
  \mclm_\RecI \left (p, q\right) = \mclm_\RecI \left (ap,bq\right).
\]
In terms of cofactors, the equalities
\[
  (q)^{p} p =  (p)^{q} q = \mclm_\RecI\left (p, q\right)=
(q)^{p}\frac1a a p =  (p)^{q}\frac 1b b q
\]
give the following cofactors in $\RecI$:
\[
  (bq)^{ap}=(q)^p\frac1a,\quad 
  (ap)^{bq}=(p)^q\frac 1b.
\]
Therefore,
\begin{gather*}
(bq_2)^{ap_2}ap_1+(ap_2)^{bq_2}bq_1=
(q_2)^{p_2}\frac1a ap_1+(p_2)^{q_2}\frac1bbp_2=
(q_2)^{p_2}p_1+(p_2)^{q_2}q_1,\\
((ap_1)^{bq_2}bq_1,(bq_2)^{ap_1}ap_2)=
(p_1)^{q_2}\frac1bbq_1,(q_2)^{p_1}\frac1aap_2)=
((p_1)^{q_2}q_1,(q_2)^{p_1}p_2)
\end{gather*}
show the result for $\oplus$ and for $\otimes$.
\end{proof}
 
Our method constructs fractions inductively from the following.
\begin{proposition}[Computation of annihilators] 
\label{prop:annihilators}
Let $u,v,w$ be sequences in $\K^{\mathbb N}$.
Then
\begin{align*}
\Ann_\FrecI(u,-v)\oplus \Ann_\FrecI
(u,-w)&\subset\Ann_\FrecI(u,-
(v+w)),\\
\Ann_\FrecI(v,-w)\otimes \Ann_\FrecI
(u,-v)&\subset\Ann_\FrecI(u,-w).
\end{align*}
\end{proposition}
\begin{proof} Take $(p_1,p_2)\in\Ann_\FrecI(u,-v)$ and $(q_1,q_2)\in\Ann_\FrecI (u,-w)$. Then the matrix
\[
N=\begin{pmatrix}p_1&-p_2&0\\ q_1&0&-q_2\end{pmatrix}
\]
is such that $N\cdot(u,v,w)^T=0$. Multiplying the first row by $(q_2)^
{p_2}$, the second one by $(p_2)^{q_2}$ and adding gives the first
inclusion. For the second one, take $(p_1,p_2)\in\Ann_\FrecI(v,-w)$ and
$(q_1,q_2)\in\Ann_\FrecI(u,-v)$. The matrix
\[
M=\begin{pmatrix}0&p_1&-p_2\\q_1&-q_2&0\end{pmatrix},
\]
is such that $M\cdot(u,v,w)^T=0$. Multiplying the first row by $(q_2)^{p_1}$, the second one by $(p_1)^{q_2}$ and adding concludes.
\end{proof}

\subsection{Ore fractions}
\label{rk:Orefrac} 
Since $\RecI$ is a left Ore domain, Ore's construction can be applied
to construct the skew field $\operatorname{Frac}(\RecI)$~\cite{Ore1931}. This is
obtained by defining sum and product with the operations of 
\zcref{def:op_pairs} (except that any common left multiple can be
used in place of the $\mclm$) and quotienting by $\RecI$-equivalence.
We call these
fractions \emph{Ore fractions} in order to distinguish them
from those constructed and used here.

By contrast, the fractions in $\FrecI$ are not equivalence classes
for $\Rec$-equivalence.
They are pairs. They can be added by adding numerators
and denominators. Numerator and denominator can be left multiplied by
the same element of $\RecI$. They can also be subjected to the
operations $\oplus$ and $\otimes$ of \zcref{def:op_pairs}, which is
why we use the word `fractions' for them. 
However, $\FrecI$ is not a ring with these operations; it is not even
a group for~$\oplus$. 

The motivation for using $\FrecI$ rather than $
\operatorname{Frac}
(\RecI)$ is that $\RecI$-equivalence of pairs is not  compatible with
operations on sequences or on their germs. For instance,
if $u_n$ is the constant sequence defined by $u_n=1$ for all $n$ and similarly $v_n=2$ for all $n$, then clearly $(S_n-1,S_n-1)\in\Ann_\Sh
(u_n,v_n)$, but the $\RecI$-equivalent pair $(1,1)$ does not belong to
$\Ann_\Sh (u_n,v_n)$. 

\section{Recurrences for coefficients}\label{sec:rec_coeff}
Given a linear differential operator~$L$ (defined formally in \zcref{sub:linear_differential_operators}),
it
is classical that
two power series~$f$ and~$g$ related by the linear
differential equation $L\cdot f=g$ have sequences
of
Taylor coefficients that are related by a linear recurrence operator.
In the case of generalized Fourier coefficients, this approach
extends, but the relation involves a \emph{fraction} of recurrence
operators.
Then when~$g=0$, the \emph{numerator} of that fraction gives a
linear recurrence for the Fourier coefficients of~$f$. 

\begin{figure}[h]
\begin{tikzpicture}[
    node distance=1cm and 2cm,
    mynode/.style={align=center, font=\normalsize},
    myarrow/.style={->, >=Stealth, thick, shorten >=2pt, shorten <=2pt}
]

\node[mynode] (diffop) {$\mathbb{K}[x]\langle \partial \rangle$};
\node[mynode, below=-0.2cm of diffop] (diffoptexte) {(differential operators)};
\node[mynode, above right=1.5cm and 0.5cm of diffop] (diffterms) {$\mathcal{D}$ (differential terms)};
\node[mynode, below right=1.5cm and 0.5cm of diffterms] (prec) {$\FrecI$};

\node[mynode, below=-0.12cm of prec] (pairs) {(fraction of recurrence
operators)};
\path (diffoptexte.east) -- (pairs.west) node[pos=0.5] {};

\draw[myarrow] (diffterms) -- (diffop) node[midway, left=0.15cm] {$
\mathfrak{D}_{\delta}$};
\draw[myarrow] (diffterms) -- (prec) node[midway, right=0.1cm] {$\mathfrak{R}_{\delta}^\psi$};

\end{tikzpicture}

  \caption{
\label{fig:difftermsopsrec}}
\end{figure}

Our approach is summarized in \zcref{fig:difftermsopsrec}.
The input
is a differential operator in the bottom left leaf
of the picture.
This operator can be written in many ways, corresponding to
distinct differential terms (\zcref{sub:differential_terms}) at
the root of the tree in the picture. To
each of these terms one can associate a fraction of recurrences 
(bottom right leaf).
Depending
on the choice of the term and sometimes on the choice of the
derivation~$\delta$, this gives distinct recurrence operators that can
sometimes be combined to reduce the order of the final recurrence.
Such a recombination is described in the next section.
This approach is very versatile. It is first illustrated in
the
simple case of Taylor series (\zcref{sub:taylor_expansions}), before
turning to the more involved case of generalized Fourier expansions 
(\zcref{sub:recurrences_for_generalized_fourier_coefficients}).


\subsection{Linear differential operators}\label{sub:linear_differential_operators}
A \emph{derivation} $\delta$ on the ring of polynomials~$\K[x]$
over the field~$\K$ of characteristic~0 is a $\K$-linear
map that satisfies $\delta(ab)=a\delta(b)+b\delta(a)$. For instance,
$\delta=d/dx$ is a derivation, but also more generally~$\delta=q
(x)d/dx$ for $q(x)\in\K[x]$ is a derivation and all derivations
of~$\K[x]$ are of this type~\cite[VIII \S5 Prop.~5.5]{Lang2002}.
Examples of such
derivations are used in \zcref{sec:singular-case}. 

\emph{Linear differential operators} are polynomials in a
variable~$\partial$ with coefficients in~$\K[x]$, with usual
commutative addition and a multiplication defined by the commutation
$\partial a=a\partial+\delta(a)$ for $a\in\K[x]$. The ring of linear
differential operators is denoted~$\K
[x]\langle\partial;\delta\rangle$ and we omit~$\delta$ when it is the
usual derivation~$d/dx$. 

\subsection{Differential terms}\label{sub:differential_terms}

\begin{figure}
\begin{tikzpicture}[scale=0.4, every node/.style={transform shape}]
\node (A) {
\begin{forest}
for tree={
  draw,
  circle,
  minimum size=1.2em,
  inner sep=1pt,
  s sep=7mm,
  l sep=10mm,
  align=center
}
[$\times$
  [$\times$
    [$\partial$]
    [$+$
      [$x$]
      [$+$
        [$x$]
        [$1$]
      ]
    ]
  ]
  [$\partial$]
]
\end{forest}
};

\node[right=4cm of A] (B) { 
\begin{forest}
for tree={
  draw,
  circle,
  minimum size=1.2em,
  inner sep=1pt,
  s sep=7mm,
  l sep=10mm,
  align=center
}
[$\times$
  [$+$
    [$\times$
      [$+$
        [$\times$
          [$2$]
          [$x$]
        ]
        [$1$]
      ]
      [$\partial$]
    ]
    [$2$]
  ]
  [$\partial$]
]
\end{forest}
};
\end{tikzpicture}
\caption{Tree representation of two differential terms:\\ $
(\partial\times (x+
(x+1)))\times\partial$ (left) and $((2\times
x+1)\times\partial+2)\times\partial$ (right).
$\mathfrak D_{d/dx}$ maps both to the differential operator
$(2x+1)\partial^2+2\partial$.
\label{fig:diff-terms}}
\end{figure}
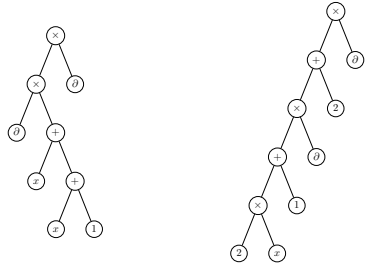

We make use of differential terms. They can be thought of as binary
trees with $+,\times$ at the internal nodes and $x,\partial$ and
elements of $\K$ at the leaves. No assumption on associativity or
commutativity of the operations is implied. Two differential terms are
displayed in \zcref{fig:diff-terms}. This notion is formalized as
follows.
\begin{definition} Let $\mathcal D$ be the term algebra (see, e.g.,
\cite{Hodges1997}) whose
constants are the elements of $\K$ and the symbols $x,\partial$,
together with two binary operations $+,\times$. The elements of
$\mathcal D$ are called \emph{differential terms}. 
\end{definition}

If $\delta=q(x)d/dx$ is a derivation of~$\K[x]$, we write $\mathfrak
D_\delta$ for the homomorphism from $\mathcal D$ to the
ring of linear differential operators $\K
[x]\langle\partial\rangle$ that maps the elements of~$\K$
and~$x$ to themselves, $\partial$ to $q(x)\partial$, and the
operations $+,\times$ to these operations in $\K
[x]\langle\partial\rangle$. 
Note that
many distinct elements of $\mathcal D$ are mapped to the same element
of $\K
[x]\langle\partial\rangle$. For instance the terms in \zcref{fig:diff-terms} both
map to the differential operator $(2x+1)\partial^2+2\partial$ in $\K
[x]\langle\partial\rangle$. This
will let us construct several different recurrences for Fourier
coefficients of solutions of linear differential equations, and
combine them.

\subsection{Taylor expansions}\label{sub:taylor_expansions}
We first illustrate the approach in the simple case of Taylor coefficients.

\begin{definition}[Taylor morphism] We write  $\mathfrak R_{
\mathrm{Taylor},d/dx}$
for the homomorphism from $(\mathcal D,+,\times)$ to $
(\Frec,\oplus,\otimes)$ defined by
\[
  \mathfrak R_{\mathrm{Taylor},d/dx}
  (x)=(1,S_n),\qquad\mathfrak R_{\mathrm{Taylor},d/dx}(\partial)=((n+1)S_n,1)
\]
and
for
all $\lambda\in\K$, $\mathfrak R_{\mathrm{Taylor},d/dx}(\lambda)=(\lambda,1)$.
\end{definition}

\begin{proposition}[Recurrences for Taylor coefficients] Let
$t\in\mathcal D$ be a differential term, $L$ be the differential
operator $\mathfrak D_{d/dx}(t)$ and let $(\mathcal R_\text{num},\mathcal
R_\text{den})=\mathfrak R_{\mathrm{Taylor},d/dx}(t)$. Then for any power
series
$F\in\K [[x]]$\footnote{The notation $[x^n](F)$ represents the
coefficient of~$x^n$ in the power series~$F$.},
\[
\mathcal R_\text{num}\cdot[x^n] (F)=\mathcal R_\text{den}\cdot[x^n] 
(L\cdot F).
\]
In particular, if $F$ is a power series such that $L\cdot F=0$ then
its sequence of Taylor coefficients is annihilated by the numerator:
$\mathcal R_\text{num}\cdot[x^n](F)=0$.
\end{proposition}
\begin{proof} The proof is by induction. For any $\lambda\in\K$ and
any sequence $u$, one has $(\lambda,1)\in\Ann_\FrecI(u,-\lambda u)$. If
$(f_n)$ denotes the sequence of coefficients of $F$, then the
coefficient of $x^n$ in $xF$ is $f_{n-1}$, thus they are related by $
[x^n](F)=S_n[x^n](xF)$ or equivalently $(1,S_n)\in\Ann_\FrecI([x^n]
(F),-[x^n](xF))$. Similarly, the coefficient of $x^n$ in $F'$ is $(n+1)f_
{n+1}$, which proves the case of $\partial$. If $t=t_1+t_2$, let
$L_i=\mathfrak D_\delta(t_i)$ for $i=1, 2$. Then we have $L=L_1+L_2$ and by
induction, for any $F$, $\mathfrak R_{\mathrm{Taylor}}(t_i)\in\Ann_\FrecI([x^n]
(F),-[x^n](L_i(F)))$ for $i=1,2$ so that the result follows by 
\zcref{prop:annihilators}. Finally, if $t=t_1\times t_2$, then with
$L_i$ as above, $L=L_1L_2$ and by induction, for any power series
$F,G$, $\mathfrak R_{\mathrm{Taylor}}(t_1)\in\Ann_\Frec([x^n](G),-[x^n](L_1
(G)))$ and $\mathfrak R_{\mathrm{Taylor}}(t_2)\in\Ann_\FrecI([x^n](F),-[x^n](L_2
(F)))$. With $G=L_2(F)$, \zcref{prop:annihilators} concludes.
\end{proof}

\begin{example} Consider the following two differential terms with
identical image $L$ by $\mathfrak D_{d/dx}$: $x\partial -2$ and $\partial
x-3$.

Applying $\mathfrak R_{\mathrm{Taylor},d/dx}$ to the first one gives the
fraction $(
(n-1)S_n,S_n)$, while the second one produces $(n-2,1)$. If one takes
for $F$ the power series $x^2$ which is such that $L\cdot F=0$, then
one
deduces two different recurrences for its coefficients $f_n$: the
first one gives $(n-1)f_{n+1}=0$ while the second one gives $
(n-2)f_n=0$. Obviously, the polynomial coefficients cannot be removed
if one does not consider germs only. Both recurrences hold for all
$n\ge0$. In both cases, one deduces that $f_1=0$ and $f_j=0$ for
$j>2$, but only the second one captures the fact that $f_0=0$. This is due to the fact that the pair $( (n+1)S_n, 1)$ does not capture all the information about multiplication by $x$.
The same type of phenomenon for generalized Fourier coefficients occurs below.

Another approach in this example consists in using the
derivation~$\theta=x{d}/{dx}$, with  $\mathfrak R_{
\mathrm{Taylor},\theta} (\partial)=
(n,1)$ and the same values for $x$ and constants. Then
the differential term of interest is~$\partial -2$,
for which~$\mathfrak R_{
\mathrm{Taylor},\theta}$
gives
the recurrence $(n-2)f_n=0$, as desired. Analogous constructions
in the case of generalized Fourier coefficients are given in \zcref{sec:singular-case}.
\end{example}

\subsection{Recurrences for generalized Fourier coefficients}\label{sub:recurrences_for_generalized_fourier_coefficients}

\begin{table}
\begin{center}
\renewcommand{\arraystretch}{1.5}
\begin{tabularx}{.9\textwidth}{p{1.8cm} | c | Y | Y}
Name & $\psi_n$ & $\Xnum$ & $\Dden$\\
\hline
Chebyshev
& $T_n$
& \cellscript{
   \tfrac12 S_n^2
   + \tfrac{1}{2}
 }
& \cellscript{
   -\tfrac{1}{2(n+1)} S_n^2
   + \tfrac{1}{2(n+1)}
 } \\
\hline
Gegenbauer
& $C_n^{(\lambda)}$
& \cellscript{
   \frac{n+1+2\lambda}{2(n+2+\lambda)} S_n^2 
   + \frac{n+1}{2(n+\lambda)}
 }
& \cellscript{
   -\frac{1}{2(n+2+\lambda)} S_n^2 
   + \frac{1}{2(n+\lambda)}
 } \\
\hline
Jacobi
& $P_n^{(\alpha,\beta)}$
& \cellscript{
   \frac{2(n+\alpha+2)(n+\beta+2)}
     {(2n+4+\alpha+\beta)_2} S_n^2 \\
   + \frac{\beta^2-\alpha^2}
     {(2n+2+\alpha+\beta)(2n+4+\alpha+\beta)} S_n \\
   \quad+ \frac{2(n+1)(n+1+\alpha+\beta)}
     {(2n+1+\alpha+\beta)_2}
 }
& \cellscript{
   -\frac{2(n+2+\alpha)(n+2+\beta)}
     {(n+2+\alpha+\beta)(2n+4+\alpha+\beta)_2} S_n^2 \\
   + \frac{2(\alpha-\beta)}
     {(2n+2+\alpha+\beta)(2n+4+\alpha+\beta)} S_n \\
   \quad+ \frac{2(n+1+\alpha+\beta)}
     {(2n+1+\alpha+\beta)_2}
 } \\
\hline
Laguerre
& $L_n^{(\alpha)}$
& \cellscript{
   -(n+\alpha+2) S_n^2
   + (2n+3+\alpha) S_n \\
   \quad\quad-(n+1)
 }
& \cellscript{
   S_n - 1
 } \\
\hline
Hermite
& $H_n$
& \cellscript{
   (n+2) S_n^2 + \tfrac12
 }
& \cellscript{
   \tfrac{1}{2(n+1)}
 } \\
\end{tabularx}
\end{center}
  \caption{Recurrence operators $\Xnum$ and $\Dden$ associated to
  classical orthogonal polynomials (see \zcref{eq:xnumxden} and 
  \zcref{eq:dnumddencop}). For Chebyshev, see
  \zcref{sec:note_on_chebyshev_series}.}
  \label{table:pairs_rec_cop}
\end{table}

In this section, $\RecI$ and $\FrecI$ depend
on the family $\psi_n$ through the ground field $\K$: $\K=\Q$ for Chebyshev and Hermite, $\K=\Q(\lambda)$ for Gegenbauer, $\Q(\alpha)$ for Laguerre and $\Q(\alpha,\beta)$ for Jacobi. Extensions are also possible if one works with differential equations involving other parameters.

\subsubsection{Basic operations on generalized Fourier coefficients}

From the 3-term recurrence \eqref{eq:3term-rec}, we derive two 
operators 
\begin{equation}\label{eq:xnumxden}
  \Xnum = \frac{1}{h_{n+1}}  (a(n) S_n^2+b(n) S_n+c(n))h_n, \quad \Xden  = S_n.
\end{equation}
The values of $\Xnum$ for classical orthogonal polynomials are
displayed in~\zcref{table:pairs_rec_cop}. Note that they all belong
to~$\Rec$. Indeed, in the case of Jacobi
polynomials, when $\alpha +
\beta + 1 = 0$, the constant coefficient of 
$\Xnum$ simplifies to $({n+1})/({2n+1})$.

\begin{lemma}\label{lem:x} 
  Let $f$ be 
 such that the sequence  of Fourier coefficients $([\psi_n](f))_{n
   \in \nat}$ exist. 
 Then $([\psi_n](xf))_{n\in\nat}$ exists and they are related by the
 recurrence
 \begin{equation}\label{eq:xnumnxdenn}
   \Xnum \cdot[\psi_n](f)= \Xden \cdot [\psi_n](xf)\quad \textrm{
   for all } n \in \N.
\end{equation}
\end{lemma}
\begin{proof} Part of the proof is due to Lewanowicz~%
\cite{Lewanowicz1995}. We give it here for the sake of
completeness. {The existence of the sequence  $([\psi_n](xf))_{n \in
\nat}$ is a consequence of the fact that, for all $n\geq 0$, $x \psi_n
(x)$ is a finite linear combination of $\psi_{n+1}, \psi_n$ and $\psi_
{n-1}$, by the 3-term recurrence~\eqref{eq:3term-rec}, and the
existence of $ ([\psi_n](f))_{n \in \nat}$}.
The proof reduces to 
\begin{align*}
 h_{n+1} \Xden \cdot [\psi_n](xf)  & =  \langle  x f | \psi_{n+1} \rangle_w =  \langle   f |  x  \psi_{n+1} (x) \rangle_w  \\
 &  =   \langle  f |   (a(n) S_n^2+b(n) S_n+c(n)) \cdot \psi_n(x)
 \rangle_w \\ 
  & =  (a(n) S_n^2+b(n) S_n+c(n))   \cdot \langle  f |  \psi_n(x) \rangle_w \\
  & = h_{n+1}  \Xnum \cdot [\psi_n](f).\qedhere
\end{align*}
\end{proof}

For the differentiation operator for classical orthogonal polynomials, we start from  the structure  relation~\eqref{eq:structure} and define  two  operators
\begin{equation}\label{eq:dnumddencop}
 \Dnum  = S_n, \qquad  \Dden =  \frac{1}{\lambda_{n+1}h_{n+1}} \cDiffden h_n,
\end{equation}
where $\cDiffden  =  \delta_2(n) S_n^2 + \delta_1(n) S_n + \delta_{0}(n)$.
The values of $\Dden$ are displayed in~\zcref{table:pairs_rec_cop}.
Again, they all belong to~$\Rec$: the constant coefficient of $\Dden$
in the case of Jacobi polynomials   when  $\alpha  + \beta +1 =0
$ simplifies to $1/({2n+1})$.

\begin{lemma} \label{lem:Dcop} Let $f$ be such that the Fourier
coefficients $([\psi_n](f))_{n \in \nat}$ and  $([\psi_n](f'))_{n \in \nat}$ exist. 
They are related by the recurrence
\begin{equation}\label{eq:Dcop}
  \Dnum \cdot[\psi_n](f) = \Dden \cdot [\psi_n](f')\quad \textrm{
  for all } n \in \N.
\end{equation}
\end{lemma}
\begin{proof} Here again, part of the proof is due to Lewanowicz~%
  \cite{Lewanowicz1995}.  We write it for the sake of readability.

For any $n \in \nat$,
\begin{align*}
h_{n+1} \Dden\cdot[\psi_n](f')&=
\lambda_{n+1}^{-1} (\delta_2(n) S_n^2 + \delta_1(n) S_n + \delta_{0}(n))  \cdot \langle f' |\psi_n\rangle_w \\
&=  \lambda_{n+1}^{-1} \langle f'|\sigma\psi'_{n+1}\rangle_w \\
&= S_n \cdot \langle f|\psi_n\rangle_w = h_{n+1}\Dnum \cdot [\psi_n]( f),
\end{align*}
by 
the definition of $\Dden$, the structure
relation~\eqref{eq:structure} and the following lemma.
\end{proof}

\begin{lemma}\label{lemma:int-fprime}
If $(\psi_n)_{n\in\mathbb N}$ is one of the \emph{classical} families of orthogonal
polynomials on~$(a,b)$ and $f$ is 
 such that the sequences $([\psi_n](f^{(j)}))_{n \in \nat}$ exist for $j = 0, 1$, then $\langle\sigma\psi_n'|f'\rangle_w = \lambda_n\langle\psi_n|f\rangle_w$ for all $n \in \nat$.
\end{lemma}
\begin{proof} The case $n =0$ is obvious. For all $a < u \le v <  b$, integration by parts gives, for any $n \in \nat \setminus \{ 0 \}$,
\begin{align*}
  \int_u^v f'(x)& \psi_{n}' (x) \sigma (x) w(x) \mathrm{d}x\\
  &=  \left [  f(x) \psi_{n}' (x) \sigma (x) w(x) \right ]_u^v - 
  \int_u^v f(x) (\psi_{n}' (x) \sigma (x) w(x))' \mathrm{d}x\\
  &=  \left [  f(x) \psi_{n}' (x) \sigma (x) w(x) \right ]_u^v + 
  \int_u^v f(x) \lambda_n\psi_n(x)w(x) \mathrm{d}x.
\end{align*}
We now show that the term between brackets is~0, following
Nikiforov and Uvarov~\cite [Chap.
2, \S 8]{NikiforovUvarov1988}. The existence and finiteness of the limit of the second integral as $u\rightarrow a$ or $v\rightarrow b$ follows from the fact that $[\psi_n](f)$ exists.
Similarly, by the structure relation  \eqref{eq:structure}, the limit of the first integral exists and is finite, since $[\psi_{n-1+k}](f)$ exists for $k = 0, 1, 2$.   
This implies that there exist $A, B \in \ree$ such that $\lim_{u \rightarrow a}  f(u) \psi_{n}' (u) \sigma (u) w(u) = A$ and  $\lim_{v \rightarrow b}  f(v) \psi_{n}' (v) \sigma (v) w(v) = B$. If $A\neq 0$, then as $u\rightarrow a$, $f(u) \psi_n' (u) w(u) \sim A/\sigma (u)$, which contradicts the existence of $\int_a^b f(u) \psi_n' (u) w(u) \mathrm{d}u$, given  the values of $\sigma(x)$ in Tables~\ref{table:classiques} and~\ref{table:classiques2}. 
  The same argument proves $B =0$.
\end{proof}

\subsubsection{Fourier morphism}
In order to manipulate convergent integrals only, it is
important to pay attention to the degree of the terms.

\begin{definition}[Degree of a term] The \emph{degree} of a
differential term $t\in\mathcal D$ wrt $\partial$ is denoted
$\deg_\partial t$ and is defined inductively: $\deg_\partial t=0$ if
$t\in\K\cup\{x\}$, $\deg_\partial \partial =1$ and
$\deg_\partial(t_1+t_2)=\max(\deg_\partial t_1,\deg_\partial t_2)$,
$\deg_\partial(t_1\times t_2)=\deg_\partial t_1+\deg_\partial t_2$.
The degree in $x$ can be defined similarly; it will not be necessary
in this work.
\end{definition}

\begin{definition}[Fourier morphism]\label{def:fourier-morphism}
  We write $\mathfrak R^\psi_{d/dx}$ for the homomorphism from $
  (\mathcal
  D,+,\times)$ to $ (\FrecI,\oplus,\otimes)$ defined by
\begin{equation}\label{eq:Phi_psi}
  \mathfrak R^\psi_{d/dx}(x)=(\Xnum, \Xden),\qquad
\mathfrak R^\psi_{d/dx}(\partial)=(\Dnum,\Dden)
\end{equation}
and for all $\lambda\in\K$, $\mathfrak R^\psi_{d/dx}(\lambda)=
(\lambda,1)$.
\end{definition}
The analogue of the proposition in the Taylor case gives a first
algorithm for the Fourier case. It takes into account the
constraints on the operations $x$ and $\partial$ coming from analysis.

\begin{proposition}[Recurrences for generalized Fourier coefficients]
\label{prop:rec_Fourier}
Let $(\psi_n)$ be a
family of \emph{classical} orthogonal polynomials. 
Let $t\in\mathcal D$ be a differential term, $L$ be the differential
operator $\mathfrak D_{d/dx}(t)$ and let $\mathfrak R^\psi_{d/dx}
(t)=(\mathcal R_
\text{num},\mathcal R_\text{den})$. Then for any function $f$ such
that {$ ([\psi_n](f^{(j)}))_{n \in \nat}$ exists} for $0\le
j\le\deg_\partial t$, the sequence $([\psi_n](L\cdot f))_{n\in\nat}$
also exists and they are related by
\begin{equation}\label{eq:rec_Fourier}
\mathcal R_\text{num}\cdot [\psi_n] (f)=
\mathcal R_\text{den}\cdot[\psi_n] (L\cdot f).
\end{equation}
In particular, if $f$ satisfies the properties above and $L\cdot f=0$
then
its sequence of Fourier coefficients is annihilated by the numerator:
\[\mathcal R_\text{num}\cdot[\psi_n](f)=0.\]
\end{proposition}
\begin{proof} The proof is very similar to that of the Taylor case,
using the lemmas above for the cases $x$ and $\partial$. The only
novelty is that it is necessary to pay attention to the {existence of
the Fourier coefficients} of the functions that are differentiated.
The existence of that of $L\cdot f$ follows by linear combination.
That
the property holds at each step of the induction follows from the definition of degree and the hypothesis on the functions~$f$.
\end{proof}

\begin{remark}
This result depends on the derivation~$d/dx$ through \zcref{lem:Dcop}
only. It extends \emph{mutatis mutandis} to any other
derivation~$\delta$ as
soon as an analogue of \zcref{eq:Dcop} is available for~$\delta$. This
is crucial in \zcref{sec:singular-case}.
\end{remark}

\subsubsection{Associated fractions}
\zcref[S,noabbrev]{prop:rec_Fourier} lets us
associate
fractions of
recurrence operators
to a
linear differential operator, in the following sense.
  \begin{definition} \label{def:assocpair}
    Let $  L=p_r(x)\partial^r+\dots+p_0(x)$ be a linear differential
    operator in $ {\K}[x]\langle \partial\rangle $. A 
    fraction $(\Pnum,\Pden) \in \Frec$ is said to  be 
    \emph{associated to
    $L$ in degree~$k$} for a family $(\psi_n)_{n\in\nat}$ if there is a
    differential
    term $t\in\mathcal D$ \emph{with} $\deg_\partial t\le k$ and a
    derivation $\delta$ on $\K[x]$ such that
    \[\mathfrak D_{\delta}(t)=L\quad\text{and}\quad \mathfrak R^
    \psi_\delta (t)=
    (\Pnum,\Pden).\]
    It is said to be \emph{associated to $L$} 
    (without reference to a degree) when for all~$j\ge0$, 
    \[\Pnum\cdot[\psi_n](\psi_j)=\Pden\cdot[\psi_n](L\cdot\psi_j).\]
  \end{definition}
This definition is consistent: by \zcref{prop:rec_Fourier}, a fraction
associated to~$L$ in degree~$k$ is associated to~$L$.

As in the case of Taylor coefficients, there are many
fractions of recurrences associated  to a given linear differential
operator. They generate modules of which our algorithms compute
generators. 
\begin{definition}
The \emph{algebraic module of fractions associated} to a linear
differential
operator~$L$ for the family~$(\psi_n)_{n\in\N}$ is the left
$\RecI$-module generated by all fractions in $\Frec$ associated
to~$L$ in degree~$\deg_\partial L$. It is denoted~$\assocal$.
Similarly, the \emph{analytic module of fractions associated} to~$L$
for the family~$(\psi_n)_{n\in\N}$ is the left
$\RecI$-module generated by all fractions in $\Frec$ associated
to~$L$. It is denoted~$\assocan$.
\end{definition}

The motivation for considering $\assocal$ comes from 
\zcref{prop:rec_Fourier}:
\begin{corollary}\label{coro:rec_Fourier_module}
Let $(\psi_n)$ be a family of classical orthogonal
polynomials and $L$ a linear differential operator. For any
function~$f$
such that~$ [\psi_n] (f^ {(j)})_ {n\in\N}$
exists for $j\le\deg_\partial L$ and any $(p,q)$ in~$\assocal$,
$p\cdot
[\psi_n](f)=q\cdot[\psi_n](L\cdot f)$.
In particular, if $L\cdot f=0$ then $p\cdot[\psi_n](f)=0$.
\end{corollary}
\begin{proof}
Since $(p,q)\in\assocal$, there are finitely many
fractions $(p_1,q_1),\dots,(p_m,q_m)$ in $\Frec$ associated to~$L$ in
degree~$\deg_\partial L$ and finitely many recurrence operators
$o_1,\dots,o_m$ in $\Rec$ such that 
\[(p,q)=o_1(p_1,q_1)+\dots+o_m(p_m,q_m).\]
By \zcref{prop:rec_Fourier}, for each $i=1,\dots,m$, $p_i\cdot
[\psi_n](f)=q_i\cdot[\psi_n](L\cdot f)$. Multiplying on the left by~$o_i$ and adding each of these
recurrence equations gives the result.
\end{proof}
A direct consequence is that $\assocal\subset\assocan$.

\begin{proposition} \label{prop:assoc-equiv}
Let $L$ be a linear differential operator in $ {\K}[x]\langle
\partial\rangle $ and let $(\psi_n)$ be a family of classical orthogonal
polynomials. Then,
\begin{enumerate}[label=(\alph*)]
\item \label{prop:assoc-equiv:a} all nonzero fractions of
$\assocan$ are $\RecI$-equivalent;
\item \label{prop:assoc-equiv:b} there exists a quasi-generating
fraction of both $\assocal$ and $\assocan$ (see def. in 
\zcref{lemma:quasi-gen-pair});
\item if there exists an irreducible fraction in $\assocal$, it is a
quasi-generating fraction of both $\assocal$ and $\assocan$ and all
quasi-generating fractions of $\assocal$ or $\assocan$ are
irreducible.
\end{enumerate}
\end{proposition}
\begin{proof}
By 
\zcref{lemma:quasi-gen-pair,lemma:equivalent-pairs,prop:minimalitySh}
and the inclusion $\assocal\subset\assocan$,
it is sufficient to show that~$\assocan$ contains only proper
fractions.

Indeed, assume a fraction $(P,0)$ belongs to $\assocan$. By 
\zcref{coro:rec_Fourier_module} applied to~$f=\psi_k$, the
operator~$P$ cancels the sequences $([\psi_n] (\psi_k))_{n \in
\nat}$ for all~$k$. This is an infinite dimensional vector space of
sequences solution of~$P$, so that~$P=0$.
\end{proof}

\section{Irreducibility and operators of minimal degree}
\label{sec:irred-fracs}

\zcref[S,noabbrev]{prop:rec_Fourier} provides us with a family of
algorithms that are fixed by choosing a differential term  
representing an input linear differential operator
\[L=p_r(x)\partial^r+\dots+p_0(x)\]
with polynomial coefficients~$p_i(x)$.
A priori, each
choice
produces a different fraction of recurrence operators. 
\zcref[S,noabbrev]{prop:assoc-equiv} shows that they are all
$\RecI$-equivalent and that finding an irreducible one ensures that it
is a quasi-generator of $\assocan$.
In this section and the next one, we show how to compute
such irreducible fractions of recurrence operators that belong
to the annihilator
$\Ann_\Rec([\psi_n](f),-[\psi_n](L\cdot f))$, 
under explicit convergence constraints on the functions~$f$ whose
Fourier
series is considered. 

\subsubsection*{Elementary computations}
For polynomials $p(x)=a_0+\dots+a_dx^d$ in~$x$ only
and no $\partial$, the
differential term corresponding to \emph{right Horner} evaluation, i.e., 
\begin{equation}\label{eq:horr-pol}
\HorrI(p(x))=\mathfrak R^\psi_{d/dx}\!\left((((\dotsb((a_d\times
x)+a_
{d-1})x+\dotsb)+a_1)\times x)+a_0\right)
\end{equation}
is associated to a fraction that is always irreducible
(\zcref{prop:horr-pol-irred}).
Moreover, the computation of this fraction is efficient thanks to the
special form $\Xden=S_n$ in~\eqref{eq:xnumxden}: the
denominators are all powers of~$S_n$, cofactors in the
additions and multiplications of fractions (\zcref{def:op_pairs}) can
be computed easily. Note also that this computation depends only
on the pair $(\Xnum,\Xden)$ and not on the
specific
derivation~$d/dx$; this irreducibility result therefore holds for
any derivation.

\subsubsection*{Left Horner evaluation}
The irreducible fractions from \zcref{eq:horr-pol} are the building blocks of the
construction of fractions for linear differential operators. 
In particular, after first rewriting the operator
in the form
\[L=\partial^rq_r(x)+\dots+q_0(x),\]
its \emph{left Horner} evaluation is
  \begin{equation}\label{eq:Horner_gauche}
    \Horl(L)=\mathfrak R^\psi_{d/dx}\!\left(\partial(\dotsb \partial 
    (\partial
    q_r
    (x)+q_ {r-1} (x)) +
    \dots + q_1 (x)) + q_0(x)\right),
  \end{equation}
where each of the coefficients $q_i(x)$ are themselves evaluated by
right
(not left) Horner evaluation as in \zcref{eq:horr-pol}. 
It is important to note that the special form~$\Dnum$ $ =S_n$
in~\eqref{eq:dnumddencop} makes the multiplications and additions easy
to compute efficiently: again, the denominators are always
simply powers of~$S_n$. Thus all operations are elementary and it is not necessary to discuss Euclid's algorithm in
that situation. Using polynomials in~$S_n$ and~$S_n^{-1}$
instead of fractions, special cases of this method have been developed
by Paszkowski~\cite[T13.1, Chap. 13]{Paszkowski1975} for
Chebyshev expansions, by Lewanowicz~\cite[\S 3]{Lewanowicz1976} for 
Gegenbauer polynomials and later for Jacobi polynomials~\cite{Lewanowicz1983,Lewanowicz1986}, \cite[\S11.2.3]{Wimp1984}.

Another feature of this representation is that
when the leading coefficient $q_r(x)(=p_r(x))$ is relatively prime
with the
polynomial~$\sigma(x)$ from the definition of the family~$(\psi_n)$
in \zcref{eq:Pearson},
then the fraction is irreducible (\zcref{prop:Lewanowicz}).
This makes it our method of choice
in this case. A typical example is the computation of a recurrence
of order~12 with coefficients of degree up to~28 for
the coefficients of the expansion of the product of Hermite
polynomials~$H_i (x)H_j (x)$ in
the basis of Jacobi polynomials~$P_n^{(\alpha,\beta)}(x)$, in a
matter of seconds (see \zcref{ssub:linearization_coefficients}).

\subsubsection*{Right Horner evaluation}
When the leading coefficient of the differential
equation is not relatively prime with~$\sigma(x)$, left Horner
evaluation does not produce an irreducible fraction in general. This
is the situation where the algebraic machinery introduced in the previous
sections pays off through non-trivial applications of Euclid's
algorithm. The \emph{right Horner} representation is 
  \begin{equation}\label{eq:Horner}
\HorrZ (L)=\mathfrak R^\psi_{d/dx}\!\left((\dotsb(p_r(x)\partial+p_
{r-1}
(x))\partial+\dotsb+p_1 (x))\partial+p_0 (x)\right)
  \end{equation}
where each of the coefficients $p_i(x)$ is itself evaluated by right
Horner evaluation as in \zcref{eq:horr-pol}. This fraction is
irreducible
up to a power of~$S_n$ as possible common factor~(%
\zcref{prop:horner_droite}). In particular, when $S_n$ does not divide
the numerator or the denominator, then the result is an irreducible
fraction.
\begin{example}
For the expansion of $\exp(x)$ on the basis of Chebyshev polynomials,
the linear differential operator~$\partial-1\in\K
[x]\langle\partial\rangle$ maps to the fraction
\[
  (S_n^2+2(n+1)S_n-1,1-S_n^2).
\]
Neither numerator nor denominator is divisible by~$S_n$. 
\zcref{prop:horner_droite} implies that the fraction is irreducible. The numerator gives a recurrence of
order~2 for the coefficients of the Chebyshev expansion of $\exp(x)$.
The recurrence can be checked in this example where the coefficients
are known to equal the sequence $2I_n (1)$, with $I_n$ the
modified Bessel
function of the first kind~\cite[10.29.1]
{OlverLozierBoisvertClark2010}. (See 
\zcref{sec:note_on_chebyshev_series} for the coefficient of $T_0$.)
\end{example}

\subsubsection*{Combination of left and right Horner evaluation} 
Finally, there are cases when the leading coefficient of the
differental operator is not relatively prime to~$\sigma(x)$ while the
right
Horner evaluation produces a fraction with a nonzero power of~$S_n$ in
factor of both its numerator and denominator. Working in $\Rec$ rather than
$\Sh$ does not allow us to conclude that the common power of~$S_n$ can be
simplified out.
\begin{example}
Howell~\cite{Howell1937} gave an explicit expression for the
coefficients of the expansion of the Laguerre polynomials~$L_m^{
(\alpha)}(x)$ in the basis~$L_n^{(\beta)}(x)$. The leading
term of the differential equation is~$x=\sigma(x)$, so that
left Horner
evaluation is not expected to produce an irreducible fraction, and
indeed it produces
\[\left((n-m+\beta-\alpha+2)S_n^2
+(2m-2n+\alpha-\beta-2)S_n
+(n-m),
-(S_n-1)^2
\right),\]
where the numerator factors as $((n-m+\alpha+\beta+2)S_n+m-n)(S_n-1)$.
Right Horner evaluation produces the fraction
\begin{equation}\label{eq:righthorner-exp}
\left((n-m+\beta-\alpha+2)S_n^2+(m-n-1)S_n,S_n-S_n^2\right),
\end{equation}
with $S_n$ as a common factor of the numerator and denominator.
By \zcref{prop:rec_Fourier}, the numerator gives a 2-term recurrence of
order~2 for the coefficients. It is
however possible to obtain a \emph{first} order equation by
subtracting termwise the second pair from the first,
leading to the right factor of the quotient of the numerator of 
\zcref{eq:righthorner-exp} by $S_n$. This is the result obtained by
our implementation, see
\zcref{ssub:connection_coefficients}.
\end{example}
The process outlined in this example is an instance of our
algorithm in the general case. Even when the fraction produced by left
Horner evaluation is not irreducible, its numerator and denominator
cannot have a power of~$S_n$ in common (\zcref{coro:horner_elim_Sn}).
Combining both fractions makes it possible to remove the common power
of~$S_n$ from $\Horr(L)$, resulting in an irreducible fraction, at the possible expense of a
factor~$c (n)\in\mathcal Z_\mathbb N$. This approach is summarized in
the following result.
\begin{theorem}[Recurrences for Fourier coefficients]
\label{thm:horner}
Let $(\psi_n)$ be a family of {classical} orthogonal polynomials 
and let $L\in\K[x]\langle\partial\rangle$ be a linear
differential operator. Then \zcref[S]
{algo:main} with $L$ as input returns an irreducible, therefore
quasi-generating,
fraction $(\Hnum,\Hden)\in\assocal$. In other words, if $f$
is a function such that
\begin{equation}\zcsetup{reftype=hypothesis}
\label{eq:hypH}\tag{H}
\text{the sequences }([\psi_n](f^{(j)}))_{n \in \nat}\text{ exist
for }j\le\deg_\partial L,
\end{equation}
then the sequence $([\psi_n](L\cdot f))_{n\in\nat}$ exists and 
\begin{equation}\label{eq:outputAlgo1}
  \Hnum\cdot[\psi_n](f)=\Hden\cdot[\psi_n](L\cdot f).
\end{equation}
In particular, if $f$ satisfies \zcref{eq:hypH} and $L\cdot f=0$ then
its
sequence of generalized Fourier coefficients is annihilated by the
numerator: $\Hnum\cdot[\psi_n](f)=0$.
\end{theorem}

\begin{example}\label{ex:exp}
In any of the bases of \zcref{table:classiques} except for the
Laguerre polynomials, the function $\exp(x)$ satisfies~
\zcref{eq:hypH}. Since $\exp(x)$ is cancelled by
$\partial-1\in\K [x]\langle\partial\rangle$, the corresponding
Fourier coefficient sequences are cancelled by a numerator of $(\Dnum, \Dden) \oplus(- 1,1)$. Since $\Dnum=S_n$, one such numerator is $S_n - \Dden,$ with $\Dden$ from the corresponding entry in \zcref{table:pairs_rec_cop}.
Here are the final expressions returned by our program in the bases
$T_n (x)$, $H_n(x)$, $C_n^{(\lambda)}(x)$ (the expression for $P_n^{
(\alpha,\beta)}(x)$ is larger and omitted here): 
\begin{Verbatim}
> for Basis in [ChebyshevT(n,x),HermiteH(n,x),GegenbauerC(n,lambda,x)]
    do FunctiontoRec(exp(x),Basis,u) od;
\end{Verbatim}
{\color{blue}
\begin{gather*}
u_{n+2}+2(n+1)u_{n + 1}-u_n  =0,\\
2(n+1)u_{n + 1}-u_n =0,\\
(\lambda+ n)u_{n+2}+ 2(\lambda + n)(\lambda + n + 2)u_{n + 1} -(\lambda + n + 2)u_{n} 
=0.
\end{gather*}}
A purely symbolic use of \zcref{thm:horner} for $\exp$ in the basis of
Laguerre polynomials, where \zcref{eq:hypH} does not hold, 
results in the nonsensical $u_n=0$. This shows the importance of
convergence
considerations when using these methods.
\end{example}

The rest of this section is devoted to a rather technical proof
of the irreducibility of the fraction, that can safely be skipped in a
first reading. The existence of $[\psi_n](L\cdot f)$ and the
recurrences then
follow
as in \zcref{prop:rec_Fourier,coro:rec_Fourier_module}.

Throughout, if $A = \sum_
{j=0}^d a_j(n) S_n^j\in \Sh$, then for any $k \in \rel$, we write
 \[A(n+k,S_n)= \sum_{j=0}^d a_j(n+k) S_n^j \in \Sh,\]
 which satisfies the commutation $S_n^kA(n,S_n)=A(n+k,S_n)S_n^k$.
 Note that if $k\in\mathbb N$ and $A(n,S_n)\in\RecN$, then $A
(n+k,S_n)\in\RecN$. 

\subsection{Right Horner evaluation of polynomials in 
\texorpdfstring{$x$}{x}}

\begin{algorithm}[t]
    \SetAlgoRefName{RightHornerPoly}\SetKwFunction{RightHornerPoly}{RightHornerPoly}
      \Input{$P=\sum_{i=0}^d p_ix^i \in \K[x]$; 
                \ the pair $X^\psi =(\Xnum,\Xden)$ from 
                \zcref{eq:xnumxden}}
      \Output{The irreducible fraction $\HorrI(P)\in\FrecI$ from 
      \zcref{eq:horr-pol}}
        $S \gets (0,1)$\\
        \For{$i\gets d$ \KwTo $0$}{
            $S \gets S\otimes X$\qquad \tcp{$\otimes$ and
            $\oplus$ from \zcref{def:op_pairs}}
            $S\gets S\oplus (p_i,1)$
        }
        \KwRet{S}
    \caption{Horner evaluation of polynomials in $x$} 
    \label{algo:HornerPol}
\end{algorithm}

\zcref[S]{algo:HornerPol} implements the formula defining $\Horr$ in
\zcref{eq:horr-pol}. The irreducibility of its output is given by the
following.

\begin{proposition}\label{prop:horr-pol-irred}
Let $P = \sum_{i = 0}^d p_i x^i$ be a nonzero polynomial in $\K[x]$ and let $(\psi_n)$ be a  family of classical orthogonal polynomials. Then the fraction
    $\HorrI (P)$ from \zcref{eq:horr-pol}
is irreducible. Its denominator is $S_n^{\deg P}$.
\end{proposition}
\begin{proof}
The proof is by induction on $\deg P$. The statement is clear
when~$\deg P=0$. Otherwise, $P(x)=Q(x)x+p_0$ with $\deg Q=\deg P-1$.
The induction hypothesis gives $\mathfrak R^\psi_{d/dx} (Q) =  (B,S_n^
{\deg P-1})$ and this fraction is irreducible. \zcref{def:op_pairs}
then gives 
\begin{align*}
  \mathfrak R^\psi_{d/dx} (P) &=( \mathfrak R^\psi_{d/dx} (Q)
  \otimes \mathfrak R^\psi_{d/dx} (x)) \oplus  \mathfrak R^
  \psi_{d/dx} (p_0) \\ & = (
  (B,S_n^{\deg P-1})  \otimes  (\Xnum,S_n))  \oplus  
  (p_0,1)\\
  & =  (  (B)^{S_n} \Xnum,(S_n)^{B} S_n^{\deg P-1})  
  \oplus (p_0,1)\\
  & =  (  B(n+1,S_n) \Xnum,S_n^{\deg P})  
    \oplus (p_0,1)\\
  &= (B(n+1,S_n) \Xnum+S_n^{\deg P}p_0,S_n^{\deg P}).
\end{align*}
By irreducibility and the special case of constant polynomials, $S_n$ is not a right divisor of $B$, which allows us to use  \zcref{lemma:tech-pairs-item-c} below in the fourth line and gives the irreducibility of the left $\oplus$ summand in that line, which implies that of the result.
 \end{proof}
\begin{lemma}\label{lemma:tech-pairs-item-c}
Let $A \in \Rec\setminus\{0\}$ be relatively prime with~$S_n$.
     Then for any $j \in \nat$, 
       \[
         \mclm_{\Rec} (A,S_n^j)=
             (A(n+j,S_n))S_n^j=(S_n^{j})A,
      \]
      where the parenthesized terms give the cofactors.
\end{lemma}
\begin{proof}
That the products on the right are left multiples in $\Rec$ of both~$A$
and~$S_n^j$ is clear from their presentation. In both cases, the
cofactors are left-coprime since $S_n$ being prime with~$A$ implies
that it does not divide~$A(n+j-k,S_n)$. Thus the co-factors
have minimal degree in~$S_n$ and this is also the case for the
products.  The minimality of the leading coefficient up to invertible elements in $\Rec^\times$ comes from the fact that the product of~$A$ by an element of~$\Rec$ of degree~$j$ in~$S_n$ is a multiple of the leading coefficient of its product by~$S_n^{j}$.
\end{proof}

\subsection{Left Horner}\label{sec:lefthorner}

\begin{algorithm}[t]
    \SetAlgoRefName{LeftHorner}
    \SetKwFunction{LeftHorner}{LeftHorner}
      \Input{$L=\sum_{i=0}^r \partial^iq_i(x) \in \K
      [x]\langle\partial\rangle$;\\
      \qquad the pairs $X^\psi = (\Xnum,\Xden)$, $D^\psi =
      (\Dnum,\Dden)$ from
    \zcref{eq:Phi_psi}}
      \Output{The evaluation of $\Horl(L)\in\Frec$ from
      \zcref{eq:Horner_gauche}}
        $S \gets (0,1)$\\
        \For{$i\gets r$ \KwTo $0$}{
            $S \gets D\otimes S$ \tcp{{\zcref{def:op_pairs}}}
            $Q_i \gets $ \zcref[noname,font=\textsf]{algo:HornerPol}$
            (q_i,X)$\\
            $S\gets S\oplus Q_i$ \tcp{{\zcref{def:op_pairs}}} }
        \KwRet{S}
    \caption{Left Horner evaluation of differential terms} 
    \label{algo:LeftHorner}
\end{algorithm}
\zcref[noabbrev,S]{algo:LeftHorner} is a direct implementation of
\zcref{eq:Horner_gauche}, where the polynomial coefficients are still evaluated by right Horner evaluation, so as to take advantage of the efficiency of that computation and irreducibility of the result. Its use relies on the following classical observation, whose proof is a simple induction. 
\begin{lemma} \label{lem:conversion}
Let $L = \sum_{j=0}^{r} p_j(x) \partial^{j}\in\K
[x]\langle\partial\rangle$, with $p_j (x) \in \K
[x]$, then $L$ can be written $\sum_{j=0}^{r}  \partial^{j} q_j(x)$,
with $q_j(x)\in\K[x]$ 
for $j = 0, \ldots, r$ and $q_r(x) = p_r(x)$.
\end{lemma}
The polynomials~$q_i$ can be computed inductively starting with~$q_k=p_k$ and 
subtracting~$\partial^kq_k$ to produce a smaller order operator.
Writing a differential operator in this way was used by
Paszkowski~\cite[Ch.~13]{Paszkowski1975} to derive expansions in the
Chebyshev basis. This was then extended by Lewanowicz to Gegenbauer
and to Jacobi polynomials~\cite{Lewanowicz1976,Lewanowicz1983}. The
following proposition gives a synthetic formula
in the general case. 
\begin{proposition}\label{prop:horner_gauche}
  Let~$L= \partial^r q_r(x) +\dots+\partial q_1(x)+q_0(x)$ be a
  linear differential operator in $\K[x]\langle\partial\rangle$.
  Let
  $(\psi_n)$ be a family of {classical} orthogonal polynomials.
  The fraction $\Horl(L)$ from \zcref{eq:Horner_gauche}
 is of the form 
\begin{multline} \label{eq:paire_horner_gauche}
  (\Hlnum,\Hlden)  = \\ \left  ( \sum_{k = 0}^r S_n^{m  - \deg q_k + k}
  \Ddenrmk(n+\deg q_{k},S_n) Q_{k},\quad   S_n^{m} \Ddenr \right ),
\end{multline}
with $m = \max_{k = 0, \ldots, r}  ( \deg q_k - k)^+$, with $x^+=\max 
(x,0)$ denoting the positive part, and
for $i = 0, \ldots, r$, the operator
 $Q_i \in \RecN$ is defined by $(Q_i,S_n^{\deg q_i})=\HorrN(q_i)$,
 while $\Ddeni$ is defined by
 \begin{equation}\label{lem:tech_paires} 
 \Ddeni = \prod_{j = 0}^{i-1} \Dden(n+i-1-j,S_n).
 \end{equation}
\end{proposition}
\begin{proof}
For $i = 0, \ldots, r$, let $d_i = \deg q_i$ and recall that the
fraction
$(Q_i, S_n^{d_i})$ is irreducible by~\zcref{lem:tech_paires}.
We prove by induction that the differential term corresponding to the
left Horner evaluation of 
\[\partial^{r-i} q_r(x) +\dots+q_i(x)\]
has for image by $\mathfrak R^\psi_{d/dx}$ the fraction
\[
  (\Qnumi,\Qdeni) = \left(\sum_{k = i, \ldots, r} S_n^{m_i - d_k + k -i}
\Ddenrmk(n+d_{k},S_n) Q_{k},
S_n^{m_i} \Ddenrmi
\right),
\]
with $m_i = \max_{k = i, \ldots, r}  ( \deg q_k + i - k)^+$.

For $i=r$, $m_r=d_r$ and the result~$(Q_r,S_n^{d_r})$ is given 
by~\zcref{prop:horr-pol-irred}. If the result holds at rank~$k$, then
Horner evaluation first evaluates the product
\begin{align*}
  (S_n,\Dden)\otimes (\Qnumi,\Qdeni) 
  & =  (\ppcm{S_n}{ S_n^{m_i}\Ddenrmi} \Qnumi, \ppcm{ S_n^{m_i}\Ddenrmi}{S_n} \Dden) \\
    & = ( S_n^{(1 -m_i)^+} \Qnumi,  S_n^{(m_i-1)^+} \Ddenrmi
    (n+1,S_n) \Dden)\\
    & = ( S_n^{(1 -m_i)^+} \Qnumi,  S_n^{(m_i-1)^+} \Ddenrmipu).
\end{align*}
Next, we use $m_{i-1}=\max ( (m_i - 1)^+ , d_{i-1})$ in the
 addition 
\begin{align*}
&( S_n^{(1 -m_i)^+} \Qnumi,  S_n^{(m_i-1)^+} \Ddenrmipu) \oplus (Q_
{i-1},S_n^{d_{i-1}})\\
& = (S_n^{m_{i-1} - (m_i-1)^+}  S_n^{(1-m_i)^+} \Qnumi +
  S_n^{m_{i-1}-d_{i-1}}  \Ddenrmipu(n+d_{i-1},S_n) Q_{i-1}, \\
&  \qquad\qquad\qquad
  S_n^{m_{i-1}} \Ddenrmipu)\\
& = ( S_n^{m_{i-1} - (m_i-1)}  \Qnumi +
  S_n^{m_{i-1}-d_{i-1}}  \Ddenrmipu(n+d_{i-1},S_n) Q_{i-1}, \\
&  \qquad\qquad\qquad
  S_n^{m_{i-1}} \Ddenrmipu).
\end{align*}
This gives $Q_{\text{den},i-1}=S_n^{m_{i-1}} \Ddenrmipu$ and
\[
Q_{\text{num},i-1}=S_n^{m_{i-1} - (m_i-1)}  \Qnumi +
  S_n^{m_{i-1}-d_{i-1}}  \Ddenrmipu(n+d_{i-1},S_n) Q_{i-1},
\]
which concludes the proof of the induction.
\end{proof}
Before stating the next technical lemma, we recall that the valuation in $S_n$ is the function
\begin{equation}\label{eq:valSn}
\valSn : \sum_{k =0}^{m} a_k(n) S_n^k \in \RecI \setminus 
\{0 \} \mapsto \min \{ k \mid a_k(n) \neq 0 \}.
\end{equation}
\begin{lemma} \label{lem:nonzeronum}
  Under the same assumptions as~\zcref{prop:horner_gauche},   if $L
  \neq 0$, let $m' = \min_{\genfrac{}{}{0pt}{1}{k = 0, \ldots, r}{\textrm{s.t. } Q_k \neq 0}}m -  ( \deg q_k - k)$, then the valuation in $S_n$ of the numerator 
   of the pair~\eqref{eq:paire_horner_gauche} is $m'$.  
\end{lemma}
\begin{proof}
By
 definition of $m'$,   $S_n^{m'}$ divides
 the numerator. Let $0
 \leq k_0 < \cdots < k_{s} \leq r$ be the
 indices such that $m' = m  + k_j - \deg q_{k_j} $ for $j = 0, \ldots, s$. For all the other indices, the term   $ S_n^{m  - \deg q_{k} + k}  \Ddenrmk(n+\deg q_{k}) Q_{k}$ is divisible  by $S_n^{m'+1}$, hence cannot contribute to the term in $S_n^{m'}$.

From~\zcref{table:pairs_rec_cop}, it follows that the
 constant term of  $\Ddenrmkj(n+\deg q_{k_j}) Q_{k_j}$ is of the same
 order as
   $ n^{\deg q_{k_j}} =  n^{k_j + m - m'}$   if 
   the $\psi_n$'s are Laguerre  polynomials;
   ${n^{k_j-r}}$    in the other cases.

 This implies that, whatever the polynomial family, these constant
 terms cannot cancel each other. Hence, the
 coefficient of $S_n^ {m'}$ of the numerator is nonzero.
\end{proof}
\begin{corollary}\label{coro:horner_elim_Sn}
  Let~$L = \partial^r q_r(x) +\dots+\partial q_1(x)+q_0(x) $ be a
  linear differential operator in $\K[x]\langle\partial\rangle$, let
  $(\psi_n)$ be a family of classical orthogonal polynomials.  The
  pair~\eqref{eq:paire_horner_gauche} is such that $S_n$
  does not
  divide its two elements simultaneously.   Its denominator is
  nonzero; so is its numerator if $L \neq 0$.
\end{corollary}
\begin{proof} 
The denominator $S_n^m \Ddenr$ is nonzero.
\zcref[S]{lem:nonzeronum} implies that the numerator $\Hlnum$ is
nonzero if $L\neq 0$.
 Regarding the first statement, if $m = 0$, it is a consequence of
 $S_n \nmid \Ddenr = \Hlden$.
If $m \geq 1$, then $m'$ from~\zcref{lem:nonzeronum} is~0 and that
 lemma gives $S_n \nmid \Hlnum$.  
\end{proof}

\subsection{Right Horner}\label{sec:righthorner}
\begin{algorithm}[t]
    \SetAlgoRefName{RightHorner}\SetKwFunction{RightHorner}{RightHorner}
      \Input{$L=\sum_{i=0}^r p_i(x)\partial^i \in \K
      [x]\langle\partial\rangle$;\\
      \qquad the pairs $X^\psi = (\Xnum,\Xden)$, $D^\psi =
      (\Dnum,\Dden)$ from \zcref{eq:Phi_psi}}
       \Output{The evaluation of $\HorrZ(L)\in\FrecZ$ from
      \zcref{eq:Horner}}
        $S \gets (0,1)$\\
        \For{$i\gets r$ \KwTo $0$}{
            $S \gets S\otimes D$
            \tcp{{\zcref{def:op_pairs}}}
            $P_i \gets $ \zcref[noname,font=\textsf]{algo:HornerPol}$
            (p_i,X)
            $\\
            $S\gets S\oplus P_i$ \tcp{{\zcref{def:op_pairs}}} }
        \KwRet{S}
    \caption{Right Horner evaluation of differential terms} 
    \label{algo:RightHorner}
\end{algorithm}

The next proposition shows that the fraction obtained thanks to a
classical
right Horner scheme as given by \zcref{eq:Horner} is
not far from irreducible: the $\gcld_\Sh$ of its
numerator and
denominator is a power of~$S_n$. \zcref[noabbrev]{algo:RightHorner}
implements this result.
It is an extension of an algorithm
 stated by Lewanowicz in the Gegenbauer case~%
\cite{Lewanowicz1976}.

\begin{proposition}\label{prop:horner_droite}
  Let~$L=p_r(x)\partial^r+\dots+p_1(x)\partial+p_0(x)$ be a linear
  differential operator in $\K[x]\langle\partial\rangle$.
   Let $(\psi_n)$ be a family of {classical} orthogonal
   polynomials. The fraction $\HorrZ(L)$ from \zcref{eq:Horner}
   is of the form $
   (S_n^\ell \Hnum,S_n^\ell\Hden)\neq(0,0)$,  with $\ell \in \nat$,
   such that $\gcld_\Sh  (\Hnum,\Hden)= 1$.
\end{proposition}
\begin{proof}
If $r=0$, then $\HorrZ(p_0)$
is irreducible by~\zcref{prop:horr-pol-irred} and the property holds.
From there, the proof is by induction using the following
lemmas.
\end{proof}

\begin{lemma}\label{lemma:shift-gcd}
Let $A,B$ be shift operators in $\Sh$. 
If $\gcld_\Sh(A,B)$ is a power of~$S_n$, then so is $\gcld_\Sh
(A,BS_n^m)$ for any $m\in\mathbb N$.
\end{lemma}
\begin{proof}
We write $\mid_\ell$ for left divisibility in $\Sh$.
In general, it is not the case that $\gcld_\Sh(AC,BC)=\gcld_\Sh
(A,B)C$, but the situation is different with~$S_n$ that is always
simultaneously a right and a left factor of the operators it divides. 
Thus, for $d\in\mathbb N$,
$\gcld_\Sh
(AS_n^d,BS_n^d)=\gcld_\Sh
(A,B)S_n^d$. This property can be seen by multiplying by~$S_n^d$ on the
right the Bézout identity for $(A,B)$ and commuting~$S_n^d$ with the
cofactors.
It follows that 
\[\gcld_\Sh(A,BS_n^m)|_\ell\gcld_\Sh(AS_n^{m},BS_n^m)=\gcld_\Sh
(A,B)S_n^{m}.\qedhere\]
\end{proof}

\begin{lemma}
Let $P\in\FrecZ$ be the fraction $(\Pnum,\Pden) = (S_n^\ell\Hnum,S_n^\ell\Hden)$ with $\gcld_\Sh(\Hnum,\Hden)=1$. Then the product $P \otimes (\Dnum,\Dden)$ has the
same form.    
\end{lemma}
\begin{proof}
This product is
\[
  (\ppcm{\Pnum}{\Dden}S_n,\ppcm{\Dden}{\Pnum}\Pden).
\]
The cofactors in this pair are related to those of $\Hnum$ and
$\Dden$: from
\[\gcld_\Sh(\ppcm{\Hnum}{\Dden},\ppcm{\Dden}{\Hnum})=1,\]
which is a consequence of the definition of cofactors, it follows that
\[\gcld_\Sh(S_n^\ell\ppcm{\Hnum}{\Dden},S_n^\ell\ppcm{\Dden}
{\Hnum})=S_n^\ell.\]
Since $\Dden$ is relatively prime with~$S_n$, so is the cofactor
$\ppcm{\Dden}{\Hnum}$ and thus by \zcref{lemma:shift-gcd},
\[\gcld_\Sh(S_n^\ell\ppcm{\Hnum}{\Dden},\ppcm{\Dden}
{\Hnum(n+\ell)})=1,\]
so that
\[
\ppcm{\Pnum}{\Dden}=S_n^\ell\ppcm{\Hnum}{\Dden}\quad\text{and}\quad
\ppcm{\Dden}{\Pnum}=\ppcm{\Dden}{\Hnum(n+\ell)}.
\]
Thus the product above rewrites
\[(S_n^\ell\ppcm{\Hnum}{\Dden}S_n,\ppcm{\Dden}
{\Hnum(n+\ell)}\Pden)=
(S_n^\ell\ppcm{\Hnum}{\Dden}S_n,S_n^\ell\ppcm{\Dden}{\Hnum}\Hden).\]
The gcld of this pair is therefore $S_n^\ell$ times $\gcld_\Sh(
\ppcm{\Hnum}{\Dden}S_n,\ppcm{\Dden}{\Hnum}\Hden)$. By
\zcref{lem:irred} below, the irreducibility
of
$(\Hnum,\Hden)$ implies that \[\gcld_\Sh(\ppcm{\Hnum}{\Dden},
\ppcm{\Dden}
{\Hnum}\Hden)=1,\] 
whence the conclusion by~\zcref{lemma:shift-gcd}.
\end{proof}
\begin{lemma} \label{lem:irred}
Let $(p_1,p_2)\in\FrecZ^2$ be an \emph{irreducible}
fraction and $q\in\RecZ$. Then the fraction $(p_1,p_2)\otimes (1,q)$ is
irreducible.
\end{lemma}
\begin{proof}
By definition of $\otimes$ and irreducibility, it is sufficient
to
prove that 
\[G:=\gcld_\Sh(
\ppcm{p_1}{q},
\ppcm{q}{p_1}p_2)=1.\]
From
\[
  G\mid_\ell \ppcm{p_1}{q} \mid_\ell \ppcm{p_1}
  {q}q=\mclm_\RecZ(p_1,q)=\ppcm{q}{p_1}p_1,
\]
it follows that
\[
G\mid_\ell \gcld_\Sh(\ppcm{q}{p_1}p_1,\ppcm{q}{p_1}p_2)=
\ppcm{q}{p_1} \gcld_\Sh(p_1,p_2)
=\ppcm{q}{p_1}.
\]
Therefore
\[G\mid_\ell\gcld_\Sh\!\left(\ppcm{p_1}{q},\ppcm{q}{p_1}\right)=1,\]
where the last equality comes from the minimality of $\mclm_\RecZ$ and 
\zcref{prop:mclm-rec}.
\end{proof}

\begin{lemma}
Let $P\in\FrecZ$ be the fraction $(\Pnum,\Pden)$
with $\gcld_\Sh(\Pnum,\Pden)=S_n^\ell$ for some~$\ell\in\mathbb N$ and
let $q\in\K [x]$ with $(Q,S_n^d)=\HorrZ(q)$. Then the sum $P
\oplus (Q,S_n^d)$ has the same form as~$P$. 
\end{lemma}
\begin{proof}
The computation of the sum begins with the cofactors of $\Pden$ and
$S_n^d$ in their $\mclm_\RecZ$.
First, there exists~$c(n)\in\K[n]$ such that $c(n)\Pden=S_n^mA$ with
$m\ge \ell$, $A\in\Rec$ and $S_n$ not dividing~$A$ in~$\Sh$. 
The cofactors then follow from
\[
  S_n^{(d-m)^+}c(n)\Pden=
S_n^{(m-d)^+}A(n+d,S_n)S_n^d
=\mclm_\RecZ(\Pden,S_n^d)
\]
and the sum of the lemma is
\[
  P \oplus (Q,S_n^d)=(S_n^{(d-m)^+}c(n)\Pnum+S_n^{(m-d)^+}
  A(n+d,S_n)Q, S_n^{(m-d)^+}A(n+d,S_n)S_n^d).
\]
The proof that the $\gcld_\Sh$ of this pair is a power
of~$S_n$ starts from
\[
  S_n^\ell=\gcld_\Sh(\Pnum,\Pden)=\gcld_\Sh(c(n)\Pnum,c(n)\Pden).
\]
Multiplying on the left by $S_n^{(d-m)^+}$ gives
\[
S_n^{\ell+(d-m)^+}=
\gcld(S_n^{(d-m)^+}c(n)\Pnum,S_n^{(m-d)^+}A(n+d,S_n)S_n^d).\]
Then the following $\gcld_\Sh$
are
also powers of~$S_n$:
\begin{multline*}
\gcld_\Sh(S_n^{(d-m)^+}c(n)\Pnum,S_n^{(m-d)^+}A(n+
d,S_n))=\\
\gcld_\Sh(S_n^{(d-m)^+}c(n)\Pnum+S_n^{(m-d)^+}A(n+d,S_n)
{Q},S_n^{
(m-d)^+}A(n+d,S_n))
\end{multline*}
and then, by \zcref{lemma:shift-gcd}, so is
\[\gcld_\Sh(S_n^{(d-m)^+}c(n)\Pnum+S_n^{(m-d)^+}
A(n+d,S_n){Q},S_n^{(m-d)^+}A(n+d,S_n)S_n^d),\]
as was to be proved.
\end{proof}

\subsection{Irreducible fractions}
\label{sec:lewafirst}

We prove the existence of an irreducible fraction in $\assocal$. By
\zcref{prop:assoc-equiv}, it is a quasi-generator of both~$\assocal$
and~$\assocan$. The
proof is effective and forms the basis of \zcref[S,noabbrev]
{algo:main} below.

\begin{proposition}\label{prop:irred-frac}
  Let~$L=p_r(x)\partial^r+\dots+p_1(x)\partial+p_0(x)$ be a linear
  differential operator in $\K[x]\langle\partial\rangle$.
   Let $(\psi_n)$ be a family of {classical} orthogonal
   polynomials. There exists an irreducible fraction in the module
   $\assocal$. 
\end{proposition}
\begin{proof}
By \zcref{prop:horner_droite}, the fraction $\Horr(L)\in\assocal$ 
is of the form $(S_n^\ell\Hnum,$  $ S_n^\ell\Hden)$ with $\Hnum,\Hden$ relatively
prime. Thus if $\ell=0$, the fraction is irreducible and gives the
result.

Otherwise, the pair $(\Hnum,\Hden)$ itself 
does not necessarily belong to~$\Frec$, but 
there exists $q(n)\in\K[n]_\Zn$ such that
\begin{equation}\label{eq:qAB}
q(n)(\Hnum,\Hden)=:(A,B)\in\Frec
\end{equation}
with $(A,B)$ irreducible and $q(n)$ can be taken of minimal degree.

By \zcref{prop:assoc-equiv}, the fraction $S_n^\ell(A,B)$  and
therefore also the fraction $ (A,B)$ are
$\RecI$-equivalent to $\Horl
(L)$.
Let $G$ be the $\gcld_\Sh$ of the members of $\Horl(L)$.
By \zcref{coro:horner_elim_Sn}, $G$~is relatively prime to~$S_n$.
By \zcref{prop:almost-gcld}, there exists $\delta(n)\in\K[n]_\Zn$ and
$\hat G,C,D$ in $\Rec$ such that 
\[
  \delta(n)\Horl(L)=(\hat GC,\hat GD)\in\assocal,
\]
 with $(C,D)$ an
 irreducible fraction. By \zcref{eq:commutewithG}, $\hat G$
also is relatively prime to~$S_n$. 
It follows from \zcref{prop:assoc-equiv} again that
 $(A,B)\equiv_{\RecI}(C,D)$ and since they are irreducible, they differ only by factors in $\K[n]_\Zn$. Up to taking left multiples of $\Horr(L)$ and $\Horl(L)$ by elements of $\K[n]_\Zn$, we can assume that they are equal.

As~$\hat G$ and~$S_n$ are relatively prime, there exist~$U$ and $V$ in $\Rec$ such that $U\hat G+VS_n^\ell=c(n)\in\K[n]_\Zn\setminus\{0\}$. It follows that
\[
  U\delta(n)\Horl(L)+Vq(n+\ell)\Horr(L)=c(n)(A,B)
\]
is an irreducible fraction that belongs to~$\assocal$.
\end{proof}

\subsection{Irreducible fractions by left Horner evaluation in
regular cases} 
\label{sec:faster}

An important property of \zcref{algo:LeftHorner} is that under a
certain coprimality condition, its output is irreducible. This
property and its proof are an extension of a result due to
Lewanowicz~\cite[Lemma 3.1]{Lewanowicz1991} for Jacobi polynomials.

\begin{proposition}
\label{prop:Lewanowicz}
With the notation of~\zcref{prop:horner_droite,prop:horner_gauche},
if $\gcd(p_r(x),\sigma (x))=1$, then the fraction 
\eqref{eq:paire_horner_gauche} obtained from a left Horner evaluation
is irreducible. 
\end{proposition}

We first prove the following result, which gives part of the case~$r=1$.
\begin{lemma}\label{lemma:partial-P}
Let $p\in\K[x] \setminus \{ 0 \}$ be relatively prime with~$\sigma$ and $
(P,S_n^{d_p})=\HorrN(p)$ with $d_p=\deg p$. Then the fraction
\[
  \mathfrak R^\psi_{d/dx}(\partial p)=(S_n,\Dden)\otimes (P,S_n^
  {d_p})
  = (S_n^{(1-d_p)^+}P,S_n^{(d_p-1)^+}\Dden)
\]
is irreducible.
\end{lemma}
\begin{proof} If $p \in \K$,  the result is obvious. For $p\in\K
[x]\setminus \K$, the proof is indirect. 
 Bézout's identity gives the existence of $u$ and $v \in \K [x]$ such
that $1 = pu + \sigma v$, with $d_u:=\deg u<\deg\sigma$ and $d_v:=\deg v<d_p$. First we note that if $u=0$, then $v\sigma=1$ implies that $\deg\sigma=0$, which happens only for
Hermite polynomials, in which case $\Dden$ has degree~0 and the result
holds by \zcref{coro:horner_elim_Sn}. Also, $v \neq 0$, otherwise $p \in \K$. 
From now on, those two cases are excluded and $d_u, d_v \ge0, d_p \geq 1$.

Multiplying the Bézout identity by $\partial$ on the left gives
$ \partial(p u)   = \partial - ( \sigma \partial +
\sigma') v$.
By \zcref{lemma:equivalent-pairs}, this implies that the images by
$\mathfrak R^\psi_{d/dx}$ of differential terms for both sides are
$\RecI$-equivalent. We use the terms as parenthesized above, 
where the polynomials $p,u,\sigma,\sigma',v$ are presented in a
right Horner form. If~$(U,S_n^{d_u})=\HorrN(u)$, the inner product in
the left-hand side is mapped to~$(P(n+d_u,S_n)U,S_n^
{d_p+d_u})$. For the right-hand side, up to $\RecI$-equivalence, the
inner term $\sigma\partial$ can be mapped to~$(\Theta_\text{num},S_n)$
from \zcref{eq:theta_rec} below. Adding~$\mathfrak R^\psi_{d/dx}
(\sigma')$
does not change the denominator since $\deg\sigma'\le1$ (see  
\zcref{table:classiques2}). Finally, multiplying on the right by
$\mathfrak R^\psi_{d/dx}(v)$ multiplies the denominator by $S_n^{d_v}$. In summary, we have obtained that for some $A\in\RecN$,
\[
  (S_n,\Dden)\otimes (P(n+d_u,S_n)U,S_n^
{d_p+d_u}) \equiv_{\RecI} (S_n,\Dden)\oplus (A,S_n^{d_v+1}).
\]
A direct computation on both sides gives
\begin{multline*}
(S_n^{1-d_p-d_u}P(n+d_u,S_n)U,S_n^{d_p+d_u-1}\Dden)
\equiv_{\RecI} \\
(S_n^{d_v+2}+\Dden(n+d_v+1,S_n)A,S_n^
{d_v+1}\Dden).
\end{multline*}
Let $\delta=d_v+1-(d_p+d_u-1)$ be the difference between
exponents of $S_n$ in the denominators of the left- and the right-hand
sides of this equivalence. 
 Recall that $d_p>d_v\ge0$, $\deg(pu)=\deg(v\sigma)$
and 
\[
  \delta=d_v+1-(d_p+d_u-1)=2-\deg\sigma\ge0.
\]
The $\Rec$-equivalence of fractions then implies that there is an 
\emph{equality} between numerators that can be written
\[
  S_n^{\delta}P(n+d_u,S_n)U = S_n^{2+d_v}+\Dden(n+d_v+1,S_n)A.
\]
Since $\Dden$ is relatively prime with $S_n$ and $\delta \leq 2 + d_v$, this implies that there is~$\tilde A\in\Sh$ such that $A=S_n^{\delta}\tilde A$ leading to 
\[
  S_n^{d_p+d_u} = P(n+d_u,S_n)U-\Dden(n+d_p+d_u-1,S_n)\tilde A.
\]
This Bézout identity shows that $\gcld_\Sh(P(n+d_u,S_n),\Dden(n+d_p+d_u-1,S_n))$ is a power of~$S_n$, which has to be~$S_n^0=1$ since $(P,S_n^{d_p})$ is irreducible. Thus
\begin{multline*}
1 = \gcld_\Sh(P(n+d_u,S_n),\Dden(n+d_p+d_u-1,S_n)) =
\gcld_\Sh(P,\Dden(n+d_p-1,S_n))\\
=\gcld_\Sh(P,\Dden(n+d_p-1,S_n)S_n^{d_p-1})=
\gcld_\Sh(P,S_n^{d_p-1}\Dden),
\end{multline*}
as was to be proved.
\end{proof}

\begin{proof}[Proof of \zcref{prop:Lewanowicz}]
By \zcref{prop:irred-frac,prop:assoc-equiv}, the
module $\assocal$ contains an
irreducible fraction and it is sufficient to prove that the
fraction in the proposition is quasi-generating in~$\assocan$.

The proof is by induction on the order~$r$. For $r=0$, the result is
given by \zcref{prop:horr-pol-irred}.  Now, we assume that the result
is satisfied for all orders from $0$ to $r-1$. Let $(\Anum,\Aden) \in 
\assocan$, i.e., 
\begin{equation}\label{eq:Anum-Aden}
\Anum\cdot[\psi_n](\psi_j)=\Aden\cdot[\psi_n]
(L\cdot\psi_j),\qquad\text{for all }j\ge0.
\end{equation}
Again by \zcref{prop:assoc-equiv}, it is then sufficient to show that
$\Aden$ is divisible by the denominator $S_n^{m(L)} \Ddenr$ of $\Horl 
(L)$, where we write $m(L)$ for the exponent~$m$ of \zcref{prop:horner_gauche}.

We write $L$ in two different ways:
\[L=L_1\partial+p_0=\partial^rq_r+L_2,\]
where~$L_1$ and~$L_2$ are linear differential operators of orders
$r_1=r-1$ and $r_2\le r-1$ and moreover, the exponents from 
\zcref{prop:horner_gauche} satisfy $m(L_2)\le m(L)$.
Applying the induction hypothesis to~$L_1$, shows that any
element of~$
\langle{A}_{\text{an}}(L_1)\rangle$ has a denominator that is a
left multiple of~$S_n^{m(L_1)}\Ddenrmu$.

First, observe that the denominator $\Aden$ can be written $S_n^
{v_0}\Aden'$ with $v_0\geq m(L)$ and $S_n
\nmid\Aden'$. Indeed, if $(\mathcal{A},\mathcal{B})$ is a
quasi-generating pair
of~$\assocal$ (its existence is granted by \zcref{prop:assoc-equiv}),
then there exists $c\in \K [n]_\Za$ and  $\mathcal{C} \in \Rec$ such
that $c(n) \Horl (L) = (
\mathcal{C} \mathcal{A}, \mathcal{C} \mathcal{B})$. 
\zcref[S]{coro:horner_elim_Sn} implies that $S_n \nmid \mathcal{C}$
and therefore~$S_n^{m(L)} \mid \mathcal{B}$.

Next, applying \zcref{prop:annihilators} twice,
\begin{gather*}
(\Anum,\Aden) \oplus \mathfrak R^\psi_{d/dx}(-p_0) \in   \Ann_\Frec
  \left([\psi_n](\psi_j), - [\psi_n] ({L_1}  \partial \cdot \psi_j)
  \right),\\
  \left( (\Anum,\Aden) \oplus \mathfrak R^\psi_{d/dx}(-p_0) \right)
  \otimes
  (\Dden,S_n) \in   \Ann_\Frec \left  ([\psi_n](\psi_j'), - 
  [\psi_n] \left ( L_1  \cdot \psi_j' \right ) \right ).
\end{gather*}
Since this holds for any~$j$ and the $\psi_j$ form a basis of the
polynomial ring, their derivatives $\psi_j'$ generate it and therefore
the fraction
on the left belongs to~$\langle A_{\text{an}}(L_1)\rangle$. 

Thus, by the induction hypothesis,
the denominator of that fraction, which can be
computed directly to be $S_n^{v_1}A'_\text{den}$ with~$v_1>v_0$,
is a left multiple of~$S_n^{m(L_1)}\Ddenrmu$.
This implies the existence of~$F\in\Sh$ such that $S_n\nmid F$
and
\begin{equation} \label{eq:A'den}
 \Aden = S_n^{v_0}   F \Ddenrmu.
\end{equation}

Next, we turn to the second rewriting of~$L$. Expanding 
\zcref{eq:Anum-Aden} and using $\Horl(\partial^{r-1})=(S_n^
{r-1},\Ddenrmu)$ from \zcref{prop:horner_gauche} gives that for all
$j\ge0$,
\begin{align*}
\Anum\cdot[\psi_n](\psi_j)&=
\Aden\cdot[\psi_n](\partial^{r-1}\partial q_r\cdot \psi_j)
+\Aden\cdot[\psi_n](L_2\cdot \psi_j)\\
&=S_n^{v_0}F\Ddenrmu\cdot[\psi_n](\partial^{r-1}\partial q_r\cdot
\psi_j)
+\Aden\cdot[\psi_n](L_2\cdot \psi_j)\\
&=S_n^{v_0}FS_n^{r-1}\cdot[\psi_n](\partial q_r\cdot\psi_j)+\Aden\cdot
[\psi_n]
(L_2\cdot \psi_j).
\end{align*}

Now, from 
\[\Horl(L_2)=(\Hlnum(L_2),S_n^{m(L_2)}\Ddenrt)\]
with $m(L_2)\le m(L)\le v_0$ and $r_2\le r-1$ and \zcref{eq:A'den}, it
follows that $\Aden$ is a left multiple of this denominator, i.e.,
there exists~$G$ such that
\[\Aden=GS_n^{m(L_2)}\Ddenrt.\]
As a consequence, for all~$j\ge0$,
\begin{gather*}
\Anum\cdot[\psi_n](\psi_j)=
S_n^{v_0}FS_n^{r-1}\cdot[\psi_n](\partial q_r\cdot\psi_j)+
GS_n^{m(L_2)}\Ddenrt\cdot[\psi_n](L_2\cdot \psi_j),\\
(\Anum-G\Hlnum(L_2))\cdot[\psi_n](\psi_j)=
S_n^{v_0} FS_n^{r-1}\cdot[\psi_n](\partial q_r\cdot\psi_j).
\end{gather*}
\zcref[S]{lemma:partial-P} then implies that $\Dden$ is a right
divisor of $FS_n^{r-1}$ and thus that $\Dden(n-r+1,S_n)$ is a right
divisor of~$F$, which concludes the proof that $\Ddenr$ is a right
divisor of $\Aden'$ and also that $S_n^{m(L)}\Ddenr$ is a right
divisor of $\Aden$.
\end{proof}

\subsection{Algorithm for irreducible fractions (Proof
of Theorem \ref{thm:horner})}
\label{sec:irred_gen_pair}

\begin{algorithm}[ht]
    \caption{Main algorithm} \label{algo:main}
          \Input{$L=\sum_{i=0}^r p_i(x)\partial^i \in \K[x]
 \langle \partial \rangle$;\quad $\sigma(x)$
          from \zcref{eq:Pearson};\\
     \qquad the pairs $X^\psi=(\Xnum,\Xden)$, $D^\psi=(\Dnum,\Dden)$
     from
    \zcref{eq:Phi_psi}.}
        \Output{An irreducible fraction $P\in\Frec$ as in
 \zcref{thm:horner}} 
        \lIf{$\gcd(p_r(x),\sigma(x))=1$}
        {\KwRet \zcref[noname,font=\textsf]{algo:LeftHorner}$(L,X,D)$}
        $(A,B)\gets$ \zcref[noname,font=\textsf]{algo:RightHorner}$
        (L,X,D)$\\
        $\ell \gets \min(\valSn A,\valSn B)$\\
        \lIf{$\ell=0$}{\KwRet{$(A,B)$}}
        \tcc{\zcref{eq:qAB} with
        $q(n)$ minimal:}
        $(A,B)\gets q(n)S_n^{-\ell}(A,B)$ \tcp{Exact division
        by~$S_n^\ell$}
        $(P,Q)\gets \Horl(\partial^r,X,D)$ \tcp{{Or use
        the explicit \eqref{lem:tech_paires} for $\Ddenr$}}
        $G\gets \tilde q(n)((S_n^{\valSn B}Q)\div B)$ \tcp{{Exact right
        division by $B$, with $\tilde q(n)$ minimal for $G\in\Rec$}}
        $UG+VS_n^\ell=1$ \tcp{{Extended left gcd in $\Sh$}}
        $c(n):=\operatorname{lcm}(\operatorname{denominators}(U,V))$\\
        \KwRet{$c(n)(A,B)$}
\end{algorithm}

Let~$L$ be a linear differential operator in~$\K[x]\langle
\partial\rangle$
. \zcref[S,noabbrev]{algo:main} follows the
steps of the proof of \zcref{prop:irred-frac}, with extra steps
to reduce the computational cost. First,
in the situation where \zcref{prop:Lewanowicz} applies, the
computation is performed by \zcref{algo:LeftHorner}, that returns an irreducible fraction, see \zcref{sec:faster}, without needing expensive computations of cofactors in its operations. 

Otherwise, the algorithm first computes a right Horner evaluation. By
\zcref{prop:horner_droite}, it is of the form~$(S_n^\ell
A,S_n^\ell B)$ with $(A,B)$ an irreducible fraction. If~$\ell=0$, the pair
is irreducible and returned. If not, the proof of \zcref{prop:irred-frac}
combines~$(A,B)$ with the fraction $\Horl(L)$. 
The algorithm avoids the computation of $\Horl(L)$ for efficiency
reasons. Instead, only $\Ddenr$ is computed. Since~$\hat G$ is
relatively prime to~$S_n$, \zcref{prop:horner_gauche} shows
that~$S_n^m\Ddenr$ is equal to $\hat G D=\hat GB$ for $m=\valSn B$, up
to an element of $\K[n]_\Zn$ on the left. Thus $\hat G$ can be computed by dividing $S_n^{m}\Ddenr$ on the right  by~$B$ and multiplying by a polynomial so that the coefficients belong to~$\K[n]_\Zn$. This allows us to compute $U,V,c(n)$ and return $c(n)(A,B)$. This is done in the last lines of Algorithm~\ref{algo:main} and completes the proof of correctness of the algorithm. \qed

\begin{remark}
In all our examples,
the end result is~$c(n)=1$ although we have not been able to find a
proof of this observation in general. 
\end{remark}

\section{Algorithms for singular cases}\label{sec:singular-case}

It may happen that \zcref{eq:hypH} from \zcref{thm:horner}
does not hold: the required sequences of generalized Fourier
coefficients do not
 exist. Then the fraction returned by \zcref[S]
{algo:main} may still have a numerator that cancels the
generalized Fourier coefficients, in which case another proof is
needed, or this recurrence can simply fail to hold and another way of
computing a recurrence is needed.

This section is devoted to two extensions of \zcref{thm:horner} to
those
cases whose problems originate in the behaviour of the
function~$f$ at the endpoints of the support of the measure used to
define the family of orthogonal polynomials. The solution consists in using
a
different derivation than~$d/dx$, that incorporates some of the
behavior at the singularities. In this section, we do not consider
Hermite
polynomials since their $\sigma$ function is $1$. 

\begin{example}\label{ex:first-sing}
The function $(1-x^2)^{-1/4}$ is annihilated by~$2
(1-x^2)\partial-x\in\K[x]\langle\partial\rangle$. Direct
application of the
algorithm gives 
\begin{Verbatim}
> FunctiontoRec((1-x^2)^(-1/4),ChebyshevT(n,x),c);
\end{Verbatim}
{\color{blue}
\begin{equation}\label{eq:coeffs-ex-5}
  (2n+3)c_{n+2}-(2n+1)c_n=0.
\end{equation}}
The actual values of the Chebyshev coefficients~$(c_n)$ can be computed
by standard properties of the Beta
integrals and turn out to be
\[
  c_n=\begin{cases} 
    0 &\text{if $n$ is odd,} \\
    \frac{\sqrt{2\pi}}{\Gamma(3/4)^2}&\text{if $n=0$,}\\
    \frac{2\Gamma\left(\frac{n}{2}+\frac{1}{4}\right)}{\sqrt{\pi}\Gamma\left(\frac{n}{2}+\frac{3}{4}\right)} & \text{otherwise.}
      \end{cases}
    \]
Thus, with the convention of 
\zcref{sec:note_on_chebyshev_series}, the recurrence is seen to be
satisfied by the coefficients. However, the first derivative of
this function does not have a Chebyshev expansion, so that 
\zcref{eq:hypH} does not hold. A proved derivation of this
recurrence
is given in \zcref{ex-5} below.
\end{example}

\begin{example}\label{ex:arccos}
The function $\arccos x$ shows that
convergence hypotheses are necessary. This function is annihilated by
$L=
(1-x^2)\partial^2-x\partial\in\K[x]\langle\partial\rangle$. In
the basis of Chebyshev
polynomials $T_n(x)$, \zcref{eq:hypH} does not hold,
since the derivative of $\arccos$ does not have Chebyshev
coefficients. Direct application of the algorithm gives the
fraction $(-n^2,1)$. This would suggest that the generalized
Fourier coefficients satisfy $n^2c_n=0$, i.e, $c_n=0$ for all~$n>0$,
but it is not the case. A correct recurrence is given in \zcref{ex:arccos2} below.
\end{example}

\subsection{Combining leading coefficient and derivative}
The following is an analogue of
\zcref{lem:Dcop}; it uses the polynomial $\sigma$ from 
\zcref{table:classiques2}. 
\begin{lemma}\label{lem:sigmadx}
  Let $f$ be such that both sequences $([\psi_n](f))_{n \in \nat}$ and $([\psi_n](\sigma f'))_{n \in \nat}$ exist.  
 Then, their Fourier coefficients are related by 
\begin{equation}\label{eq:theta_rec}
(\Thnum,\Thden)\in\Ann_\Frec([\psi_n]
(f),-[\psi_n](\sigma f')),
\end{equation}
where
\[
  (\Thnum,\Thden)=  (-\lambda_{n+1} \Dden-a_1\Xnum-a_0S_n,S_n),
\]
with $\lambda_n,\tau(x)=a_1x+a_0$ from \zcref{table:classiques2}.
\end{lemma}
\begin{table}
   \begin{center}
     \renewcommand{\arraystretch}{1.5}
\begin{tabularx}{.9\textwidth}{p{1.7cm} | c | Y     }
   Name                  & $\psi_n$      & $\Thnum$  \\\hline
Chebyshev &
    $T_n$   &
    \cellscript{\left(\tfrac{n}{2}+1\right)S_n^2-\tfrac{n}{2}} \\
    \hline 
Gegenbauer &
    $C_n^{(\lambda)}$ & 
    \cellscript{\tfrac{(n+1+2\lambda)(n+2+2\lambda)}{2(n+2+\lambda)} 
    S_n^2 - \tfrac{n(n+1)}{ 2 (n+\lambda)}}  \\
    \hline
Jacobi & 
    $P_n^{(\alpha, \beta)}$  &
    \cellscript{\tfrac{2(n + 2 + \alpha)(n + 2 + \beta)(n + 3 +
    \alpha + \beta)} {( 2n + 4 + \alpha + \beta)_2 } S_n^2\\ 
    + \tfrac{ 2(\alpha - \beta)(n + 1)(n + 2 + \alpha + \beta)}{(2n +
    2 + \alpha + \beta)(2n + 4 + \alpha + \beta)}  S_n - \tfrac{2n(n + 1)(n  + 1 + \alpha + \beta)}{(2n + 1 + \alpha + \beta)_2}} \\
    \hline
Laguerre  &
    $L_n^{(\alpha)}$ &
    \cellscript{- (n+1+\alpha) S_n +   n}         
\end{tabularx}
  \caption{Recurrence operators $\Thnum$ associated to
  classical orthogonal polynomials with non-constant $\sigma$ 
  ($\Thden=1$ for 
  Laguerre, $\Thden=S_n$ otherwise).}
  \label{table:thetap_rec_cop}
  \end{center}
\end{table}
\noindent The values of $(\Theta_{\text{num}},\Theta_{\text{den}})$ for
the classical orthogonal polynomials are given in 
\zcref{table:thetap_rec_cop}. They follow from \zcref{lem:sigmadx},
except in the case of
Laguerre polynomials, where the value is given by 
\zcref{lem:sigmadxlag}
below. In
all cases, it can be seen that $\Theta_{\text{num}}$ belongs
to~$\Rec$: 
in the case of Jacobi polynomials  when $\alpha + \beta + 1 = 0$,
the constant coefficient wrt~$S_n$ simplifies to~$- {n(n+1)}/
({2n+1})$. 
\begin{proof} Here again, part of the proof was given by Lewanowicz~%
\cite{Lewanowicz1995} and we write it for the sake of readability and,
again, to make the required convergence assumptions explicit.

 A direct computation shows that
\[\mathfrak R^\psi_{\delta}((a_1\times x)+a_0)=
(a_1\Xnum+a_0S_n,S_n),\]
independently of the derivation~$\delta$. This implies that
\[(a_1\Xnum+a_0S_n)\cdot[\psi_n](f)=[\psi_{n+1}](\tau f).\]

The proof then follows from the sequence of equalities
\begin{align*}
&\lambda_{n+1}\Dden\cdot [\psi_n](f)+[\psi_{n+1}](\sigma f')+[\psi_
{n+1}](\tau f)\\
&\quad =\frac1{h_{n+1}}\langle\sigma\psi_{n+1}'|f\rangle
+\frac 1{h_{n+1}}\langle\sigma f'|\psi_{n+1}\rangle+\frac1{h_
{n+1}}\langle\tau f|\psi_{n+1}\rangle\\
&\quad=\frac1{h_{n+1}}\int_a^b{(\sigma w) (\psi_{n+1}f)'+(\tau w)\psi_
{n+1}f\,dx}\\
&\quad=\frac1{h_{n+1}}\int_a^b{(\sigma w) (\psi_{n+1}f)'+(\sigma
w)'(\psi_
{n+1}f)\,dx}\\
&\quad=\frac1{h_{n+1}}[\sigma w \psi_{n+1}f]_a^b=0,
\end{align*}
which use, in this order, the definition of $\Dden$ and the structure
relation~\eqref{eq:structure}, the integral definition of Fourier
coefficients, Pearson's equation~\eqref{eq:Pearson}, and the zero
limit from the proof of \zcref{lemma:int-fprime}.
\end{proof}

A direct application of \zcref{lem:sigmadx} to the
Laguerre polynomials
 gives the fraction
\begin{equation}\label{eq:Laguerre1}
((-(n+2+\alpha)S_n^2 + (n +
1)S_n, S_n) \in \Ann_\Frec (  [L_n^{(\alpha)}](f), - [L_n^{(\alpha)}]
(x f') ),
\end{equation}
which is
not irreducible. 
A simpler form is given by the following.

\begin{lemma} \label{lem:sigmadxlag} 
  Let $f$ be such that both sequences $([L_n^{(\alpha)}](f))_{n \in
 \nat}$ and $([L_n^{(\alpha)}](x f'))_{n \in \nat}$ exist. 
   Then, \zcref{eq:theta_rec} also holds with
\[(\Thnum,\Thden)=(-(n+1+\alpha)S_n+n,1).\]
\end{lemma}
\begin{proof} From $x f' = (x f)' - f$, it follows that $([\psi_n]((x f)'))_{n \in \nat}$ exists.
\zcref{prop:rec_Fourier} then applies to give
\[\mathfrak R^\psi_{d/dx}((\partial \times x)-1)\in\Ann_\Rec([L_n^{
(\alpha)}]
(f),  -  [L_n^{(\alpha)}]( xf' )).\]
A direct computation shows that this pair is
\[(   -(n+2+\alpha) S_n^2 +  (2n+2+\alpha)S_n - n, S_n - 1).\]
Subtracting this pair from that of \eqref{eq:Laguerre1} gives that
their difference, which is
the pair in the statement,
belongs to the module $\Ann_\Rec([L_n^{
(\alpha)}]
(f),  -  [L_n^{(\alpha)}]( xf' ))$.
\end{proof}

\subsection{Singularities at endpoints}\label{sub:singularities_at_endpoints}

\zcref[noabbrev]{lem:sigmadx} is a direct analogue of \zcref{lem:Dcop}
for the derivation $\sigma d/dx$. This allows us to obtain an analogue
of \zcref{prop:rec_Fourier} for linear differential operators using
this derivation. For
clarity, we use differential terms in a new variable~$\theta$ instead
of~$\partial$ and
the map~$\mathfrak D_{\sigma d/dx}$ maps~$\theta$ to~$\sigma\partial$.

Next, the morphism~$\mathfrak R^\psi_{\sigma d/dx}$
mapping~$\mathcal D$ to~$\Frec$ is 
defined like~$\mathfrak R^\psi_{d/dx}$ in 
\zcref{def:fourier-morphism}, except that the image
of~$\theta$
is
$ (\Theta_{\text{num}},\Theta_{\text{den}})$ from \zcref{table:thetap_rec_cop}.
This leads to an extension of \zcref{thm:horner}.

\begin{theorem}[Recurrences in
singular cases] \label{thm:main}
Let $(\psi_n)$ be a family of classical orthogonal polynomials with
$\sigma\neq1$ from \zcref{table:classiques2}.
Let $L = p_r(x)\partial^r+\dots+p_0(x)\in\K [x]\langle\partial\rangle$ be a linear differential operator with polynomial coefficients such that
\begin{equation}\zcsetup{reftype=hypothesis}
 \label{eq:hypH'1}\tag{H${}_1'$}
  \sigma^k\divides p_k   \textrm{ for }  k =0,\dots,r.
 \end{equation}
Then $L$ can be written $q_r(x)(\sigma\partial)^r+\dots+q_0(x)$ and $q_r,\dots,q_0 \in \K[x]$.
 Let $(\Pnum,\Pden)\in\Frec$ be  the
 fraction
 \[\Hortr (L) := \mathfrak R^\psi_{\sigma d/dx}\!\left((\dotsb(q_r
  (x)\theta +
  q_
  {r-1} (x)) \theta + \dotsb + q_1 (x)) \theta + q_0 (x) \right).\]
Then $(\Pnum,\Pden)$ is a
quasi-generating fraction of~$\langle A_{\text{al}}(L)\rangle$
and~$\langle
A_{\text{an}}(L)\rangle$.

If $f$ is a function such that
\begin{equation}\zcsetup{reftype=hypothesis}
 \label{eq:hypH'2}\tag{H${}_2'$}
   \text{the sequence }([\psi_n](\sigma^k f^{(k)}))_{n\in\mathbb N} 
   \text{ exists for }k=0,\dots,r,
 \end{equation}
then the sequence $([\psi_n](L\cdot f))_{n\in\N}$ exists and satisfies
\[\Pnum\cdot [\psi_n](f)=\Pden\cdot [\psi_n](L\cdot f).\]
In particular, if $f$ satisfies \zcref{eq:hypH'1,eq:hypH'2}
and
$L\cdot f=0$,
then its sequence of 
Fourier coefficients satisfies $\Pnum\cdot [\psi_n](f)=0$.
\end{theorem}
\begin{proof}
The rewriting of~$L$ under \zcref{eq:hypH'1} is by induction as in
\zcref{lem:conversion}. That the pair $(\Pnum,\Pden)$ belongs
to~$\langle A_{\text{al}}(L)\rangle$
and~$\langle A_{\text{an}}(L)\rangle$ is a direct consequence of
\zcref{prop:rec_Fourier}, as is the conclusion of the \zcref
[noabbrev,noref,nocap]
{thm:main}. The only remaining statement to be proved is that $
(\Pnum,\Pden)$ is a 
quasi-generating fraction of these modules. By 
\zcref{prop:assoc-equiv}, it is sufficient to
prove
that this fraction is irreducible.

The proof distinguishes two cases depending on whether~$\Theta_
{\text{den}}=1$ or~$S_n$ (Laguerre polynomials, or the other classical
ones).
In the first case, a simple recurrence shows that  $\Hortr(L)$ 
 is $\Zn$-equivalent to 
\[
  \left (\sum_{k=0}^r  S_n^{m - \deg q_k} Q_k \Thnum^k, S_n^m \right )
\]
with $m =  \max_{k = 0, \ldots, r} \deg q_k$. In the second one, a
similar computation shows that $\Hortr(L)$ 
 is $\Zn$-equivalent to \[
  \left ( \sum_{k=0}^r  S_n^{m - \deg q_k - k} Q_k(n+k,S_n)
  \prod_{j=0}^{k-1} \Thnum(n+k - 1 - j,S_n), S_n^m \right )
\]
with $m = \max_{k = 0, \ldots, r} \deg q_k + k$. 

In both cases, in order to show irreducibility, it is sufficient to
show that the
constant coefficient of the numerator wrt~$S_n$ is
nonzero. Let~$k$ be such that the exponent of~$S_n$ in the sum is~0
(the other summands do not contribute to the constant term). In
view of \zcref{table:thetap_rec_cop},  a
direct computation shows that the degree of that summand with respect
to~$n$ is exactly $k+m$ for the Laguerre case and $k$ in the other
ones. This implies that the degree of the sum is that of the term 
with largest value of~$k$ and therefore is not~0. Hence, the constant
term of the numerator wrt~$S_n$ is nonzero and  $\Hortr(L)$ is
irreducible.
\end{proof}

Note that \zcref{eq:hypH'1,eq:hypH'2} make it
possible to
deal with some functions that are singular at $\pm1$ for Jacobi
polynomials or at $0$ for Laguerre polynomials, but whose singularity
is not ``too bad'': it is a regular singular point~\cite{Ince1956}.
\begin{example}\label{ex-5}
The function $(1-x^2)^{-1/4}$ is annihilated by~$2
(1-x^2)\partial-x\in\K[x]\langle\partial\rangle$. Recall from 
\zcref{ex:first-sing} that in the
Chebyshev basis $T_n
(x)$,
\zcref{eq:hypH} does not hold. However it is easy to see that
\zcref{eq:hypH'1,eq:hypH'2} do. Since
$\mathfrak R^{T_n}_{(1-x^2)d/dx}(\theta)$ and $\mathfrak R^{T_n}_
{(1-x^2)d/d}(x)$
have
the same denominator~$S_n$, application of the morphism reduces to
adding the numerators, giving
\[\Pnum=\left(n+\frac32\right)S_n^2-\frac{2n+1}2,\quad\Pden=S_n.\] 
The theorem then proves that the numerator provides the desired
recurrence for the Chebyshev coefficients.
\end{example}
\begin{example}\label{ex:arccos2} In the case of the function
$\arccos
x$, \zcref{eq:hypH'2} holds, but 
none of \zcref{eq:hypH,eq:hypH'1} do. However, left multiplying $L$ by
$(1-x^2)$
gives the operator $\theta^2+x\theta$, which is such that 
\zcref{eq:hypH'1,eq:hypH'2}
both hold. Then the theorem proves that the coefficients satisfy
\begin{Verbatim}
> DiffeqtoRec((1-x^2)*((1-x^2)*diff(y(x),x,x)-x*diff(y(x),x)),
    y(x),c(n),ChebyshevT(n,x));
\end{Verbatim}
{\color{blue}\begin{equation}\label{eq arccos}
(n+4)^2c_{n+4}-2(n+2)^2c_{n+2}+n^2c_n=0.
\end{equation}}
This can again be checked against the known coefficients (modulo
\zcref{sec:note_on_chebyshev_series} again):
\[c_n=\begin{cases} 
    \pi/2&\text{if $n=0$},\\
    0 &\text{if $n>0$ is even,} \\
    -\frac{4}{n^2\pi} & \text{otherwise.}
\end{cases}\]
\end{example}

\subsection{Singularity at one endpoint in Jacobi and special cases}
The reasoning above extends to the case of solutions of linear differential
equations that are singular at only one of the endpoints of the
support $(a,b)$ of the measure. The following is inspired by a result
of Lewanowicz~\cite{Lewanowicz1986}.

Exactly as before, we set $\tau_{\varepsilon} = 1 + \varepsilon x$
with $\varepsilon=\pm1$, we consider the term algebra
with symbols~$x,\theta_\varepsilon$. The morphism
$\mathfrak D_{\tau_\varepsilon d/dx}$ maps
$\theta_\varepsilon$ to $\tau_\varepsilon\partial$, while 
$\mathfrak R^\psi_{\tau_\varepsilon d/dx}$ maps
$\theta_\varepsilon$
to $(\Theta^{\psi_n,\epsilon}_{\text{num}},
\Theta^{\psi_n,\epsilon}_{\text{den}})$ from \zcref{lem:1+epsxdx}
below.
This leads to a right Horner scheme and the following analogue of 
\zcref{thm:horner,thm:main}, without the irreducibility of the
fraction.

\begin{proposition}\label{prop:main_jacobi}
 Let $(\psi_n)$ be the family of Jacobi, Gegenbauer or Chebyshev
 polynomials and let $\varepsilon\in\{-1,1\}$ and $\tau_{\varepsilon}
 = 1 + \varepsilon x$.  
 Let $L=p_0(x)+\dots+p_r(x)\partial^r\in\K
 [x]\langle\partial\rangle$ be a linear differential
 operator with polynomial coefficients such that
\begin{equation}\zcsetup{reftype=hypothesis}
\label{eq:hypH''1}
\tag{H${}_1''$} 
\tau_{\varepsilon}^k\divides p_k 
\quad \text{for}\quad  k =0,\dots,r.
\end{equation}
Then $L$ can be written $q_r(x)(\tau_\epsilon\partial)^r+\dots+q_0(x)$
with
$q_i\in\K[x]$. Let $(\Pnum,\Pden)\in\Frec$ be the fraction
\[(\Pnum,\Pden):=\mathfrak R^\psi_{\tau_\varepsilon d/dx}    
\!\left((\dotsb(q_r
  (x)\theta_\varepsilon +
  q_
  {r-1} (x)) \theta_\varepsilon + \dotsb + q_1 (x)) \theta_\varepsilon
  + q_0 (x) \right).
\]
If $f$ is a
 function such that
\begin{equation}\zcsetup{reftype=hypothesis}
\label{eq:hypH''2}
\tag{H${}_2''$} 
\text{the sequence }([\psi_n](\tau_{\varepsilon}^k  f^{(k)}))_
{n\in\nat}\text{ exists for } k =0,\dots,r,
\end{equation}
then the sequence $ ([\psi_n] (L\cdot f))_ {n\in\nat}$
exists and satisfies
\[\Pnum\cdot[\psi_n](f)=\Pden\cdot[\psi_n](L\cdot f).\]
In particular, if $f$ satisfies \zcref{eq:hypH''1,eq:hypH''2} and $L\cdot f=0$,
then its sequence of Fourier coefficients satisfies $\Pnum\cdot
[\psi_n](f)=0$.
\end{proposition}
At this stage, the only part of this result that is still unproved is
the following.

\begin{lemma}\label{lem:1+epsxdx}
With the same notation, if $f$ is such that both
sequences $([\psi_n](f))_{n \in
\nat}$ and $([\psi_n]((1+\varepsilon x) f'))_{n \in \nat}$ exist, then
\begin{equation}\label{eq:Jac1pe}
(P_{\text{num},\varepsilon},P_{\text{den},\varepsilon})
\in 
\Ann_\Frec([\psi_n(f)],-[\psi_n]((1+\varepsilon x)f')),
\end{equation}
where this fraction is: $(\Jnum,\Jden)$ in the Jacobi case, 
$(G_{\text{num},\varepsilon},G_{\text{den},\varepsilon})$ in the
Gegenbauer case and
$(T_{\text{num},\varepsilon},T_{\text{den},\varepsilon})$ in the
Chebyshev case, 
with
\begin{align*}
    \Jnum & =    \left (   (\alpha+\beta+n+1) S_n 
 + \left (\frac{1 + \varepsilon}{2} \alpha  -  \frac{1-\varepsilon}{2} \beta  + \varepsilon (n+1) \right ) \right ) \lambda_nh_n, \\
    \Jden & =  (\alpha+\beta+n+1) \left (    \varepsilon (n+1)
      S_n -  \left (\frac{1 + \varepsilon}{2} \alpha  +  
      \frac{1-\varepsilon}{2} \beta  + n + 1 \right )   \right )h_n,\\
 G_{\text{num},\varepsilon}&=
    (2 \lambda + n + 1)(\lambda + n)S_n + \varepsilon n (\lambda + n +
    1),\\
G_{\text{den},\varepsilon}&=  \varepsilon 
(\lambda + n)S_n +(\lambda + n + 2)(\lambda + n + 1),\\
T_{\text{num},\varepsilon}&=
\varepsilon
(n+1)S_n+ n,\\
T_{\text{den},\varepsilon}&=
- \varepsilon S_n + 1
\end{align*}
and $\lambda_n,h_n$ from \zcref{table:classiques,table:classiques2}. 
\end{lemma}
\begin{proof} 
We detail the Jacobi case; the other ones are similar. With the
notation of the lemma, the following identity \cite [p. 262,
Eqs. (3),(6)]{Rainville1971} is easily
verified:
\begin{equation}\label{eq:rainville}
\Jnum\cdot \frac{1-\varepsilon x}{\lambda_n}{P_n^{(\alpha,\beta)}}'= - \Jden\cdot P_n^{(\alpha,\beta)}.
\end{equation}
It implies  
\begin{align*}
\Jden\cdot [P_n^{(\alpha,\beta)}]\! \left ( (1+\varepsilon x) f'
\right) 
&=\Jden\cdot \left\langle (1+\varepsilon x) f' \big | P_n^{(\alpha,\beta)}\right\rangle_{\!w}\\
&=\left\langle (1+\varepsilon x)  f' \big |\Jden\cdot
 P_n^{(\alpha,\beta)}\right\rangle_{\!w}\\
&= - \left\langle (1+\varepsilon x)  f' \big |\Jnum\cdot 
\frac{1-\varepsilon x}
{\lambda_n}{P_n^{(\alpha,\beta)}}'\right\rangle_{\!w}\\
&= -\Jnum\cdot\frac{1}{\lambda_n} \left\langle\sigma f' \big |{P_n^{(\alpha,\beta)}}'
\right\rangle_{\! w}\\
  &=\Jnum\cdot\left \langle P_n^{(\alpha,\beta)} \big | f\right\rangle_w
    = \Jnum\cdot [P_n^{(\alpha,\beta)}]f,
\end{align*}
where the last step follows from \zcref{lemma:int-fprime}.
\end{proof}

\section{Previous algorithms}\label{sec:previous-algo}

We now discuss relations between this work and alternative
algorithms by 
Area, Godoy, Ronveaux and Zarzo~%
\cite{RonveauxZarzoGodoy1995,GodoyRonveauxZarzoArea1997,ZarzoAreaGodoyRonveaux1997,AreaGodoyRonveauxZarzo1998,GodoyRonveauxZarzoArea1998}, Rebillard and Zakraj\v{s}ek~\cite{Rebillard1997,Rebillard1998,RebillardZakrajsek2007} and Lewanowicz~\cite{Lewanowicz2002}.

\subsection{Navima algorithm}

The Navima approach was used in the context of connection problems,
linearization problems for various types of orthogonal polynomial
families~%
\cite{RonveauxZarzoGodoy1995,GodoyRonveauxZarzoArea1998,GodoyRonveauxZarzoArea1997,ZarzoAreaGodoyRonveaux1997,AreaGodoyRonveauxZarzo1998,RonveauxHounkonnouBelmehdi1995,HounkonnouBelmehdiRonveaux2000}. The relevant situation in our context is the case of expansions in a family of classical orthogonal polynomials $(\psi_n)_{n\in\nat}$. These works target coefficients of \emph{polynomials} $P_k$, so that the existence of coefficients required by our \zcref{eq:hypH,eq:hypH'2,eq:hypH''2} automatically hold. The first step of the method is to construct a linear differential operator~$L_k$ that cancels~$P_k$ (this can be obtained using the classical closure properties of linear differential operators, starting from a small database for the classical families~\cite{Stanley1999,SalvyZimmermann1994}). Next, the algorithm uses the structure relation~\zcref{eq:structure} and the three-term recurrence~\zcref{eq:3term-rec} repeatedly until a recurrence is found for the coefficients.

In the setting presented here, the algorithm~%
\cite{RonveauxZarzoGodoy1995} is very closely related to the
discussion of \zcref{sub:singularities_at_endpoints}: it rewrites the
differential operator~$L_k$, after multiplying it by a power
of~$\sigma$ if necessary, as a polynomial in $
(x,\sigma\partial)$.
Next, one can apply \zcref{thm:main},
or
directly the explicit formulas in its proof. In practice, the powers
of~$\sigma$ that are needed in this method often lead to recurrences
of order larger than necessary~\cite{GodoyRonveauxZarzoArea1997}.

\subsection{Rebillard's algorithms}\label{sec:rebillard}
If $L$ is a linear differential operator and $f,g$ are functions
related by~$L\cdot f=g$, Rebillard's method uses this relation to
deduce
one of the form
\[
  \Pnum\cdot[\psi_n](f)=\Pden\cdot[\chi_n](g),
\]
for a well-chosen family~$(\chi_n)$~\cite{Rebillard1997,Rebillard1998,RebillardZakrajsek2007}. 

In its simplest variant~\cite{Rebillard1997,Rebillard1998}, the family~$(\chi_n)$ is
formed with the $r$th derivatives~$(\psi_{n}^{(r)})$ where~$r$ is
the order of~$L$. For a classical orthogonal family~$(\psi_n)$, the family
$(\psi_{n+r}^{(r)})$ is again an orthogonal family~\cite[p.~150]{Chihara1978}, which is used
to justify the method under conditions dictated by 
analysis~\cite[Prop.~10]{Rebillard1997}.

Algebraically, let $\smallint$ denote the anti-differentiation
operator, with commutation
$\partial \int = \textrm{Id}$ (but
$\int\partial\neq\textrm{Id}$ as can be seen by application to
constants). As starting points, the same fractions
$D=(S_n,\Dden)$ for differentiation and $X^\psi=(\Xnum,S_n)$ for
multiplication
as in \zcref{lem:Dcop,lem:x} are used. Integration by parts and
induction
give
\[
  {\smallint}^k x \partial^k = 
  x - k \smallint,\qquad  k \in \nat.
\]
Writing~$I_\psi$ for the fraction~$(D^\psi)^{-1}=(\Dden,S_n)$ that
corresponds to integration and~$X_r=X\oplus
(-rI_\psi)$,
the operator ${\smallint}^rL$ is then mapped to a fraction for the
coefficients as follows:
\[
  \sum_{i = 0}^r {\smallint}^r p_i (x) \partial^r {\smallint}^{r-i}\mapsto
   \sum_{i=0}^{r}{p_i(X_r)\Ipsi^{r-i}}.
\]
An algorithm then follows by Horner evaluation.
It is analogous to the process
described by Olver and Townsend~\cite{OlverTownsend2013}.

Another way of obtaining this fraction is to multiply the result of
the left Horner evaluation of~$L$ from \zcref{eq:paire_horner_gauche}
by~$D^{-r}$ on the left. 
Since $D^r=(S_n^r,\Ddenr)$, we have $(D^\psi)^{-r}=(\Ddenr,S_n^{r})$
and the denominator in \zcref{eq:paire_horner_gauche} simplifies in
the
product. It follows that when $\gcd(p_r,\sigma) = 1$, Rebillard's
method produces the same irreducible fraction as ours.
Otherwise, the result of these algorithms may have different orders
\begin{example} The Chebyshev expansion of the function $(1-x^2)^
{-1/4}$ has been dealt with in
\zcref{ex:first-sing,ex-5}, where the second order recurrence~%
\eqref{eq:coeffs-ex-5} was computed. The numerator returned by the
other algorithms has order~4:
\[(2n+1)c_n-4(n+2)c_{n+2}+(2n+7)c_{n+4}=0.\]
This operator $A=(2n+7)S_n^4-4(n+2)S_n^2+(2n+1)$ factors as $
(S_n^2-1)\Pnum$ with~$\Pnum$ as in Example~\ref{ex-5}.
\end{example}
In his work with Zakrajšek, other bases of expansions are explored and
a method to find a ``minimal expansion basis'' is given. In some
cases, it returns recurrences of smaller order~\cite[Example~5.1]
{RebillardZakrajsek2007}. This method is purely formal and therefore
applicable without care only for the expansions of polynomials.
\begin{example}\cite[Ex.~5.1]{RebillardZakrajsek2007}\label{ex:RZ}
For the linear differential equation
\[(x+1)^2y''(x)-(x+1)y'(x)+(x+7/4)y(x)=0,\]
the recurrence for the Chebyshev coefficients returned by our method
and others before has order~4.
With their method, Rebillard and Zakrajšek find a right divisor of it
of order only~3:
\begin{multline}\label{eq:rec-RZ}
2(2n + 1)c_{n+3} + (2n + 1)(4n^2 + 24n + 33)c_{n+2} \\
+ (2n + 5)
(4n^2 - 3)c_{n+1} + 2(2n + 5)c_n=0.
\end{multline}
However, the general solution of the differential
equation is 
\[c_1(x+1)J_1(2\sqrt{x+1})+c_2(x+1)Y_1(2\sqrt{x+1}),    
\]
where~$J_1$ and~$Y_1$ are Bessel functions of the first and second
kinds. If~$c_1=0$, this function behaves like~$c(x+1)^{1/2}$ at~$x=-1$
for a constant~$c$,
while it is analytic at~$x=1$. This implies that its Chebyshev
coefficients behave asymptotically like~$C(-1)^n/n^2$ for some
constant~$C$~\cite{Elliott1964},
but this is not compatible with \zcref{eq:rec-RZ}:
when~$c_n=(-1)^nn^\alpha$ with $\alpha\neq-4$, the
left-hand side of \eqref{eq:rec-RZ} is asymptotic to~$
8(4+\alpha)(-1)^nn^{-\alpha-2}$ and thus only~$\alpha=-4$
can correspond to a solution. Thus, the recurrence \eqref{eq:rec-RZ}
only holds for the Chebyshev coefficients of some of the
solutions of the differential
equation.
\end{example}

\subsection{Another algorithm by Lewanowicz}
\label{sec:lewa2}

Generalizing some of his previous works~%
\cite{Lewanowicz1991,Lewanowicz1995,Lewanowicz1996}, Lewanowicz~%
\cite{Lewanowicz2002} gave an algorithm that takes advantage of
factors of $\sigma$ occurring in the leading coefficient of a
differential equation. For instance, if a linear differential
operator is of the form~$L=\sigma(x)v(x)\partial^n+Q$ with $Q$ of
order less than~$n$, then it can be rewritten~$L=\partial^{n-2}
(\sigma(x)\partial^2+\tau(x)\partial)v(x)+\tilde Q$ with $\tilde Q$ of
order less than~$n$. The motivation is that the fraction associated
 to~$M:=(\sigma\partial^2+\tau\partial)$ is simply~$(-\lambda_n,1)$. Higher
powers of $\sigma$ can be accommodated by using higher powers of~$M$,
up to~$\lfloor n/2\rfloor$. If~$n=1$, one can still use
$U=\sigma\partial+\tau$, whose associated pair is~$
(\lambda_k\Dnum,\Dden)$.

In the case of Jacobi polynomials~$P_n^{(\alpha,\beta)}$, it is
also possible to exploit leading
coefficients of the form~$\sigma(x)^r(x-\epsilon)^sw(x)$ with
$\epsilon\in\{-1,1\}$, $w(-1)w(1)\neq0$ and $s>0$. For this, one 
makes use of the operators~$Z_\epsilon=(x-\epsilon)\partial$ and
$K_\epsilon=(Z_\epsilon+\nu_\epsilon)\partial$, with $\nu_\epsilon$
related to~$\alpha,\beta$ and whose pair has numerator~$S_n$ and a
denominator of order~2 with a right factor of order~1 that is a right
divisor of $\Dden$.

These considerations lead to completely explicit formulas for the
pairs associated to the product~$\partial^aK_\epsilon^bZ_\epsilon^cM^dU^e$ and an algorithm that finds which of these monomials to use as a function of the
leading coefficient and order and proceeds inductively by reduction of
orders. The whole method is Euclid-free, but there is no proof that
the resulting fraction is irreducible.

\subsection{Creative telescoping}\label{sub:creative_telexcoping}
The method of creative telescoping is a family of algorithms that
compute linear recurrences or differential equations satisfied by
definite sums or integrals. Initiated by 
Zeilberger~\cite{Zeilberger1991a}, this method was improved and
extended in several directions. A general purpose implementation
based on Koutschan's algorithm~\cite{Koutschan2010} is available in
Mathematica in his package 
\textsf{HolonomicFunctions}\footnote{\url{https://risc.jku.at/sw/holonomicfunctions/}} and
another one based on a
recent reduction-based
algorithm~\cite{BostanChyzakLairezSalvy2018}, called 
\textsf{CreativeTelescoping}\footnote{
\url{https://github.com/HBrochet/CreativeTelescoping}
}, can be used in Maple.

In order to deal with a definite integral like the one defining
the generalized Fourier coefficients~$a_n$ in \zcref{eq:integrals},
this methods starts from linear differential and recurrence operators
annihilating the integrand~$f\psi_nw/h_n$. These operators generate a
left ideal~$\mathcal I$ in a suitable ring of operators. Next, for
increasing
order~$p$, the
method looks for a shift operator~$T(n,S_n)$ of order~$p$ ($T$ is
known as a \emph{telescoper}) such that there exists another
operator~$C$ in all relevant variables ($C$ is
called the \emph{certificate}) and~$T-\partial C$ belongs to
the ideal~$\mathcal I$. This implies that $T\cdot f$ is a derivative.
Under favorable conditions, the integral of $T\cdot f$ equals $T\cdot
\int_a^b f$ and the integral of the derivative $\partial C\cdot f$
is~0. Then, the shift operator~$T$ cancels the sequence~$a_n$, i.e.,
it is a recurrence operator for the generalized Fourier coefficients. 

In many of the examples presented in this article, this
general-purpose method returns the same
recurrence as the one our method produces. However, the computation is
often very much faster with \zcref{algo:main} that takes
advantage of the specific structure of the integrand. For
instance, for
the connection coefficients of the Jacobi polynomials $P_n^{(a,b)}(x)$
in the basis $(P_m^{(c,d)}(x))$, a recurrence of order~2 is obtained
in less than 1~sec. by our implementation of \zcref{algo:main}.
It is found in 33~sec. by Koutschan's heuristic (using the function
\textsf{FindCreativeTelescoping} of his package 
\textsf{HolonomicFunctions}), in 6~min. by reduction-based
creative telescoping in 
\textsf{CreativeTelescoping} and in more than 2~hours using the
non-heuristic implementation of creative telescoping in 
\textsf{HolonomicFunctions}.

Another issue is that a direct use of the method of creative
telescoping without checking that the integral of the certificate
vanishes is bound to produce wrong results. For instance, for the
expansion of the product $C_i^{(1)}(x)C_j^{(1)}(x)$ of Gegenbauer
polynomials in the basis~$(C_n^{(1)}(x))$, our method produces a
linear recurrence of order~4, while creative telescoping produces~$u_
{n+2}-u_n=0$, which does not hold at~$n=i+j$ and~$n=|i-j|-2$. On
another connection problem, from the Laguerre polynomial $xL_k^{-1/2}
(x^2)$ to the Hermite polynomials~$(H_n(x))$, creative telescoping
returns~$u_n=0$, which is only correct for $n>2k+1$, while our method
produces a linear recurrence of order~4 that is valid for all~$n\ge0$.

Still, there are also a few examples where creative telescoping
produces a
recurrence that is valid for all~$n\ge0$ and of order smaller than the
one produced by \zcref{algo:main}. This happens for instance for
the expansion of $x\exp(x)$ in the basis of Chebyshev polynomials~$
(T_n(x))$, where the method presented here produces the recurrence
\[u_{n+4}+2(n+4)u_{n+3}+2nu_{n+1}-u_n=0,\]
while creative telescoping finds a recurrence of order only~2 in that
case:
\[(n^2+n+1)u_{n+2}+2(n+1)^3u_{n+1}-(n^2+3n+3)u_n=0.\]
For a discussion of why finding a minimal fraction does not
necessarily
produces a minimal recurrence, see \zcref{sec:conclusion}.

\section{Examples}\label{sec:examples}

We report on the results of a Maple implementation of our algorithm on
a selection of examples\footnote{A session with more examples from
which
those presented here are extracted can be found together with the
software at\\
\url{https://perso.ens-lyon.fr/bruno.salvy/software/the-GeneralizedFourierSeries-package/}}. Given a function as
input, that package first computes a differential equation it
satisfies using the package \texttt{gfun}~\cite{SalvyZimmermann1994}
that constructs the equation from a database using closure properties.
From there, the algorithm presented here deduces a recurrence. In some
cases, it is more convenient to enter a differential equation as input
directly and this is also possible. The functions of the package take
as input a name~$u$ for the sequence of coefficients $u_n=[\psi_n]f$.

\smallskip
\emph{
Except when stated explicitly, all the examples presented in this
section take less than
1~sec. with our implementation.}\footnote{The timings were obtained with Maple2025.2 on a
2018~Mac mini.}

\subsection{Relations between polynomials}
This is the situation that has been considered most frequently by
existing software. 

When expanding polynomials in a basis of classical orthogonal
polynomials, the convergence conditions of \zcref{thm:main} are
trivially
satisfied so that the output of our program is satisfied by the
coefficients. 

\subsubsection{Connection coefficients}\label{ssub:connection_coefficients}
For a summary of previous approaches to recurrence relations for the
coefficients of
expansions of a family of polynomials in another one, we refer
to~\cite{RonveauxZarzoAreaGodoy2002} where the authors also summarize
their Navima algorithm. We list here a few
representative examples obtained by our algorithm.

\medskip
\paragraph{\em A trivial recurrence}
Expanding the $k$th Hermite polynomial $H_k(x)$ in the basis $(H_n
(x))$
results in a recurrence of order~0:
\begin{Verbatim}
> FunctiontoRec(HermiteH(k,x),HermiteH(n,x),u)=0;
\end{Verbatim}
\[{\color{blue}
(n-k)u_n=0}
\]
The meaning of such a recurrence is that for all $n\neq k$, $u_n=0$,
as expected, while the coefficient $u_k$ is not constrained by the
recurrence (it would be an initial condition of it).

The same recurrence is obtained for the decomposition of the Hermite
polynomial~$H_{2n}(\sqrt x)$ in the Laguerre basis $(L_k^{(-1/2)}
(x))$, recovering the known relation~\cite[18.7.E17]{NIST:DLMF}
between these polynomials
(in that case the initial condition is recovered from their leading
coefficients). The relation between Bessel and Laguerre
polynomials~\cite[Eq. (2.5)]{Al-Salam1957} can also be found that way.

\medskip
\paragraph{\em Laguerre to Laguerre}
The expansion of the Laguerre polynomial $L_m^{(\alpha)}(x)$ in the
basis $L_\nu^{(\beta)}(x)$ has been given by Howell~\cite{Howell1937}.
It is easily recovered automatically:
\begin{Verbatim}
> FunctiontoRec(LaguerreL(m,alpha,x),LaguerreL(nu,beta,x),u)=0;
\end{Verbatim}
\[{\color{blue}
(-m + \nu)u_\nu + (\alpha - \beta + m - 1  - \nu)u_{\nu+1}=0
}\]
The initial condition at $m=\nu$ can be deduced from the leading
coefficients of $L_\nu^{(\alpha)}(x)$ and $L_\nu^{(\beta)}(x)$.

It is interesting to note that the method of Godoy \emph{et al.}~\cite{GodoyRonveauxZarzoArea1997}
yields a recurrence of order~2 rather than~1 for these coefficients.

\medskip
\paragraph{\em Jacobi to Jacobi}
The most general situation is the expansion of $P_n^{(a,b)}(x)$ in the
basis $(P_m^{(c,d)}(x))$, which leads to a recurrence of order~2 for
the coefficients that we do not reproduce here due to its length 
(see~\cite[Thm.~9.1.1]{Ismail2005}). An interesting special case is
the expansion in the basis $(P_n^{(a,c)}(x))$ where the recurrence has
order only~1:
\begin{Verbatim}
> FunctiontoRec(JacobiP(n,a,b,x),JacobiP(k,a,c,x),u)=0;
\end{Verbatim}
{\color{blue}
\begin{multline*}
(a +c +2 k +3) (k +a +c +1) (k -n ) (a +b +k +n +1) u_k\\
-(a+k +1)(a +c +2 k +1) (a +c +k +n +2) (b-c -k +n -1) u_{k+1} =
0
\end{multline*}
}
Again, the initial condition at $n=k$ is deduced from the leading
coefficients and results in the known formula for these coefficients
\cite[Cor.~1.2]{Ismail2005}.

\medskip
\paragraph{\em Associated Laguerre to Laguerre}
The associated Laguerre polynomials satisfy a linear differential
equation of order~4~\cite{Ronveaux1988} which can be given as input to
our program:
\begin{Verbatim}
> deq:=n*(n+2)*y(x)+(3*n+3*alpha+6*c-3*x)*diff(y(x),x)+
(-alpha^2+2*alpha*x+4*c*x+2*n*x-x^2+4)*diff(y(x),x$2)
+5*x*diff(y(x),x$3)+x^2*diff(y(x),x$4);
\end{Verbatim}
{\color{blue}
\begin{multline*}
n \left(n +2\right) y \! \left(x \right)+\left(3 n +3 \alpha +6 c -3 x
\right) \left(\frac{d}{d x}y \! \left(x \right)\right)\\
+\left(-\alpha^{2}+2 x \alpha +4 c x +2 n x -x^{2}+4\right) \left(
\frac{d^{2}}{d x^{2}}y \! \left(x \right)\right)\\
+5 x \left(\frac{d^{3}}{d x^{3}}y \! \left(x \right)\right)+x^{2} \left(\frac{d^{4}}{d x^{4}}y \! \left(x \right)\right)
\end{multline*}
}
\begin{Verbatim}
> DiffeqtoRec(deq,y(x),u(m),LaguerreL(m,alpha,x))=0;	
\end{Verbatim}
{\color{blue}
\begin{multline*}
\left(m +n +2\right) \left(m -n \right) u_m+\left(4 c
m -3 m^{2}+2 n m +n^{2}+6 c -9 m +5 n -6\right) u_{m+1}\\
-2\left (2 c -m +n -2\right) \left(\alpha +2+m \right) u_{m+2}=0
\end{multline*}
}
This is a case where our algorithm returns a recurrence of
smaller order than the one produced by the Navima method 
\cite{RonveauxZarzoGodoy1995}.

\subsubsection{Linearization coefficients}\label{ssub:linearization_coefficients}
For a given family of orthogonal polynomials ($P_n(x;\mathbf a)$)
with parameters $\mathbf a$, the
linearization coefficients are the coefficients~$c_{m,n,k}(\mathbf a)$
in the
decomposition 
\[P_m(x;\mathbf a)P_n(x;\mathbf a)=\sum_{k=0}^{m+n}{c_{m,n,k}
(\mathbf a)P_k(x;\mathbf a)},\]
see for instance \cite[Ch.~9]{Ismail2005}. 

\medskip
\paragraph{{\em Hermite Polynomials}}
For these polynomials, the coefficients have been given by
Feldheim~\cite{Feldheim1938}. They can be computed easily from the
recurrence:
\begin{Verbatim}
> FunctiontoRec(HermiteH(m,x)*HermiteH(n,x),HermiteH(nu,x),u)=0;
\end{Verbatim}
\[{\color{blue}
\left(m +n-\nu  \right) u_\nu-\left(\nu-n +2+m \right) \left(n +2-m
+\nu \right) u_{\nu+2} = 0
}\]
This is a 2-term recurrence, due to the fact that the Hermite
polynomials have the parity of their index. Thus only one initial
condition is needed; it can be obtained from the leading
coefficients of the polynomials.

\medskip
\paragraph{\em Laguerre Polynomials}
A direct use of our program gives a 3rd order recurrence.
However, there is a technique  introduced by Lewanowicz~[Thm. 2.6, Rem. 2.7]\cite{Lewanowicz1996} that can be applied to reduce it to order~2: start from a linear differential operator $L(x,\partial)$ that annihilates $L_i^{(\alpha)}(x)L_j^{(\alpha)}(x)$ and instead of applying \zcref{thm:main} to~$L(x,\partial)$, apply it to its derivative~$\partial L (x,\partial)$, provoking a simplification in the pair.
\begin{Verbatim}
> deq:=gfun:-holexprtodiffeq(
	LaguerreL(i,alpha,x)*LaguerreL(j,alpha,x),y(x),false):
> DiffeqtoRec(diff(deq,x),y(x),u(k),LaguerreL(k,alpha,x))=0;
\end{Verbatim}
{\color{blue}
\begin{multline*}
-2 \left(2+k \right) \left(k +1\right) \left(i +j -k \right) u_k\\
+\left(2+k \right) \left(i^{2}-2 i j +2 i k +j^{2}+2 j k -3 k^{2}+3 i
+3 j -7 k -4\right) u_{k +1}\\
-\left(k +2-j +i \right) \left(-k -2-j
+i \right) \left(\alpha +2+k \right) u_{k+2} = 0
\end{multline*}
}
The same method applies to Gegenbauer and Jacobi polynomials, where
the order is decreased from 4 to~2 (the Jacobi case takes 11~sec.).

\medskip
\paragraph{\em Generalized linearization problems}
Ronveaux \emph{et al.}~\cite{RonveauxHounkonnouBelmehdi1995}
have proved the existence of recurrences satisfied by the coefficients
of the decomposition of a product of two orthogonal polynomials not
necessarily of the same family in the basis given by a 3rd family.
These can be produced by our method, but the orders
tend to become large. 
For instance: 
\begin{itemize}
	\item[--] order 9 for $L^{(\alpha)}_i(x)L^{(\alpha)}_j(x)$ in the
	basis
	$ (H_k(x))$ and $H_i(x)H_j(x)$ in the basis $(L^{(\alpha)}_k(x))$; 
\item[--] order 12 
for $H_i(x)H_j(x)$ and for $L^{(\alpha)}_i(x)L^{(\alpha)}_j(x)$ in the
basis $(P^{(a,b)}_k(x))$ (in 1.6 and 2.8~sec.), 
for $P^{(a,b)}_i(x)P^{(a,b)}_j(x)$ in the bases $(H_k(x))$ and $(L^{(\alpha)}_k(x))$;
\item[--] order 20 for $H_i(x)P_j^{(a,b)}(x)$ in the basis $(H_k(x))$
and for $L_i^{(\alpha)}(x)P_j^{(a,b)}(x)$ in the basis $(L_k^{
(\alpha)}(x))$ (the latter one in 5~sec.).
\end{itemize}

\subsubsection{Other relations between polynomials}
Many relations between Hermite and Laguerre polynomials, changing the
argument $x$ into $x+y$ or $cx$, have been given by Feldheim~%
\cite{Feldheim1938,Feldheim1940} and rederived automatically by Godoy
\emph{et al.}~\cite{GodoyRonveauxZarzoArea1997}; they are also found
by our program, with recurrences of the same orders.

Luke~\cite[Ch.~8]{Luke1969} gives the coefficients of $
(1-x)^m$ in the basis of Jacobi polynomials. 
The function $(1-x)^m$ for $m\in\mathbb N$ satisfies the first order
equation $(1-x)y'+my=0$. Our program returns the following recurrence for the
Jacobi coefficients:
\begin{Verbatim}
> FunctiontoRec((1-x)^m,JacobiP(n,alpha,beta,x),u)=0;
\end{Verbatim}
{\color{blue}\begin{multline*}
\left(\alpha +\beta +2 n +3\right) \left(\alpha +\beta +n +1\right)
\left(m -n \right) u_n\\
+\left(\alpha +n +1\right)
\left(\alpha +\beta +2 n +1\right) \left(\alpha +\beta +m +n +2\right)
u_{n +1}=0.
\end{multline*}}
By reduction to the Beta integral, one obtains the initial condition
\[u_0=\frac{2^m\Gamma(\alpha+\beta+2)\Gamma(\alpha+m+1)}{\Gamma
(\alpha+\beta+m+2)\Gamma(\alpha+1)},\]
so that we recover the known formula for the coefficients~\cite[(30)
p.~277]{Luke1969},
\[(1-x)^m=\sum_{n=0}^m{\frac{2^{m}m! \left(-1\right)^{n} \left(\alpha
+\beta +2 n
+1\right) \Gamma \! \left(n +\alpha +\beta +1\right) \Gamma \! \left(m +\alpha +1\right)}{ 
\left(m -n\right)!\,\Gamma \! \left(m +n +\beta +\alpha +2\right)
\Gamma \! \left(\alpha +n +1\right)}P_n^{(\alpha,\beta)}(x)}.
\]
A similar result is obtained for~$x^m$. This time, our program returns the second-order linear
recurrence
\begin{Verbatim}
> FunctiontoRec(x^m,JacobiP(n,alpha,beta,x),u)=0;
\end{Verbatim}
\begin{small}
{\color{blue}\begin{multline*}
2 \left(2 n +4+\beta +\alpha \right) \left(\alpha +\beta +2 n
+5\right) \left(n +2+\alpha +\beta \right) \left(n +\alpha +\beta
+1\right) \left(m -n \right) u_n\\
+\left(\alpha +\beta +2 n +5\right) \left(n +2+\alpha +\beta \right)
\left(\alpha -\beta \right) \left(\alpha +\beta +2 m +2\right) \left
(\alpha +\beta +2 n +1\right) u_{n +1}\\
-2 \left(\alpha +\beta +2 n +1\right) \left(\beta +n +2\right) \left
(\alpha +n +2\right) \left(2 n +2+\alpha +\beta \right) \left(m +n
+3+\alpha +\beta \right) u_{n+2}
 = 0.
\end{multline*}}
\end{small}
Indeed, this can be checked to be satisfied by the known
coefficients~\cite[(28) p.~277]{Luke1969}
\[x^m=\sum_{n=0}^m{\frac{2^{n} m !\, \Gamma \! \left(n +\alpha +\beta
+1\right) }{\left(m -n \right)!\,
\Gamma \! \left(\alpha +\beta +2 n +1\right)}\,
{}_2F_{1}\left(\begin{matrix}n
-m ,\alpha +n +1\\2 n +2+\alpha +\beta
\end{matrix}\,\middle|\,2\right)P_n^ {(\alpha,\beta)} (x)}. \]
The initial conditions can again be reduced to the Beta integral and
the proof of the identity is then reduced to checking these first two
values.

A simplification occurs for the coefficients of $x^m$ in the Gegenbauer
basis $C_n^{(\lambda)}(x)$. Our program returns
a two-term recurrence
\begin{Verbatim}
> FunctiontoRec(x^m,GegenbauerC(n,lambda,x),u)=0;
\end{Verbatim}
{\color{blue}\[\left(m -n \right) \left(\lambda +n +2\right) u_n
-\left(\lambda +n \right) \left(m +2 \lambda +n +2\right) u_{n+2}=0.\]}
Again, the initial conditions are not the objective of this work, but
they can be obtained without too much work, leading to the following
formula (a variant of~\cite[(29) p.~277]{Luke1969})
\[x^{2m}=\sum_{n=0}^m{\frac{(-1)^m(2m)!\,(\lambda+2n)\Gamma(\lambda)}
{4^m(m-n)!\,\Gamma(n+m+\lambda+1)}C_ {2n}^ {(\lambda)} (x)}\]
and a similar one for $x^{2m+1}$.

\subsection{Generating functions}
Generating functions provide good tests for our algorithms: by their
nature, they have simple sequences of coefficients, and the extra
variable can be made small enough that convergence conditions are met.

A nontrivial example is provided by the function
\begin{Verbatim}
> F:=GAMMA(nu+1/2)*exp(x*t)*(t/2*sqrt(1-x^2))^(1/2-nu)
  *BesselJ(nu-1/2,t*sqrt(1-x^2));
\end{Verbatim}
{\color{blue}
\[F:=\Gamma \! \left(\nu +\frac{1}{2}\right) {\mathrm e}^{x t} {\left(\frac{t \sqrt{-x^{2}+1}}{2}\right)}^{\frac{1}{2}-\nu} J_{\nu -\frac{1}{2}}\! \left(t \sqrt{-x^{2}+1}\right)\]
}
Its coefficients in the Gegenbauer basis $(C_k^{(\nu)}
(x))$ satisfy a 4th-order recurrence
\begin{Verbatim}
> rec:=FunctiontoRec(F,GegenbauerC(k,nu,x),u)=0;
\end{Verbatim}
{\color{blue}
\begin{multline*}
rec:=-\left(\nu +k +3\right) t \left(k +2\right) \left(k +1\right)
u_k\\
+\left(\nu +k +3\right) \left(k +2\right) \left(k^{2}+2 k \nu -t^
{2}+k +2 \nu \right) u_{k +1}\\
-2 t \left(k^{2}+2 k \nu +2 \nu^
{2}+4 k +3 \nu +3\right) u_{k +2}\\
+\left(2 \nu +k +2\right) \left
(\nu +k +1\right) \left(k^{2}+2 k \nu -t^{2}+7 k +8 \nu +12\right) u_
{k +3}\\
+\left(\nu +k +1\right) t \left(2 \nu +k +2\right) \left(2 \nu
+k +3\right) u_{k +4}=0
\end{multline*}
}
While this is not of minimal order, Petkov\v sek's algorithm readily
finds the solution $t^k/\Gamma(2\nu+k)$ from which the exact
value~\cite[5.13.1.3]{PrudnikovBrychkovMarichev1986a} follows.

\subsection{Chebyshev expansions}
A few examples have been given before (\zcref{ex:exp,ex-5,ex:arccos}).

Other examples can be obtained as special cases of the Fourier and
Jacobi expansions below. We just show two more examples due to
Lewanowicz. (Recall again the convention of \zcref{sec:note_on_chebyshev_series}.)

\noindent \textit{The Elliptic $E$~function~\cite[Ex.~2.1]{Lewanowicz1976}}:
\begin{Verbatim}
> FunctiontoRec(EllipticE(x),ChebyshevT(m,x),u)=0;
\end{Verbatim}
{\color{blue}
\[
\left(m -1\right) \left(m +1\right) u_m+\left(-8-4
m \right) u_{m +2}-\left(m +5\right) \left(m +3\right) u_{m
+4}	= 0 \]
}
\textit{The Lommel $s_{\mu\nu}(x)$ function~\cite[Ex.~3.1]{Lewanowicz1976}}. It satisfies the inhomogeneous differential equation
\begin{Verbatim}
> deq:=x^2*diff(y(x), x, x) + x*(2*mu - 1)*diff(y(x), x) 
  + (a^2*x^2 + mu^2 - nu^2 - 2*mu + 1)*y(x) - a^2*x^2;
\end{Verbatim}
{\color{blue}
\[x^{2} \left(\frac{d^{2}}{d x^{2}}y \! \left(x \right)\right)+x \left(2 \mu -1\right) \left(\frac{d}{d x}y \! \left(x \right)\right)+\left(a^{2} x^{2}+\mu^{2}-\nu^{2}-2 \mu +1\right) y \! \left(x \right)-a^{2} x^{2}\]}
Since the inhomogeneous part has only finitely many nonzero Chebyshev
coefficients, it is sufficient to apply the algorithm to the
homogeneous part, which recovers Lewanowicz' 8th~order recurrence:
\begin{Verbatim}
> hom:=deq-eval(deq,y(x)=0):
> rec:=DiffeqtoRec(hom,y(x),u(m),ChebyshevT(m,x))=0;
\end{Verbatim}
{\color{blue}
\begin{multline*}
-a^{2} \left(m +5\right) u_m
+\left(-4 m^{3}-8 m^{2} \mu -4 m\,\mu^{2}+4 m \,\nu^{2}-2 a^
{2}\right.\\ 
\left.-28 m^
{2}-48 \mu  m -20 \mu^{2}+20 \nu^{2}-44 m -40 \mu -20\right) u_{m
+2}\\ 
+2 \left(m +4\right) \left(a^{2}-4 m^{2}+4 \mu^{2}-4 \nu^{2}-32m
-16 \mu -44\right) u_{m +4}\\ 
+\left(-4 m^{3}+8 m^{2} \mu -4 m\,\mu^{2}+4 m \,\nu^{2}+2 a^{2}-68 m^
{2}\right.\\ 
\left.+80 \mu  m -12 \mu^{2}+12 \nu^{2}-364 m +168 \mu -588\right) u_{m
+6}\\
-a^{2} \left(m +3\right) u_{m +8}=0
\end{multline*}
}

\subsection{Fourier expansions}

Koepf and Chiadjeu give a direct algorithm for what they call 
\emph{trigonometric holonomic functions}~\cite{KoepfChiadjeu2015},
that are solutions of linear differential equations with
trigonometric coefficients. This gives an algorithmic access to linear
recurrences for Fourier series in a large variety of examples.

There are cases where our method can also recover Fourier expansions thanks to the relation
with Chebyshev polynomials. An example is provided by the function~\cite [Ex. 4.9] {KoepfChiadjeu2015}
\begin{Verbatim}
> f:=cos(t)*ln(2+cos(t));
\end{Verbatim}
{\color{blue}
\[f:=\cos(t)\ln(2+\cos(t))\]
}
Expanding $g=f(\arccos x)$ in the basis $T_n(x)$ and then evaluating
at
$x=\cos t$ shows that the Chebyshev coefficients of $g$ are the
Fourier coefficients of $f$. Starting from the inhomogeneous
differential equation
\begin{Verbatim}
> deq:=gfun:-holexprtodiffeq(eval(f,t=arccos(x)),y(x),false);
\end{Verbatim}
{\color{blue}
\[\left(x^{2}+2 x \right) \left(\frac{d}{d x}y \! \left(x
\right)\right)
-(x+2)y \! \left(x \right)  -x^{2}\]
}%
\noindent%
and using the homogeneous part gives a simple 4th order
recurrence for
these coefficients:
\begin{Verbatim}
> hom:=deq-eval(deq,y(x)=0):
> DiffeqtoRec(hom,y(x),u(n),ChebyshevT(n,x))=0;
\end{Verbatim}
{\color{blue}
\[
\left(n -1\right) u_n
+4 n u_{n+1}+2\left( n +2\right) u_{n +2}
+4\left(n +4\right) u_{n +3}
+\left(n +5\right) u_{n +4}=0
\]
}
(Koepf and Chiadjeu's method find a more complicated recurrence, but
of order~2 only, which is minimal in this case.)
A basis of solutions of this recurrence can be found by Petkov\v sek's
algorithm and from there one can recover the correct solution for
$n$ sufficiently large (for small $n$, the inhomogeneous part of the
differential equation plays a role).

\subsection{Laguerre expansions}
In view of \zcref{table:pairs_rec_cop}, linear differential operators
with \emph{constant} coefficients lead to recurrences that do not
depend on the parameter~$\alpha$ of the Laguerre polynomials. Thus, for the
expansion of $\exp(cx)$ (for $c<1$), we obtain a recurrence that is
independent of~$\alpha$:
\begin{Verbatim}
> FunctiontoRec(exp(c*x),LaguerreL(n,alpha,x),u)=0;
\end{Verbatim}
{\color{blue}
\[cu_n+(1-c)u_{n+1}=0\]}
The actual value is easily computed from the integral representation
of the $\Gamma$ function.

A slightly more involved expansion can be found for 
\begin{Verbatim}
> F:=exp(x-x/lambda)/lambda^(n+alpha+1)*LaguerreL(n,alpha,x/lambda);
\end{Verbatim}
{\color{blue}
\[F:=\lambda^{-n-\alpha-1}e^{x-x/\lambda}L_n^{(\alpha)}(x/\lambda)\]
}
\begin{Verbatim}
> FunctiontoRec(F,LaguerreL(r,alpha,x),u)=0;
\end{Verbatim}
{\color{blue}
\[-(\lambda - 1)(1 + r)u_r + (n - r - 1)u_{r+1}=0\]
}
In view of the simplification that takes place on the exponentials,
together with orthogonality, it follows that the coefficients are~0
for~$r<n$. For $r=n$, it is sufficient to reason on the leading
coefficient of~$L_r^{(\alpha)}$, which gives the initial
condition~$u_r=1$. Thus, we recover~\cite[Eq.~(18)]{Howell1937}:
\[e^{x-x/\lambda}L_n^{(\alpha)}\!\left(\frac
x\lambda\right)=\lambda^{n+\alpha+1}\sum_ {r\ge n} \binom{r} {n} 
(1-\lambda)^ {r-n}L_r^ {(\alpha)} (x),\qquad\lambda>0.\]
Convergence of the integrals makes this meaningful for
$\lambda>0$ only. Howell~\cite{Howell1937}
also gives an unnecessary restriction $\lambda<2$ coming from his
derivation.

\subsection{Hermitian expansions}

There are several examples of expansions in the basis of Hermite
polynomials in~\cite[sec. 5.12 and 4.5]{PrudnikovBrychkovMarichev1986a}. 

A simple example is
\begin{Verbatim}
> F:=sin(2*a*x)/x;
\end{Verbatim}
{\color{blue}
\[F:=\frac{\sin(2ax)}x\]
}
Its coefficients in the Hermite basis satisfy a 4th~order recurrence:
\begin{Verbatim}
> FunctiontoRec(F,HermiteH(k,x),u)=0;
\end{Verbatim}
{\color{blue}
\[a^{2} u_k
+\left(2+k \right) \left(2 a^{2}+k+3\right)u_{k+2}
+2 \left(2+k \right) \left(3+k \right) \left(4+k\right) u_{k+4}=0\]
}
Since the function is even, the initial conditions $u_1=u_3=0$ are
easily found. With some help, Maple also finds
\[u_0=\sqrt\pi\operatorname{erf}(a),\quad u_2=\frac a2e^{-a^2}-
\frac{\sqrt\pi}4\operatorname{erf}(a).\]
With these, one recovers~\cite[5.12.3.1]{PrudnikovBrychkovMarichev1986a}
\[\frac{\sin2ax}x=\sum_{k=0}^\infty{\frac{(-1)^k}{(2k)!}
\,\gamma\!\left(k+\frac12,a^2\right)H_{2k}(x)},\]
where $\gamma$ is the incomplete gamma function.

\medskip

Another typical example is the function
\begin{Verbatim}
> f:=sqrt(x)*exp(-x^2*t^2/(1-t^2))*BesselI(-1/4,t*x^2/(1-t^2));
\end{Verbatim}
{\color{blue}
\[f:=\sqrt x\exp\!\left(-\frac{t^2x^2}{1-t^2}\right)I_{-1/4}\left(
\frac{tx^2}{1-t^2}\right)\]
}
which is in $L^2((0,\infty),e^{-x^2})$ for $t\in(0,1)$, together with
all its derivatives. Thus \zcref{thm:main} implies that its
coefficients in the Hermite basis satisfy
\begin{Verbatim}
> FunctiontoRec(f,HermiteH(k,x),u)=0
\end{Verbatim}
{\color{blue}
\[
t^2u_k-4(k+3)(k+4)u_{k+4}=0
\]}
This is a variant of~\cite[5.12.1.22]{PrudnikovBrychkovMarichev1986a}
that can also be obtained by our program, but would take more space
here.

\medskip

We conclude with an example that is not analytic: the sign function.
Starting from the first order differential equation~$xy'=0$ gives a
linear
recurrence of order~2:
\begin{Verbatim}
> DiffeqtoRec(x*diff(y(x),x),y(x),u(k),HermiteH(k,x))=0;
\end{Verbatim}
{\color{blue}
\[ku_k+2(k+2)(k+1)u_{k+2}=0\]}
Since the function is odd, the only nontrivial initial condition
is~$u_1$, easily found to be~$1/(2\sqrt\pi)$. With this, it is
immediate to recover~\cite[5.12.1.13]{PrudnikovBrychkovMarichev1986a}
\[\operatorname{sgn}x=\frac1{\sqrt\pi}\sum_{m\ge0}{\frac{(-1)^m4^{-m}}
{2(2m+1)m!}}H_{2m+1}(x).\]
For the convergence of this type of expansion in the complex plane,
see~\cite{Hille1939,Hille1940,Hille1980}.

\subsection{Bessel expansions}\label{example:bessel_expansions}
Al-Salam gave many properties of Bessel polynomials, in a notation
slightly different from that of the DLMF~%
\cite{OlverLozierBoisvertClark2010}, which we follow. We consider only
his Eq.~(7.5), giving the expansion of~$(-x/2)^s$ for $s\in\mathbb N$:
\begin{Verbatim}
> FunctiontoRec((-x/2)^s,Bessely(n,b+2,x),u)=0;
\end{Verbatim}
{\color{blue}
\[
\left(b +1+n \right) \left(b +2 n +3\right) \left(n -s \right)u_n
-\left(b +2 n +1\right) \left(n +1\right) \left(b +n +s +2\right) u_{n +1} = 0
\]}
The initial condition at~$n=s$ is deduced from the leading coefficient
of~$y_n(x;b+2)$, leading to
\[\left(-\frac x2\right)^s=
\sum_{n=0}^s{\frac{(-s)_n(b+2n+1)(b+n)!}{(b+n+s+1)!n!}y_n(x;b+2)}.
\]

\subsection{Jacobi expansions}
Examples of expansions in the basis of Jacobi polynomials can be found
in \cite[5.14]{PrudnikovBrychkovMarichev1986a}. 

For instance, the function
\begin{Verbatim}
> F:=(1+t)^(-alpha-beta-1)*hypergeom([(alpha+beta+1)/2,
  (alpha+beta+2)/2],[beta+1],(2*t*x+2*t)/(t+1)^2);
\end{Verbatim}
{\color{blue}
\[
F:=\left(1+t \right)^{-\alpha -\beta -1} {}_{2}F_{1}\! \left(\left.
\begin{matrix} \frac{\alpha}{2}+\frac{\beta}{2}+1,\frac{\alpha}{2}+
\frac{\beta}{2}+\frac{1}{2}\\ 
\beta +1\end{matrix}\right|\frac{2 t x +2 t}{\left(1+t
\right)^{2}}\right)
\]}
satisfies a linear differential equation of order~2 with respect
to~$x$, from which our code finds a recurrence of order~2 for the
Jacobi coefficients:
\begin{Verbatim}
> FunctiontoRec(F,JacobiP(k,alpha,beta,x),u)=0;
\end{Verbatim}
{\color{blue}
\begin{multline*}
- \left(k +\alpha +\beta+2 \right) 
 \left(k +\alpha +\beta +1\right) t u_k\\
+ \left(\beta +k +1\right)
 \left(2+k +\alpha +\beta\right) (t^{2}+1) u_{k+1}\\ 
- \left(\beta +k+2 \right)
 \left(\beta +k +1\right) t u_{k+2}=0
\end{multline*}
}
The initial conditions
\[u_0=1,\quad u_1=\frac{\alpha+\beta+1}{\beta+1}t\]
can be deduced from a variant of the integral \cite[2.21.1.24]{PrudnikovBrychkovMarichev1989}
\begin{multline*}\int_{-1}^{1}{
{}_{2}F_{1}\! \left(\left.
\begin{matrix}
\frac{\alpha}{2}+\frac{\beta}{2}+\frac{1}{2},
\frac{\alpha}{2}+\frac{\beta}{2}+1
\\ \beta +1
\end{matrix}\right| \frac{2 t
\left (x +1\right)}{\left(t +1\right)^{2}}\right)
\left(1-x \right)^{\alpha +k}\left(1+x\right)^{\beta +\ell}
\,dx} =\\
B\! \left(\alpha +k+1 , \beta +\ell +1\right) 2^{\alpha +\beta +k+\ell+1}
{}_{3}F_{2}\! \left(\left.
\begin{matrix}
\frac{\alpha}{2}+\frac{\beta}{2}+\frac{1}{2},
\frac{\alpha}{2}+\frac{\beta}{2}+1,
\beta +\ell +1
\\
\beta
+1,\alpha+\beta+k+\ell +2
\end{matrix}
\right|\frac{4 t}{\left(t
+1\right)^ {2}}\right).
\end{multline*}
From these initial conditions, using Petkov\v sek's algorithm on
the recurrence recovers the expansion~\cite[5.14.1.6]
{PrudnikovBrychkovMarichev1986a}
\[F=\sum_{k=0}^\infty{\frac{(\alpha+\beta+1)_k}{(\beta+1)_k}t^kP_k^
{(\alpha,\beta)}(x).}\]

\medskip

\noindent
Another recurrence of order~4 is obtained for $\sin(xz)$ on
the basis
$(P_k^{(\alpha,\alpha)}(x))$:
\begin{Verbatim}
> FunctiontoRec(F,JacobiP(k,alpha,alpha,x),u)=0;
\end{Verbatim}
{\color{blue}
\begin{multline*}
\left(2 \alpha +7+2 k \right)
  \left(k +3+2 \alpha \right) 
  \left(2\alpha +9+2 k \right)\\
  \left(k +4+2 \alpha \right)
  \left(k +2+2 \alpha \right)
  \left(k +2 \alpha +1\right) z^{2} u_k\\ 
+\left(\alpha +2+k  \right)
  \left(\alpha +1+k \right)
  \left(k +4+2 \alpha \right)
  \left(k +3+2  \alpha \right)
  \left(2 \alpha +9+2 k \right)\\
  \qquad\left(2 \alpha +2 k +1\right)
  \left(4 \alpha^{2}+8 \alpha  k +4 k^{2}-2 z^{2}+20 \alpha +20 k
  +21\right)u_{k +2}\\ 
+\left(\alpha +2+k \right)
  \left(2 \alpha +3+2 k \right)
  \left(\alpha +1+k \right)\\
  \left(2 \alpha +2 k +1\right)
  \left(\alpha +3+k \right)
  \left(\alpha +4+k \right) z^{2} u_{k +4}=0
\end{multline*}}
As that function is odd, the initial conditions
$u_0=u_2=0$ are clear and one obtains the other two in terms of
Bessel functions as
\[u_1=\frac{2^{\alpha +\frac{3}{2}} z^{-\alpha -\frac{1}{2}} \Gamma \!
\left(\frac{5}{2}+\alpha \right)}{\alpha +1} J_{\alpha +\frac{3}{2}}\! \left(z
\right),\qquad
u_3=-\frac{16 z^{-\alpha -\frac{1}{2}} 2^{\frac{1}{2}+\alpha} \Gamma
\! \left(\frac{9}{2}+\alpha \right) }{\left(3+\alpha \right) \left
(\alpha +2\right) \left(2 \alpha +5\right)}J_{\alpha+\frac{7}
{2}}\! \left(z \right).
\]
At this stage, it is not clear that algorithms are available to deduce
the general form \cite[5.14.5.1]{PrudnikovBrychkovMarichev1986a}
\[u_{2k+1}=
\left(-1\right)^{k}2^{\alpha+2k+\frac12}
z^{-\alpha -\frac{1}{2}}
\frac{
 \left(2\alpha+4 k +3\right)  \Gamma \! \left(\alpha +k+1 \right)
  \Gamma \! \left(\alpha +k+\frac{3}{2} \right)  }
 {\Gamma \! \left(\alpha +2 k+2 \right)} J_{2 k +\alpha + \frac{3}
 {2}}(z), \]
but at least, the formula can be proved by checking that it matches
the initial conditions and satisfies the recurrence.

\medskip

We conclude by mentioning a more difficult example~\cite[5.14.1.11]%
{PrudnikovBrychkovMarichev1986a}, the function
\begin{equation}\label{eq:gfJac}
{}_{2}F_{1}\! \left(\left.\begin{matrix}\gamma ,\alpha +\beta
-\gamma +1\\\alpha +1\end{matrix}\right| \frac{1-t+\rho}{2}\right)
{}_{2}F_{1}\! \left(\left.\begin{matrix}\gamma ,\alpha +\beta -\gamma
+1\\\beta +1\end{matrix}\right| \frac{1+t-\rho}{2}\right),
\end{equation}
with $\rho=\sqrt{t^{2}-2 x t +1}$. For this function, and even for special values of
$\alpha,\beta,\gamma,t$, our
implementation returns a recurrence of order~50 with coefficients of
degree~20 in $x$ (in about 10~min.),
while the solution is
simply
\[\frac{(\gamma)_k(\alpha+\beta-\gamma+1)_k}{(\alpha+1)_k
(\beta+1)_k}t^k,\]
solution to a recurrence of order~1. This is the worst example we have
encountered. One reason for this large order is that the algorithm
starts from a 
linear differential equation satisfied by the generating function~%
\eqref{eq:gfJac} that has order~10 and coefficients of degree up
to~32. Our algorithms return a recurrence that is satisfied by the
coefficients of all solutions of that differential equation. 
They do not take into account the somewhat symmetric form of the
generating function.

\section{Final comments}\label{sec:conclusion}
\subsection{Other expansions} 
The main tools used in this work are: recurrence rather than shift
operators; their fractions and Euclid-based algorithms for their sum
and products; the existence of a morphism from term algebras to
fractions of recurrence operators whose numerator annihilates a target
sequence. 

This article focuses on expansions on classical orthogonal
polynomials for solutions of linear differential equation and, to a
lesser extent, to Taylor expansions. There, we have used many ideas
coming from work of Lewanowicz, Paszkowski, Rebillard and others and
cast them in an algebraic framework where algorithms can be designed
and shown to produce irreducible fractions. Similar ideas have also
been applied to 
semi-classical orthogonal polynomials~\cite{LewanowiczWozny2004},
discrete orthogonal polynomials~\cite{AreaGodoyRonveauxZarzo1998},
hypergeometric polynomials~\cite{RebillardZakrajsek2007,RonveauxZarzoAreaGodoy2002},
$q$-analogues~\cite{Lewanowicz2003,LewanowiczGodoyAreaRonveauxZarzo2000}. 
We expect that the tools developed in this article also apply in these
cases.

\subsection{Minimality issues}
Our algorithm computes an irreducible fraction, which is therefore
minimal in some sense. However, this does not imply that its
numerator, which is the recurrence of interest, is also of minimal
order. This appears clearly for the Chebyshev coefficients of
$\arccos$ in \zcref{ex:arccos2}: the numerator is a recurrence of
order~4, whereas it is clear from the closed form of the coefficients
that they satisfy a right factor of that recurrence of order only~2. A
more extreme instance of that situation is provided by the product of
hypergeometric functions in \zcref{eq:gfJac}. Among the many questions
that deserve further study is how to capture minimal generators that
would
be associated not only to the given linear differential operator~$L$,
but also to all its left multiples (This would generalize Lewanowicz'
idea mentioned with the linearization coefficients of the Laguerre
polynomials in \zcref{ssub:linearization_coefficients}.) Another
question is how to use more
information than the linear differential equation, like the local
or asymptotic behavior of the solution of interest, so as to deal with
situations like that in \zcref{ex:RZ}. 
We plan to come back to these issues in future work.

\bibliographystyle{abbrv}
\bibliography{./biblio}

\end{document}